\documentclass[12 pt,reqno]{amsart}
\usepackage{pdfpages}
\usepackage{amsmath,amsthm,amssymb,amscd}
\usepackage{setspace}
\usepackage{upgreek}

\usepackage{cite}
\usepackage{url}
\usepackage{mathtools}
\usepackage{fullpage}
\usepackage{float}
\usepackage{comment}

\usepackage[bookmarksopen,bookmarksdepth=3]{hyperref}

\input xy
\xyoption{all}

\usepackage{microtype}

\allowdisplaybreaks[1]

\newtheorem{thm}{Theorem}
\newtheorem{prop}{Proposition}[section]
\newtheorem{lm}[prop]{Lemma}

\newtheorem{cl}[prop]{Proposition}

\theoremstyle{definition}
\newtheorem{dfn}[prop]{Definition}

\theoremstyle{remark}
\newtheorem{rem}[prop]{Remark}
\newtheorem{ex}[prop]{Example}

\DeclareMathOperator{\Id}{Id}

\DeclareMathOperator{\rdim}{rdim}

\newcommand{\lrarr}{\longrightarrow}

\newcommand{\R}{\mathbb{R}}

\newcommand{\Z}{\mathbb{Z}}

\newcommand{\M}{\mathcal{M}}

\renewcommand{\P}{\mathbb{C}P}
\renewcommand{\L}{\Lambda}

\newcommand{\mI}{\mathcal{I}}

\newcommand{\Hh}{\widehat{H}}

\newcommand{\qkl}{\mathfrak{q}_{k,l}}
\newcommand{\q}{\mathfrak{q}}

\newcommand{\m}{\mathfrak{m}}
\renewcommand{\d}{\partial}

\newcommand{\mR}{\mathfrak{R}}

\newcommand{\I}{[0,1]}
\newcommand{\at}{\tilde{\alpha}}

\newcommand{\xit}{\tilde{\xi}}
\newcommand{\etat}{\tilde{\eta}}

\newcommand{\Mt}{\widetilde{\M}}

\newcommand{\gt}{\tilde\gamma}
\newcommand{\mt}{\tilde\m}
\newcommand{\qt}{\tilde\q}

\newcommand{\mC}{\mathfrak{C}}
\newcommand{\mD}{\mathfrak{D}}
\newcommand{\evbt}{\widetilde{evb}}
\newcommand{\evit}{\widetilde{evi}}
\newcommand{\evt}{\widetilde{ev}}

\newcommand{\mg}{\m^{\gamma}}
\newcommand{\mgp}{\m^{\gamma'}}
\newcommand{\mgt}{\mt^{\gt}}

\renewcommand{\ll}{\langle\!\langle}
\renewcommand{\gg}{\rangle\!\rangle}

\newcommand{\RP}{\mathbb{R}P}

\newcommand{\sly}{\Pi}
\newcommand{\pr}{\varpi}
\newcommand{\A}{\overline{\!A}\mbox{}}

\newcommand{\e}{{\bf{e}}}
\newcommand{\mbar}{\bar{\m}}
\newcommand{\lp}{{\prec}}
\newcommand{\rp}{{\succ}}
\newcommand{\cC}{\mathcal{C}}
\newcommand{\lpt}{{\preccurlyeq}}
\newcommand{\rpt}{{\succcurlyeq}}
\renewcommand{\a}{\alpha}

\newcommand{\s}{\mathfrak{s}}

\newcommand{\uu}{\mathbf{u}}

\newcommand{\Ac}{A_c}

\title{Differential forms, Fukaya $A_\infty$ algebras, and Gromov-Witten axioms}

\date{March 2023}

\author[J. Solomon]{Jake P. Solomon}
\address{Institute of Mathematics\\ Hebrew University, Givat Ram\\Jerusalem, 91904, Israel } \email{jake@math.huji.ac.il}
\author[S. Tukachinsky]{Sara B. Tukachinsky}
\address{School of Mathematical Sciences\\ Tel Aviv University\\Tel Aviv, 6997801, Israel }\email{sarabt1@gmail.com}

\begin{document}

\begin{abstract}
Consider the differential forms $A^*(L)$ on a Lagrangian submanifold $L \subset X$. Following ideas of Fukaya-Oh-Ohta-Ono, we construct a family of cyclic unital curved $A_\infty$ structures on $A^*(L),$ parameterized by the cohomology of $X$ relative to $L.$ The family of $A_\infty$ structures satisfies properties analogous to the axioms of Gromov-Witten theory. Our construction is canonical up to $A_\infty$ pseudoisotopy. We work in the situation that moduli spaces are regular and boundary evaluation maps are submersions, and thus we do not use the theory of the virtual fundamental class.
\end{abstract}

\maketitle

\tableofcontents

\section{Introduction}

In the beautiful series of papers~\cite{Fukaya,Fukaya2,FOOO1,FOOOtoricI,FOOOtoricII}, Fukaya and Fukaya-Oh-Ohta-Ono reworked and extended in the language of differential forms the theory of $A_\infty$ algebras associated to Lagrangian submanifolds from their book~\cite{FOOO}. With the help of this new tool, they obtained many striking results in Floer theory and mirror symmetry. They work in a very general setting, and introduce fundamental new ideas in the theory of the virtual fundamental class to address the technical difficulties that arise.

The present paper uses differential forms to construct a family of cyclic unital curved $A_\infty$ algebras associated to a Lagrangian submanifold. We consider Lagrangian submanifolds that satisfy an analog of the convex condition in algebraic geometry~\cite{FultonPandharipande}, so the construction can be made without using virtual fundamental class techniques.

Our family of $A_\infty$ algebras is parameterized by the cohomology of $X$ relative to $L$, as opposed to absolute cohomology of $X$ as found in the literature.
The family satisfies differential equations analogous to the fundamental class and divisor axioms of Gromov-Witten theory.
Our definition of unitality is stronger than the standard one. The use of relative cohomology is of crucial importance for proving unitality and the divisor equation.

We use the framework developed here in~\cite{ST2,ST3} to define open Gromov-Witten invariants and establish their properties. For this purpose, we also include a discussion of the operator $\m_{-1}$ as defined in~\cite{Fukaya2}.

\subsection{Setting}\label{ssec:setting}

Consider a symplectic manifold $(X,\omega)$ with $\dim_{\R}X=2n$, and a connected Lagrangian submanifold $L$ with relative spin structure~$\s =\s_L.$ For the definition of relative spin structure, see~\cite[Definition 8.1.2]{FOOO} and \cite[Definition 3.1.2(c)]{WehrheimWoodward}.
Let $J$ be an $\omega$-tame~\cite[p.2]{MS} almost complex structure on $X$. Denote by $\mu:H_2(X,L) \to \Z$ the Maslov index as in~\cite[Section~2]{CieliebakGoldstein}. See also~\cite[Appendix]{Banyaga} and references therein.
Let $\sly$ be a quotient of $H_2(X,L;\Z)$ by a possibly trivial subgroup contained in the kernel of the homomorphism $\omega \oplus \mu : H_2(X,L;\Z) \to \R \oplus \Z.$ Thus the homomorphisms $\omega,\mu,$ descend to $\sly.$ Denote by $\beta_0$ the zero element of $\sly.$ We use a Novikov ring $\L$ which is a completion of a subring of the group ring of $\sly$. The precise definition follows. Denote by $T^\beta$ the element of the group ring corresponding to $\beta \in \sly$, so $T^{\beta_1}T^{\beta_2} = T^{\beta_1 + \beta_2}.$ Then,
\[
\L=\left\{\sum_{i=0}^\infty a_iT^{\beta_i}\,\bigg|\;a_i\in\R, \beta_i\in \sly, \omega(\beta_i)\ge 0,\; \lim_{i\to \infty}\omega(\beta_i)=\infty\right\}.
\]
A grading is defined on $\L$ by declaring $T^\beta$ to be of degree $\mu(\beta).$

For $k\ge  -1,$ denote by $\M_{k+1,l}(\beta)$ the moduli space of genus zero $J$-holomorphic open stable maps to $(X,L)$ of degree $\beta \in \sly$ with one boundary component, $k+1$ boundary marked points, and $l$ interior marked points. The boundary points are labeled according to their cyclic order. Denote by
$evb_i^\beta:\M_{k+1,l}(\beta)\to L,$ and
$evi_j^\beta:\M_{k+1,l}(\beta)\to X,$ the boundary and interior evaluation maps respectively, where $ i=0,\ldots,k,$ and $j=1,\ldots, l$.
Assume that $\M_{k+1,l}(\beta)$ is a smooth orbifold with corners. Then it carries a natural orientation induced by the relative spin structure on $(X,L)$, as in~\cite[Chapter 8]{FOOO}. Assume in addition that $evb_0^\beta$ is a proper submersion. See Example~\ref{rem:assumptions} and Remark~\ref{rem:assumptionsgeneral} for a discussion and examples of when these assumptions hold. See Section~\ref{ssec:moc} for background on orbifolds with corners and Section~\ref{ssec:osm} for background on open stable maps.

For any manifold $M$, possibly with corners, denote by $A^*(M)$ the algebra of smooth differential forms on $M$ with coefficients in $\R$.
For $m>0$, denote by
$A^m(X,L)$ the smooth differential $m$-forms on $X$ that pull back to zero on $L$, and denote by $A^0(X,L)$ the functions on $X$ that are constant on $L$. The exterior derivative $d$ makes $A^*(X,L)$ into a complex.

Let $t_0,\ldots,t_N,$ be formal variables with degrees in $\Z$.
Define graded-commutative rings
\[
R:=\L[[t_0,\ldots,t_N]],\quad Q:=\R[t_0,\ldots,t_N] ,
\]
thought of as differential graded algebras with trivial differential.
Set
\[
C:= A^*(L)\otimes R,\quad\text{and}\quad D:= A^*(X,L)\otimes Q ,
\]
where $\otimes$ is understood as the completed tensor product of differential graded algebras.
Write $\Hh^*(X,L;Q)=H^*(D).$ The gradings on $C,D,$ and $\Hh^*(X,L;Q)$, take into account the degrees of $t_j,T^\beta,$ and the degree of differential forms.

Define a valuation
\[
\nu:R\lrarr \R,
\]
by
\[
\nu\left(\sum_{j=0}^\infty a_jT^{\beta_j}\prod_{i=0}^Nt_i^{l_{ij}}\right)
= \inf_{\substack{j\\a_j\ne 0}} \left(\omega(\beta_j)+\sum_{i=0}^N l_{ij}\right).
\]
The valuation $\nu$ induces a valuation on $Q, C, D,$ and their tensor products, which we also denote by $\nu$.
Define $\mI_R: = \{\alpha\in R\,|\,\nu(\alpha)>0\},$ and similarly $\mI_Q: = \{\alpha\in Q\,|\,\nu(\alpha)>0\}$. Let $\overline{R}: = R/\mI_R = \R$ and
\[
\overline{C} : = C/(\mI_R C) = A^*(L).
\]

\subsection{Statement of results}\label{ssec:sor}

Let $\mathcal R$ be a differential graded algebra over $\R$ with valuation $\varsigma_{\mathcal R}$ and let $\cC$ be a graded module over $\mathcal R$ with valuation $\varsigma_{\cC}.$ We implicitly assume elements of graded rings and modules are of homogeneous degree and denote the degree by $| \cdot |.$ Let $\delta_{i,j}$ denote the Kronecker delta.

\begin{dfn}\label{dfn:cycunit}
An $n$-dimensional (curved) \textbf{cyclic unital $A_\infty$ structure} on $\cC$ is a triple $(\{\m_k\}_{k\ge 0},\lp\;,\,\rp,\e)$ of
maps $\m_k:\cC^{\otimes k}\to \cC[2-k]$, a pairing $\lp\;,\,\rp:\cC\otimes \cC\to \mathcal{R}[-n]$, and an element $\e\in \cC$ with $|\e|=0$, satisfying the following properties. We denote by $\alpha,$ possibly with subscripts, an element of $\cC,$ and by $a$ an element of $\mathcal{R}.$
\begin{enumerate}
		\item\label{it:lin}
		The operations $\m_k$ are $\mathcal{R}$-multilinear in the sense that
		\[
		\m_k(\alpha_1,\ldots,\alpha_{i-1},a\cdot\alpha_i,\ldots,\alpha_k)=
		(-1)^{|a|\cdot\big(i+\sum_{j=1}^{i-1}|\a_j|\big)}
		a\cdot\m_k(\alpha_1,\ldots,\alpha_k)+\delta_{1,k}\cdot da\cdot\alpha_1.
		\]
		\item\label{it:plin}
		The pairing $\lp\;,\,\rp$ is $\mathcal{R}$-bilinear in the sense that
		\[
		\lp a\cdot\a_1,\a_2\rp=a\lp\a_1,\a_2\rp,\quad
		\lp\a_1,a\cdot\a_2\rp=(-1)^{|a|\cdot (1+|\a_1|)}a\lp\a_1,\a_2\rp.
		\]
    \item\label{it:a_infty}
    The $A_\infty$ relations hold:
    \[
    \sum_{\substack{k_1+k_2=k+1\\1\le i\le k_1}}(-1)^{\sum_{j=1}^{i-1}(|\alpha_j|+1)}
\m_{k_1}(\alpha_1,\ldots,\alpha_{i-1},\m_{k_2}(\alpha_i,\ldots,\alpha_{i+k_2-1}), \alpha_{i+k_2},\ldots,\alpha_k)=0.
    \]
    \item\label{it:val}
    $\varsigma_{\cC}(\m_k(\alpha_1,\ldots,\a_k))\ge \sum_{j=1}^k\varsigma_{\cC}(\alpha_j)$ and $\varsigma_{\cC}(\m_0) > 0.$
    \item\label{it:val2}
    $\varsigma_{\mathcal R}(\lp\alpha_1,\alpha_2\rp)\ge\varsigma_{\cC}(\alpha_1) +\varsigma_{\cC}(\alpha_2).$
    \item\label{it:symm}
    $\lp \alpha_1,\alpha_2\rp=(-1)^{(|\alpha_1|+1)(|\alpha_2|+1)+1} \lp\alpha_2,\alpha_1\rp$.
	\item\label{it:cyclic}
     The pairing is cyclic:
	\begin{multline*}
    \qquad\lp \m_k(\alpha_1,\ldots,\alpha_k),\alpha_{k+1}\rp=\\
	=(-1)^{(|\alpha_{k+1}|+1)\sum_{j=1}^k(|\alpha_j|+1)}
	\lp\m_k(\alpha_{k+1},\alpha_1,\ldots,\alpha_{k-1}),\alpha_k\rp+
	\delta_{1,k}\cdot d\lp\alpha_1,\alpha_2\rp.
    \end{multline*}
	\item\label{it:unit1}
	$\m_k(\alpha_1,\ldots,\alpha_{i-1},\e,\alpha_{i+1},\ldots,\alpha_k)=0 \quad \forall k\ne 0,2.$
    \item\label{it:unit2}
    $\lp \m_0,\e\rp=0$.
	\item\label{it:unit3}
	$\m_2(\e,\alpha)=\alpha=(-1)^{|\alpha|}\m_2(\alpha,\e).$
\end{enumerate}
\end{dfn}
\begin{rem}
The intuition behind the signs of properties~\eqref{it:lin},~\eqref{it:plin},~\eqref{it:a_infty}, and~\eqref{it:cyclic}, is that we consider the shifted degree of elements of $\cC$, and the shifted degree of the operators $\m_k$, which is $1$. Thus, ``passing'' $a\in\mathcal{R}$ ``through'' $\alpha\in \cC$ contributes $(-1)^{|a|\cdot(|\alpha|+1)}$,``passing'' $a$ through $\m_k$ adds $(-1)^{|a|}$, and ``passing'' $\m_k$ through $\a$ adds $(-1)^{|\a|+1}$.
The sign in property~\eqref{it:symm} reflects the fact that the pairing is graded anti-symmetric.
\end{rem}
\begin{rem}
Our definition differs from that of~\cite{FOOOtoricI, FOOOtoricII, Fukaya} in that $\m_0$ is required to respect the unit $\e$.
\end{rem}

Equip $R$ with the trivial differential $d_R=0$. Consider the $R$-module $C$.
For $\gamma\in \mI_QD$ with $d\gamma=0$, $|\gamma|=2,$
and $\beta\in\sly$,
define maps
\[
\m^{\gamma,\beta}_k:C^{\otimes k}\lrarr C
\]
by
\[
\m_1^{\gamma,\beta_0}(\alpha)=d\alpha,
\]
and for $k \geq 0$ when $(k,\beta)\ne(1,\beta_0),$ by
\begin{align*}
\m^{\gamma,\beta}_k(\alpha_1,\ldots,\alpha_k):=& (-1)^{\sum_{j=1}^kj(|\alpha_j|+1)+1}
\sum_{l\ge0}\frac{1}{l!}{evb_0^\beta}_* (\bigwedge_{j=1}^k (evb_j^\beta)^*\alpha_j\wedge\bigwedge_{j=1}^l(evi_j^\beta)^*\gamma).
\end{align*}
Define also
\[
\mg_k:C^{\otimes k}\lrarr C
\]
by
\[
\mg_k:=\sum_{\beta\in\sly}T^{\beta}\m^{\gamma,\beta}_k.
\]
Denote by $\langle\;,\;\rangle$ the signed Poincar\'e pairing,
\begin{equation}\label{eq:pairing}
\langle\xi,\eta\rangle:=(-1)^{|\eta|}\int_L\xi\wedge\eta.
\end{equation}
Denote by $1$ the constant function $1\in A^0(L)$. The main results of the paper are the following theorems.

\begin{thm}\label{thm:str}
The triple $(\{\mg_k\}_{k\ge 0},\langle\,,\,\rangle,1)$ is a cyclic unital $A_\infty$ structure on $C$.
\end{thm}

Set
\begin{gather}\label{eq:mC}
\mathfrak{R}:=A^*([0,1];R),\notag \\
\mC:=A^*([0,1]\times L;R),\mbox{ and }\mD:=A^*([0,1]\times X,[0,1]\times L;Q).
\end{gather}
The valuation $\nu$ induces valuations on $\mathfrak{R},\mC,$ and $\mD$, which we still denote by $\nu$.
For $t\in[0,1]$ and letting $M$ be either $L$ or the point, denote by
\[
j_t:M\to [0,1]\times M
\]
the inclusion $j_t(p)=(t,p)$.
\begin{dfn}
Let $S_1=(\m,\lp\;,\,\rp,\e)$ and $S_2=(\m',\lp\;,\,\rp',\e')$ be cyclic unital $A_\infty$ structures on $C$.
A cyclic unital \textbf{pseudoisotopy} from $S_1$ to $S_2$ is a cyclic unital $A_\infty$ structure $(\mt,\lpt\;,\rpt,\tilde{\e})$ on the $\mathfrak{R}$-module $\mC$ such that
 for all $\at_j\in \mC$ and all $k\ge 0$,
\begin{gather*}
j_0^*\tilde{\m}_k(\at_1,\ldots,\at_k)=\m_k(j_0^*\at_1,\ldots,j_0^*\at_k),\\
j_1^*\tilde{\m}_k(\at_1,\ldots,\at_k)=\m'_k(j_1^*\at_1,\ldots,j_1^*\at_k),
\end{gather*}
and
\begin{gather*}
j_0^*\lpt\at_1,\at_2\rpt=\lp j_0^*\at_1,j_0^*\at_2\rp,\qquad j_0^*\tilde{\e}=\e,\\
j_1^*\lpt\at_1,\at_2\rpt=\lp j_1^*\at_1,j_1^*\at_2\rp',\qquad j_1^*\tilde{\e}=\e'.
\end{gather*}
\end{dfn}

\begin{thm}\label{thm:isot}
Let $\gamma,\gamma'\in \mI_QD$ be closed with $|\gamma|=|\gamma'|=2$.
If $[\gamma]=[\gamma']\in \Hh^*(X,L;Q),$
then there exists a cyclic unital pseudoisotopy from $(\mg,\langle\;,\,\rangle,1)$ to $(\m^{\gamma'},\langle\;,\,\rangle,1)$.
\end{thm}
In Section~\ref{pseudoisot} we also
discuss pseudoisotopies arising from varying $J,$ under regularity assumptions on the family moduli spaces similar to those already assumed for $\M_{k+1,l}(\beta).$

By property~\eqref{it:val}, the maps $\m_k$ descend to maps on the quotient
\[
\mbar_k:\overline{C}^{\otimes k}\lrarr \overline{C}.
\]
\begin{thm}\label{thm:prop}
Suppose $\d_{t_0}\gamma=1\in A^0(X,L)\otimes Q$ and $\d_{t_1}\gamma=\gamma_1\in A^2(X,L)\otimes Q$.
Assume the map $H_2(X,L;\Z)\to Q$ given by $\beta\mapsto\int_\beta\gamma_1$ descends to $\sly$.
Then the operations $\mg_k$ satisfy the following properties.
\begin{enumerate}
	\item\label{prop1} (Fundamental class)
		$\d_{t_0}\mg_k=-1\cdot\delta_{0,k}$.
	\item\label{prop2} (Divisor)
		$\d_{t_1}\m^{\gamma,\beta}_k=\int_\beta\gamma_1 \cdot\m^{\gamma,\beta}_k.$
	\item\label{prop3} (Energy zero)
		The operations $\mg_k$ are deformations of the usual differential graded algebra structure on differential forms. That is,
		\[
		\bar{\m}^\gamma_1(\alpha)= d\alpha,\qquad\bar{\m}^\gamma_2(\alpha_1,\alpha_2)=(-1)^{|\alpha_1|}\alpha_1\wedge\alpha_2,\qquad \bar{\m}^\gamma_k=0, \quad k \neq 1,2.
		\]
\end{enumerate}
\end{thm}

In Section~\ref{ssec:form} we also construct a distinguished element $\mg_{-1} \in R$ following~\cite{Fukaya2}. In the subsequent sections, we prove its properties along with the properties of $\mg_k$ for $k \geq 0.$ In Section~\ref{pseudoisot} we construct $\mgt_{-1}$, the analogous structure for a pseudoisotopy. In Section~\ref{ssec:uniform} we reformulate the $A_\infty$ structure equations of the pseudoisotopy so that the structure equation for $\mgt_{-1}$ fits more naturally. The reformulated $A_\infty$ structure equations are used in~\cite{ST2} to prove that the superpotential is invariant under pseudoisotopy.

\begin{ex}\label{rem:assumptions}
Suppose $J$ is integrable, and we are given a Lie group $G_X$ with a transitive action $\alpha : G_X \times X \to X$ such that for each $g \in G$ the diffeomorphism $\alpha(g,\cdot)$ is $J$-holomorphic. Moreover, suppose $c_X: X \to X$ is an anti-holomorphic involution with $L = Fix(c_X)$ and $c_G : G \to G$ is an involutive homomorphism such that $\alpha(c_G g,c_X x) = c_X \alpha(g,x).$ Then it is shown in~\cite{Zernik2} that our assumptions that $\M_{k,l}(\beta)$ is a smooth orbifold with corners and $evb_0^\beta$ is a submersion are satisfied.
Indeed, it is well-known that the moduli space of closed genus zero stable maps to $X$ is a complex orbifold in the presence of a transitive group action~\cite{FultonPandharipande,RobbinRuanSalamon}.
The moduli space of open stable maps $\M_{k,l}(\beta)$ is constructed from the fixed points of the induced anti-holomorphic involution of the moduli space of closed stable maps by cutting along the compactification divisor. The subgroup $G_L = Fix(c_G) \subset G_X$ acts on $\M_{k,l}(\beta)$ and acts transitively on $L$, and $evb_0$ is $G_L$ equivariant. It follows that $evb_0$ is a submersion.

Thus, examples of $(X,L)$ which satisfy our assumptions include $(\P^n,\RP^n)$ with the standard complex
and symplectic structures or, more generally, flag varieties, Grassmannians, and products thereof.
\end{ex}

\begin{rem}\label{rem:assumptionsgeneral}
More generally, suppose $J$ is integrable, there exists a Lie group $G_X$ that acts transitively on $X$ by $J$-holomorphic diffeomorphisms, and there exists a Lie subgroup $G_L \subset G_X$ that preserves $L$ and acts transitively on $L.$ We outline an argument showing that $\M_{k,l}(\beta)$ is a smooth orbifold with corners and $evb_0^\beta$ is a submersion.
Indeed, \cite[Proposition 7.4.3]{MS} shows that all $J$-holomorphic genus zero stable maps to $X$ without boundary are regular. A modification of the argument there shows that $J$-holomorphic genus zero stable maps to $(X,L)$ with one boundary component are also regular. See~\cite[Lemma 3.2]{EvansLekili} for the special case where the domain of the map is a disk under weaker assumptions. For regularity of holomorphic disks, instead of Grothendieck's classification~\cite{Grothendieck}, one uses Oh's work on the Riemann-Hilbert problem~\cite{Oh}. The argument applies equally well to maps that are not somewhere injective in the sense
of~\cite[Section 2.5]{MS}. So, the fact that a $J$-holomorphic map from a domain with boundary need not factor through a somewhere injective map~\cite{KwonOh,Lazzarini} does not affect the argument. Once all stable maps are regular, it should be possible to modify the techniques of~\cite{RobbinRuanSalamon} to include Lagrangian boundary conditions, and hence conclude that the moduli space is a smooth orbifold with corners. Since $G_L$ acts transitively on~$L,$ it follows that $evb_0$ is a submersion.

Virtual fundamental class techniques should allow the extension of our results to general target manifolds.
\end{rem}

\subsection{Outline}
In Section~\ref{ssec:orinetation} we review orientation conventions and properties of the push-forward of differential forms. Sections~\ref{ssec:form}-\ref{ssec:q-1} formulate and prove the $A_\infty$ structure relations for the closed-open maps $\q_{k,l}$ for $k \geq -1$. In Section~\ref{ssec:properties} we formulate and prove additional properties of the $\q$ operators. The section closes with the proofs of Theorems~\ref{thm:str} and~\ref{thm:prop}. Section~\ref{pseudoisot} constructs pseudo-isotopies and uses them to prove Theorem~\ref{thm:isot}.
Section~\ref{ssec:uniform} reformulates the $A_\infty$ structure relations in a way that incorporates $\m_{-1}$ more naturally.

\subsection{Acknowledgments}
The authors would like to thank D. Auroux and an anonymous referee for many helpful comments, D. McDuff, K. Wehrheim, and A. Zernik, for helpful conversations, and X. Chen for a sign correction.
The authors were partially supported by ERC starting grant 337560 and ISF Grant 1747/13. The first author was partially supported by ISF Grant 569/18. The second author was partially supported by the Canada Research Chairs Program and by NSF grant No. DMS-163852.

\subsection{Notation}\label{ssec:notation}\leavevmode

We write $I:=[0,1]$ for the closed unit interval.

We denote by $pt$ the map from any space to the point or, depending on the context, the point itself.

Use $i$ to denote the inclusion $i:L\hookrightarrow X$. By abuse of notation, we also use $i$ for $\Id\times i:I\times L\to I\times X$. The meaning in each case should be clear from the context.

Whenever a tensor product is written, we mean the completed tensor product. For example, $A^*(L)\otimes R$ is the completion of the tensor product $A^*(L)\otimes_{\R}R$ with respect to $\nu$. The tensor product of differential graded algebras is again a differential graded algebra in the standard way.
In particular,
\begin{gather*}
d(t_j\cdot \a)=(-1)^{|t_j|}t_j\cdot d\a,\quad
\a\cdot t_j=(-1)^{|t_j|\cdot|\a|}t_j\cdot\a,
\qquad
\forall \a\in \Upsilon,
\quad
\Upsilon=A^*(L), A^*(X,L).
\end{gather*}

Write $A^*(L;R)$ for $A^*(L)\otimes R$. Similarly, $A^*(X;Q)$ and $A^*(X,L;Q)$ stand for $A^*(X)\otimes Q$ and $A^*(X,L)\otimes Q$ respectively.

For $f:M\to N$ a smooth map between orbifolds with corners, define the relative dimension by
\[
\rdim f: = \dim M - \dim N.
\]
In particular, if $f$ is a submersion, then $\rdim f$ is the dimension of the fiber of $f$.

For two lists $B_1=(v_1,\ldots,v_n),$ $B_2=(w_1,\ldots,w_m),$ denote by $B_1\circ B_2$ the concatenation $(v_1,\ldots,v_n,w_1,\ldots,w_m)$.

\section{Structure}\label{ssec:constr}

\subsection{Orientations and integration}\label{ssec:orinetation}

\subsubsection{Orbifolds with corners}\label{ssec:moc}

We use the definition of orbifolds with corners from~\cite{ST4}. We also use the definitions of smooth maps, strongly smooth maps, boundary and fiber products of orbifolds with corners given there. In particular, for $M$ an orbifold with corners, $\d M$ is again an orbifold with corners and comes with a natural map $i_M : \d M \to M.$ In the special case of manifolds with corners, our definition of boundary coincides with~\cite[Definition~2.6]{Joyce2}, our smooth maps coincide with weakly smooth maps in~\cite[Definition~2.1(a)]{Joyce3}, and our strongly smooth maps are as in~\cite[Definition~2.1(e)]{Joyce3}, which coincides with smooth maps in~\cite[Definition 3.1]{Joyce2}. We say a map of orbifolds is a submersion if it is a strongly smooth submersion in the sense of~\cite{ST4}. In the special case of manifolds with corners, our submersions coincide with submersions in~\cite[Definition~3.2(iv)]{Joyce2} and with strongly smooth horizontal submersions in~\cite[Definition 19(a)]{Zernik2}.
For a strongly smooth map $f : M \to N,$ we use the notion of the vertical boundary $\partial^v_fM \subset \partial M$ defined in~\cite[Section 2.1.1]{ST4}, which extends to orbifolds with corners the definition of~\cite[Section 4]{Joyce2} for manifolds with corners. We write $i_f^v : \partial^v_fM \to M$ for the restriction of $i_f$ to $\partial^v_f M.$ When $f$ is a submersion, the vertical boundary is the fiberwise boundary, that is, $\partial^v_fM = \coprod_{y \in N} \partial (f^{-1}(y)).$ If $\partial N = \emptyset,$ then $\partial^v_f M = \partial M.$
A strongly smooth map of orbifolds $f : M \to N$ induces a strongly smooth map $f|_{\d^v_f M} = f \circ i^v_f : \d^v_f M \to N,$ called the restriction to the vertical boundary.  If $f$ is a submersion, then the restriction $f|_{\d^v_f M}$ is also a submersion.
As usual, diffeomorphisms are smooth maps with a smooth inverse.
We use the notion of transversality from~\cite[Section 3]{ST4}, which is induced from transversality of maps of manifolds with corners as defined in~\cite[Definition~6.1]{Joyce2}. In particular, any smooth map is tranverse to a submersion. Weak fiber products of strongly smooth transverse maps exist by~\cite[Lemma 5.3]{ST4}. Below, we omit the adjective `weak' for brevity.

\subsubsection{Orientation conventions}
For a diffeomorphism $f : M \to N$ of oriented orbifolds with corners, we define $sgn(f)$ to be $1$ if $f$ preserves orientation and $-1$ if it reverses orientation. We use similar notation for isomorphisms of oriented vector spaces.
We use the definition for orientations of orbifolds with corners given in~\cite[Section 3]{ST4}. We use the conventions of Sections 2.2, 3, and~5.1 of~\cite{ST4} for orienting boundary and fiber products of orbifolds with corners.  For a submersion of orbifolds with corners $h: Q \to S$ and $y \in S,$ we orient the fiber $h^{-1}(y)$ by identifying it with the fiber product,
\begin{equation}\label{eq:fiberid}
h^{-1}(y) = \{y\} \times_S Q.
\end{equation}
For manifolds,
our convention for orienting boundaries agrees with~\cite[Convention~7.2(a)]{Joyce2} and our convention for orienting fiber products agrees with~\cite[Convention~7.2(b)]{Joyce2}. For submersions of manifolds, our convention for orienting fiber products agrees also with~\cite{FOOO}.

\subsubsection{Integration properties}
For a detailed discussion of differential forms on orbifolds with corners we refer to~\cite{ST4}.
Let $f: M\to N$ be a proper submersion of oriented orbifolds with corners of relative dimension $\rdim f =r,$ and let $\Upsilon$ be a graded-commutative algebra over~$\R$. Denote by $f_* : A^*(M;\Upsilon) \to A^*(N;\Upsilon)[-r]$ the push-forward of forms along $f$ as defined in~\cite[Section 4.1]{ST4}, that is, integration over the fiber. We will need the following properties of $f_*$ from~\cite[Theorem 1]{ST4}, where they were formulated for $\Upsilon=\R$. Property~\eqref{prop:pushpull} below allows the reduction of integrals with coefficients in general $\Upsilon$ to integrals with coefficients in $\R$. In the following, all orbifolds are oriented.

\begin{prop}\label{prop:proppp}\leavevmode
\begin{enumerate}
	\item \label{normalization}
	Let $M$ be compact and let $f:M\to pt.$ Let $\alpha\in A^m(M)\otimes\Upsilon$. Then
	\[
	f_*\alpha=\begin{cases}
	\int_M\alpha,& m=\dim M,\\
	0,&\text{otherwise}.
	\end{cases}
	\]
	\item\label{prop:pushcomp}
		Let $g: P\to M$, $f:M\to N,$ be proper submersions. Then
		\[
		f_*\circ g_*=(f\circ g)_*.
		\]
	\item\label{prop:pushpull}
		Let $f:M\to N$ be a proper submersion, $\alpha\in A^*(N;\Upsilon),$ $\beta\in A^*(M;\Upsilon)$. Then
		\[
		f_*(f^*\alpha\wedge\beta)=\alpha\wedge f_*\beta.
		\]
	\item\label{prop:pushfiberprod}
		Let
		\[
		\xymatrix{
		{M\times_N P}\ar[r]^{\quad p}\ar[d]^{q}&
        {P}\ar[d]^{g}\\
        {M}\ar[r]^{f}&N
		}
		\]
		be a pull-back diagram of smooth maps, where $g$ is a proper submersion. Let $\alpha\in A^*(P).$ Then
		\[
		q_*p^*\alpha=f^*g_*\alpha.
		\]
\end{enumerate}
\end{prop}

Furthermore, we have the following generalization of Stokes' theorem, also from~\cite[Theorem~1]{ST4}. It uses the notion of vertical boundary from Section~\ref{ssec:moc}.
\begin{prop}[Stokes' theorem]\label{stokes}
Let $f:M\to N$ be a proper submersion with $\dim M=s$, and let $\xi\in A^t(M;\Upsilon)$. Then
\[
d (f_*\xi)=f_*(d \xi)+(-1)^{s+t}\big(f\big|_{\partial^v_f M}\big)_*\xi.
\]
\end{prop}

\begin{rem}
Proposition~\ref{stokes} applied to $f:M\to pt$ yields the classical Stokes' theorem up to a sign,
\begin{equation}\label{eq:classstokes}
\int_M d\xi = (-1)^{\dim M+|\xi|+1}\int_{\d M} \xi.
\end{equation}
The sign arises from the possibly non-trivial grading of the coefficient ring.
To derive this sign, assume without loss of generality that $\xi = t \eta,$ where $t \in \Upsilon$ and $\eta \in A^*(M)$. Then, the classical Stokes' theorem gives
\[
\int_M d\xi = \int_M d(t\eta) = (-1)^{|t|} t\int_M d\eta = (-1)^{|t|}t \int_{\partial M} \eta = (-1)^{|t|} \int_{\partial M} \xi.
\]
On the other hand, the integrals vanish unless $|\eta| = \dim M - 1.$ In this case,
\[
|t| \equiv |\xi| + |\eta| \equiv \dim M +|\xi| + 1 \pmod 2.
\]
\end{rem}

\subsubsection{Currents}\label{sssec:currents}
For a detailed discussion of currents on orbifolds with corners we refer to~\cite{ST4}. We recall below some key notation.
Let $M$ be an oriented orbifold with corners. Denote by $\A^k(M)$ the space of currents of cohomological degree $k$, that is, the dual space of compactly supported differential forms $\Ac^{\dim M - k}(M)$. Differential forms are identified as a subspace of currents by
\begin{gather*}
\varphi:A^k(M)\hookrightarrow \A^k(M),\\
\varphi(\eta)(\alpha)=\int_{M}\eta\wedge\alpha,\quad \alpha\in \Ac^{\dim M -k}(M).
\end{gather*}
Accordingly, for a general current $\zeta,$ we may use the notation
\begin{equation}\label{eq:diw}
\zeta(\alpha)=\int_{M} \zeta\wedge\alpha,\quad  \a\in \Ac^*(M).
\end{equation}
Define
\[
d : \A^k(M) \to \A^{k+1}(M)
\]
by $d\zeta(\alpha)=(-1)^{1+|\zeta|}\zeta(d\alpha).$ Thus, if $M$ is closed, we have $d\varphi(\eta) = \varphi(d\eta).$
Currents are a bimodule over differential forms with
\[
(\eta\wedge \zeta)(\gamma)
:= (-1)^{|\eta|\cdot|\zeta|} \zeta(\eta\wedge\gamma),
\qquad
\gamma\in \Ac^*(M),\quad \eta\in A^*(M),\quad\zeta\in\A^*(M),
\]
and
\[
(\zeta\wedge \eta)(\gamma)
:= \zeta(\eta\wedge\gamma),
\qquad
\gamma\in \Ac^*(M), \quad \eta\in A^*(M),\quad\zeta\in\A^*(M).
\]
This bimodule structure makes $\varphi$ a bimodule homomorphism.

Let $f:M\to N$ be a proper map of orbifolds. Define the push-forward
\begin{equation}\label{eq:pfc}
f_* : \A^k(M) \to \A^{k-\rdim f}(N)
\end{equation}
by
\[
(f_*\zeta)(\xi)=(-1)^{m\cdot \rdim f}\zeta(f^*\xi),\qquad \xi\in A^m(N).
\]
So, when $f$ is a submersion, $f_* \varphi(\eta) = \varphi(f_*\eta).$ Analogs of the integration properties of Propositions~\ref{prop:proppp} and~\ref{stokes} for currents are given in Propositions ~6.1 and~6.5 of~\cite{ST4} respectively.

\subsection{Formulation}\label{ssec:form}
In this section, we construct a family of $A_\infty$ structures following~\cite{FOOO, Fukaya, FOOO1}.

\subsubsection{Open stable maps}\label{ssec:osm}
A $J$-holomorphic genus-$0$ open stable map to $(X,L)$ of degree $\beta \in \sly$ with one boundary component, $k+1$ boundary marked points, and $l$ interior marked points, is a quadruple $(\Sigma, u,\vec{z},\vec{w})$ as follows. The domain $\Sigma$ is a genus-$0$ nodal Riemann surface with boundary consisting of one connected component,
\[
u: (\Sigma,\d\Sigma) \to (X,L)
\]
is a continuous map, $J$-holomorphic on each irreducible component of $\Sigma,$ with
\[
u_*([\Sigma,\partial\Sigma]) = \beta,
\]
and
\[
\vec{z} = (z_0,\ldots,z_k), \qquad \vec{w} = (w_1,\ldots,w_l),
\]
with $z_j \in \partial \Sigma, \, w_j \in int(\Sigma),$ distinct from one another and from the nodal points. The labeling of the marked points $z_j$ respects the cyclic order given by the orientation of $\partial \Sigma$ induced by the complex orientation of $\Sigma.$ Stability means that
if $\Sigma_i$ is an irreducible component of $\Sigma$, then either $u|_{\Sigma_i}$ is nonconstant or it satisfies the following requirement: If $\Sigma_i$ is a sphere, the number of marked points and nodal points on $\Sigma_i$ is at least 3; if $\Sigma_i$ is a disk, the number of marked and nodal boundary points plus twice the number of marked and nodal interior points is at least $3$.
An isomorphism of open stable maps $(\Sigma,u,\vec{z},\vec{w})$ and $(\Sigma',u',\vec{z}',\vec{w}')$ is a homeomorphism $\theta : \Sigma \to \Sigma'$, biholomorphic on each irreducible component, such that
\[
u = u' \circ \theta, \qquad\qquad  z_j' = \theta(z_j), \quad j = 0,\ldots,k, \qquad w_j' = \theta(w_j), \quad j = 1,\ldots,l.
\]
Thus we obtain the category of stable maps, which has stable maps for objects and isomorphisms of stable maps for morphisms. Since all morphisms are isomorphisms, the category of stable maps is a groupoid.

Denote by $\M_{k+1,l}(\beta) = \M_{k+1,l}(\beta;J)$ the moduli space of $J$-holomorphic genus zero open stable maps to $(X,L)$ of degree $\beta$ with one boundary component, $k+1$ boundary marked points, and $l$ internal marked points. In particular, $\M_{k+1,l}(\beta)$ is a topological groupoid that is equivalent to the category of stable maps. Here, by topological groupoid, we mean a small groupoid along with a topology on the sets of objects and morphisms such that the structure maps of the groupoid are continuous.
Denote by
\begin{gather*}
evb_j^\beta:\M_{k+1,l}(\beta)\to L, \qquad  \qquad j=0,\ldots,k, \\
evi_j^\beta:\M_{k+1,l}(\beta) \to X, \qquad \qquad j=1,\ldots,l,
\end{gather*}
the evaluation maps given by $evb_j^\beta((\Sigma,u,\vec{z},\vec{w}))=u(z_j)$ and $evi_j^\beta((\Sigma,u,\vec{z},\vec{w}))= u(w_j).$
We may omit the superscript $\beta$ when the omission does not create ambiguity.

As mentioned above in Section~\ref{ssec:setting}, Example~\ref{rem:assumptions}, and Remark~\ref{rem:assumptionsgeneral}, throughout the paper we assume that $\M_{k+1,l}(\beta)$ is a smooth orbifold with corners and $evb_0^\beta$ is a proper submersion. In particular, the spaces of objects and morphisms of $\M_{k+1,l}(\beta)$ are smooth manifolds with corners and the groupoid structure maps are local diffeomorphisms. Corners of codimension $m$ in $\M_{k+1,l}(\beta)$ consist of open stable maps $(\Sigma, u,\vec{z},\vec{w})$ where $\Sigma$ has $m$ boundary nodes. A precise description of corners of codimension $m = 1$ is given in Proposition~\ref{rem:gluingsign} below. In the special case when $k=-1$ and $\beta$ belongs to the image of the map $H_2(X) \to H_2(X,L) \to \sly,$ an additional boundary component arises from the collapse of the boundary of a disk to a point. Alternatively, one can view this phenomenon as the bubbling of a $J$-holomorphic sphere from a constant disk, which is unstable and thus forgotten. The instability of the constant disk causes such bubbling to occur in codimension $1.$ A precise description of this type of boundary component is given in Proposition~\ref{rem:pintgluing} below.

\subsubsection{Operators}
For any list $a=(a_1,\ldots,a_k)\in \Z_{\ge 0}^{\times k},$ define
\[
\varepsilon(a):=\sum_{j=1}^kj(a_j+1)+1.
\]
To simplify notation in the following, we allow differential forms as input, in lieu of their degrees.
In particular, for a list $\a\in C^{\times k}$,
\[
\varepsilon(\alpha)=\sum_{j=1}^kj(|\alpha_j|+1)+1.
\]

For all $\beta\in\sly$, $k,l\ge 0$,  $(k,l,\beta) \not\in\{ (1,0,\beta_0),(0,0,\beta_0)\}$, define
\[
\qkl^\beta:C^{\otimes k}\otimes A^*(X;Q)^{\otimes l} \lrarr C
\]
by
\begin{align*}
\q^{\beta}_{k,l}(\alpha_1\otimes\cdots\otimes\alpha_k;\gamma_1\otimes\cdots\otimes\gamma_l):=
(-1)^{\varepsilon(\alpha)}
(evb_0^\beta)_* \left(\bigwedge_{j=1}^l(evi_j^\beta)^*\gamma_j\wedge\bigwedge_{j=1}^k (evb_j^\beta)^*\alpha_j\right).
\end{align*}
The case $\q_{0,0}^{\beta}$ is understood as $-(evb_0^{\beta})_*1.$
Furthermore, for $l\ge 0$, $(l,\beta)\neq (1,\beta_0),(0,\beta_0)$, define
\[
\q_{-1,l}^\beta:A^*(X;Q)^{\otimes l}\lrarr R
\]
by
\begin{gather}\label{eq:q-1def}
\q_{-1,l}^\beta(\gamma_1\otimes\cdots\otimes\gamma_l):= \int_{\M_{0,l}(\beta)} \bigwedge_{j=1}^l (evi_j^\beta)^*\gamma_j.
\end{gather}
Define
\[
\q_{1,0}^{\beta_0}(\alpha):=d\alpha,\quad \q_{0,0}^{\beta_0}:=0,\quad \q_{-1,1}^{\beta_0}:=0,\quad \q_{-1,0}^{\beta_0}:=0.
\]
Set
\begin{align*}
\qkl:=\sum_{\beta\in\sly}T^{\beta}\qkl^{\beta}.
\end{align*}

Lastly, define similar operations using spheres,
\[
\q_{\emptyset,l}:A^*(X;Q)^{\otimes l}\lrarr A^*(X;R),
\]
as follows. For $\beta\in H_2(X;\Z)$ let $\M_{l+1}(\beta)$ be the moduli space of stable $J$-holomorphic spheres with $l+1$ marked points indexed from 0 to $l$ representing the class $\beta$, and let $ev_j^\beta:\M_{l+1}(\beta)\to X$ be the evaluation maps. Assume that all the moduli spaces $\M_{l+1}(\beta)$ are smooth orbifolds and $ev_0$ is a submersion. Let
\begin{equation}\label{eq:varpi}
\pr: H_2(X;\Z) \to \sly
\end{equation}
denote the projection.
Recall that the relative spin structure $\s$ determines a class $w_{\s} \in H^2(X;\Z/2\Z)$ such that $w_2(TL) = i^* w_{\s}$. By abuse of notation we think of $w_{\s}$ as acting on $H_2(X;\Z)$.

For $l\ge 0$, $(l,\beta)\ne (1,0),(0,0)$, set
\begin{gather*}
\q_{\emptyset,l}^\beta(\gamma_1,\ldots,\gamma_l):=
(-1)^{w_{\s}(\beta)}
(ev_0^\beta)_*(\wedge_{j=1}^l(ev_j^\beta)^*\gamma_j),\\
\q_{\emptyset,l}(\gamma_1,\ldots,\gamma_l):=
\sum_{\beta\in H_2(X)}
T^{\pr(\beta)}\q_{\emptyset,l}^\beta(\gamma_1,\ldots,\gamma_l),
\end{gather*}
and define
\[
\q_{\emptyset,1}^0:= 0,\qquad \q_{\emptyset,0}^0:= 0.
\]
The sign $(-1)^{w_\s(\beta)}$ is designed to compensate for the gluing sign in Proposition~\ref{rem:pintgluing}, as in Lemma~\ref{lm:starofDtype}.

In Proposition~\ref{lm:qlinear} we prove that the $\q$ operators defined in this section are $R$-linear in the proper sense.

\subsubsection{Relations}

In dealing with the next result, we will be using the following notation conventions.

A list is a finite sequence. We write $A\leq B$ if $A$ is a sublist of $B$.  Denote by $[k]$ a fundamental list of integers, namely,
\[
[k]:=(1,\ldots,k).
\]
An ordered $3$-partition of $[k]$ is a partition of $[k]$ to three  sublists $(1:3),(2:3),$ and $(3:3),$ such that
\[
(1:3)\circ (2:3)\circ (3:3)=[k].
\]
For example, a possible ordered $3$-partition of $[7]$ is $\{(1,2),(3,4,5,6),(7)
\}$, so in the above notation $(1:3)=(1,2),$ $(2:3)=(3,4,5,6),$ and $(3:3)=(7)$.
On the other hand, $\{(1,2),(4,3,5,6),(7)
\}$ and $\{(1,3),(2,4,5,6),(7)
\}$ are not ordered 3-partition of $[7]$, because the order in each is violated.

Use $|(i:3)|$ to denote the length of the corresponding sub-list, $i=1,2,3$. So, if
\[
(1:3)=(1,\ldots,i_1),\quad(2:3)=(i_1+1,\ldots,i_1+i_2),\quad (3:3)=(i_1+i_2+1,\ldots,k),
\]
then $|(1:3)|:=i_1$, $|(2:3)|=i_2$, and $|(3:3)|=k-i_1-i_2$. We allow a sub-list to be empty, in which case its length is $0$.

Denote the set of all ordered $3$-partitions of $[k]$ by $S_3[k]$.
Similarly, denote by $S_2[k]$ the set of ordered 2-partitions of $[k]$.

For a list $\alpha=(\alpha_1,\ldots,\alpha_k)$ and any (ordered) sub-list of indices $I\le [k]$, write $\alpha^{I}$ for the ordered sub-list of $\alpha$ with indices in $I$. Write $|\alpha^I|$ for $\sum_{i\in I}|\alpha_i|$. In the special case $I=[k]$ write simply $|\alpha|:=|\alpha^I|$.

Let $I\sqcup J=[l]$ be a partition of $[l]$ in the usual sense, not respecting the order of $[l].$ Equip the subsets $I$ and $J$ with the order induced from $[l].$ Let $\gamma=(\gamma_1,\ldots,\gamma_l)$ be a list of differential forms. Define
$sgn(\sigma^\gamma_{I\cup J})$ by the equation
\[
\bigwedge_{i \in I} \gamma_i \wedge \bigwedge_{j \in J} \gamma_j = (-1)^{sgn(\sigma^\gamma_{I\cup J})}\bigwedge_{k \in [l]} \gamma_k,
\]
where the wedge products are taken in the order of the respective lists.
Explicitly,
\[
sgn(\sigma^\gamma_{I\cup J})\equiv
\sum_{\substack{i\in I,j\in J\\ j<i}}|\gamma_i|\cdot|\gamma_j|\pmod 2.
\]

\begin{cl}\label{q_rel}
For any fixed $\alpha=(\alpha_1,\ldots,\alpha_k)$, $\gamma=(\gamma_1,\ldots,\gamma_l)$,
\begin{align}\label{q_rel_eq}
&0=\sum_{\substack{S_3[l]\\(2:3)=\{j\}}}(-1)^{|\gamma^{(1:3)}|+1} \qkl(\alpha;
\gamma^{(1:3)}\otimes d\gamma_j\otimes\gamma^{(3:3)})+\mbox{}\\
&+\sum_{\substack{P\in S_3[k]\\I\sqcup J=[l]}}
(-1)^{\iota(\alpha,\gamma;P,I)}
\q_{|(1:3)|+|(3:3)|+1,|I|}(\alpha^{(1:3)}\otimes \q_{|(2:3)|,|J|} (\alpha^{(2:3)};\gamma^J)\otimes\alpha^{(3:3)};\gamma^I),\notag
\end{align}
where
\[
\iota(\alpha,\gamma;P,I)\equiv
\big(\sum_{j\in J}|\gamma_j|+1\big)\cdot \sum_{j\in (1:3)}(|\alpha_j|+1)+\sum_{j\in I}|\gamma_j|+sgn(\sigma_{I\cup J}^\gamma)\pmod 2.
\]
\end{cl}

\begin{figure}[ht]
\centering
\includegraphics[width=9cm]{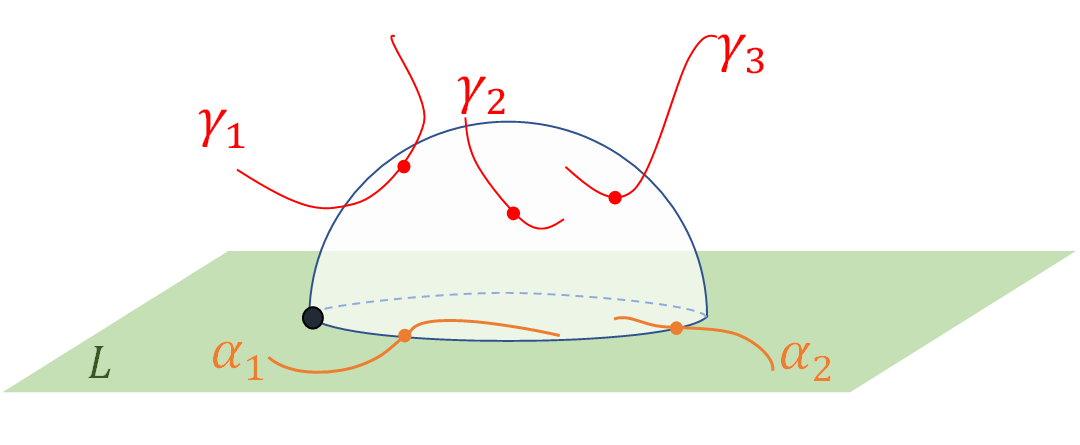}
\caption{The fat black dot represents a point of the chain $Q$ ``Poincar\'e dual'' to $\q_{2,3}(\a_1,\a_2;\gamma_1,\gamma_2,\gamma_3)$.}
\label{fig:mdsp}
\end{figure}

\begin{figure}
\centering
\includegraphics[width=9cm]{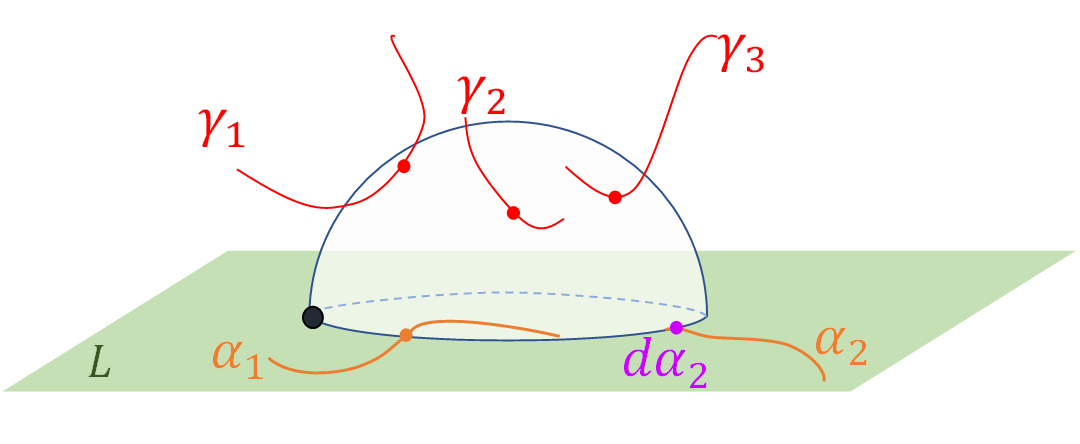}
\caption{The fat black dot represents a point of the portion of the boundary of $Q$ arising from $J$-holomorphic disks passing through the boundary of the constraint corresponding to $\a_2$.}
\label{fig:bdda}
\end{figure}

\begin{figure}
\centering
\includegraphics[width=9cm]{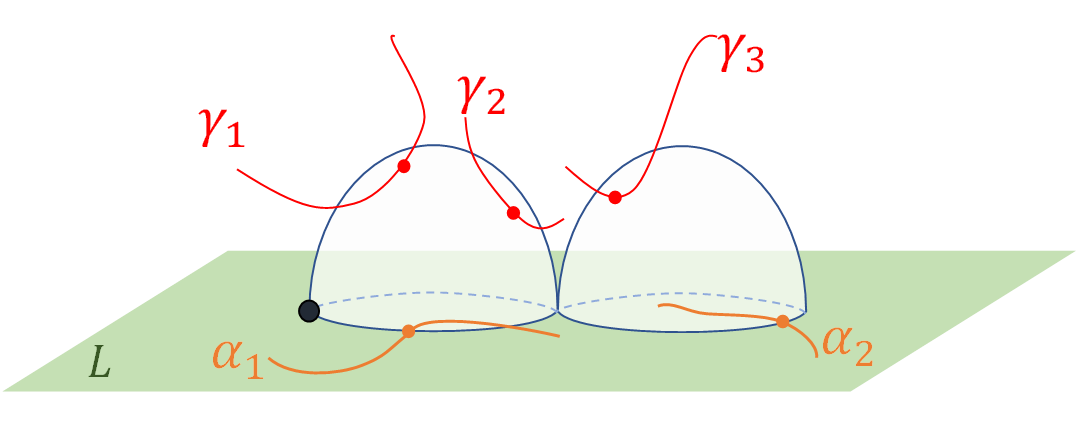}
\caption{The fat black dot represents a point of the portion of the boundary of $Q$ arising from the boundary of the moduli spaces $\M_{3,3}(\beta)$.}
\label{fig:bdbubb}
\end{figure}

Intuitively, equation~\eqref{q_rel_eq} describes the boundary of the chain Q ``Poincar\'e dual'' to $\qkl(\a;\gamma)$. Indeed, the term of the second summand of equation~\eqref{q_rel_eq} with $(|(1:3)|,|I|,\beta) = (1,0,\beta_0)$ corresponds to the boundary of $Q.$ Think of the chain $Q$ as the union of the boundaries of $J$-holomorphic disks passing through constraints ``Poincar\'e dual'' to $\a$ and $\gamma.$ See Figure~\ref{fig:mdsp}. One type of contribution to the boundary of Q comes from $J$-holomorphic disks passing through the boundary of the constraints corresponding to $\a$ and $\gamma$, as in the first summand of equation~\eqref{q_rel_eq} and the part of the second summand that corresponds to $(|(2:3)|,|J|,\beta)=(1,0,\beta_0)$. See Figure~\ref{fig:bdda}. The other type of contribution to the boundary of $Q$ comes  from the boundaries of the moduli spaces $\M_{k+1,l}(\beta)$ arising from disk bubbling and is reflected in the remainder of the second summand of equation~\eqref{q_rel_eq}. See Figure~\ref{fig:bdbubb}.

A proof of Proposition~\ref{q_rel} is given in Section~\ref{sssec:proof} below using Propositions~\ref{prop:proppp} and~\ref{stokes} and the description of the boundary of $\M_{k+1,l}(\beta)$ in terms of fiber-products.

\begin{cl}\label{q_-1_rel}
For any $\gamma=(\gamma_1,\ldots,\gamma_l)$,
\begin{align}\label{eq:q-1}
0=&\sum_{(2:3)=\{j\}}(-1)^{|\gamma^{(1:3)}|+1} \q_{-1,l}(\gamma^{(1:3)}\otimes d\gamma_j\otimes\gamma^{(3:3)})+\\
&+\frac{1}{2}\sum_{I\sqcup J=\{1,\ldots,l\}}
(-1)^{\iota(\gamma;I)}
\langle\q_{0,|I|}(\gamma^I),\q_{0,|J|}(\gamma^J)\rangle
+(-1)^{|\gamma|+1}\int_L i^* \q_{\emptyset,l}(\gamma), \notag
\end{align}
where
\[
\iota(\gamma;I)\equiv
\sum_{j\in I}|\gamma_j|+sgn(\sigma_{I\cup J}^\gamma)\pmod 2.
\]
\end{cl}

Intuitively, equation~\eqref{eq:q-1} describes the boundary of the 1-dimensional part of the chain $Q'$ ``Poincar\'e dual'' to $\bigwedge_{j=1}^l (evi_j^\beta)^*\gamma_j$, the integrand in the definition of $\q_{-1,l}(\gamma)$. Think of $Q'$ as the space of $J$-holomorphic disks passing through constraints Poincar\'e dual to $\gamma.$ One type of contribution to the boundary of $Q'$ comes from $J$-holomorphic disks passing through the boundary of the constraints corresponding to $\gamma$, as in the first summand of equation~\eqref{eq:q-1}. The second type of contribution to the boundary of $Q'$ comes from the boundaries of the moduli spaces $\M_{0,l}(\beta)$; disk bubbling is reflected in the second summand of equation~\eqref{eq:q-1} and the collapse of the boundary of the disk to a point is reflected in the third summand.

A proof of Proposition~\ref{q_-1_rel} is given in Section~\ref{ssec:q-1} below using Propositions~\ref{prop:proppp} and~\ref{stokes} and the description of the boundary of $\M_{0,l}(\beta)$ in terms of fiber-products.

Fix a closed form $\gamma\in \mI_QD$ with $|\gamma|=2$.
For $\alpha_1,\ldots,\alpha_k \in C,$ define
\begin{equation}\label{eq:mkqk}
\m_k^{\beta,\gamma}(\otimes_{j=1}^k\alpha_j)=
\sum_l\frac{1}{l!}\qkl^\beta(\otimes_{j=1}^k\alpha_j;\gamma^{\otimes l}),\quad
\m_k^{\gamma}(\otimes_{j=1}^k\alpha_j)=
\sum_l\frac{1}{l!}\qkl(\otimes_{j=1}^k\alpha_j;\gamma^{\otimes l}),
\end{equation}
for all $k\ge -1,l\ge 0$. In particular, $\m_k^{\beta,\gamma}(\otimes_{j=1}^k\alpha_j) \in C$ and $\m_{-1}^\gamma\in R$. Observe that this definition of $\m_k$ agrees with the definition in Section~\ref{ssec:sor} for $k \geq 0.$
\begin{cl}[$A_\infty$ relations]\label{cl:a_infty_m}
The operations $\{\m_k^\gamma\}_{k\ge 0}$ define an $A_\infty$ structure on $C$. That is,
\begin{equation*}\label{eq:a-infty}
\sum_{S_3[k]}(-1)^{\sum_{j\in(1:3)}(|\alpha_j|+1)}
\mg_{|(1:3)|+|(3:3)|+1}(\alpha^{(1:3)}\otimes\mg_{|(2:3)|}(\alpha^{(2:3)}) \otimes\alpha^{(3:3)})=0.
\end{equation*}
\end{cl}
\begin{proof}
Since we have assumed $d\gamma = 0$ and $|\gamma|=2$, this follows from equation~\eqref{eq:mkqk} and Proposition~\ref{q_rel}.
\end{proof}

\subsection{Proof for \texorpdfstring{$k\ge 0$}{k >=0}}\label{sssec:proof}
This section is devoted to the proof of Proposition~\ref{q_rel}.

\begin{lm}\label{lm:deg_q}
The map $evb_0:\M_{k+1,l}(\beta)\to L$ satisfies
$\rdim(evb_0^\beta)\equiv k\pmod 2.$
\end{lm}
\begin{proof}
Since $L$ is orientable, $\mu(\beta)$ is even. Therefore,
\[
\rdim(evb_0^{\beta})
=n-3+\mu(\beta)+k+1+2l-n
=\mu(\beta)+k+2l-2
\equiv k \pmod 2.
\]
\end{proof}

For a list of indices $I$, denote by $\M_{k+1,I}(\beta)$ the moduli space diffeomorphic to $\M_{k+1,|I|}(\beta)$ with interior marked points labeled by $I$. It carries evaluation maps $evb_j^\beta$ with $j\in\{0,\ldots,k\}$ and $evi_m^\beta$ with $m$ taken from $I$. Note that the diffeomorphism
\[
\M_{k+1,I}(\beta)\stackrel{\sim}{\lrarr} \M_{k+1,|I|}(\beta)
\]
preserves orientation, no matter how we identify $I$ with $[|I|]$.

\begin{prop}\label{rem:gluingsign}
Let $k\in \Z_{\ge -1},$ $l\in \Z_{\ge 0},$ $\beta\in \sly$. Let $k_i,\beta_i,$ ($i=1,2$) be such that $k_1+k_2=k+1$ and $\beta_1+\beta_2=\beta$. Let $I\sqcup J=[l]$ be a partition.
Let $B\subset \d \M_{k+1,l}(\beta)$ be the boundary component where a disk bubbles off at the $i$-th boundary point, with $k_2$ of the boundary marked points and the interior marked points labeled by $J$. See Figure~\ref{fig:q_eq}. Then the canonical map
\[
\vartheta:\M_{k_1+1,\,I}(\beta_1) \prescript{}{evb_i^{\beta_1}}\times_{evb_0^{\beta_2}}  \M_{k_2+1,\,J}(\beta_2) \stackrel{\sim}{\lrarr}B
\]
is a diffeomorphism unless $k = -1, I = \emptyset = J$, and $\beta_1 = \beta_2.$ In the exceptional case, $\vartheta$ is a $2$ to $1$ local diffeomorphism in the orbifold sense. In both cases, $\vartheta$ changes orientation by the sign $(-1)^{\delta_1}$ with
\begin{equation}\label{eq:delta}
\delta_1:=k_2(k_1-i)+i-n.
\end{equation}
\end{prop}
\begin{proof}
The case $i = 1$ is~\cite[Proposition 8.3.3]{FOOO}. The proof for other values of $i$ is similar. See also~\cite[Theorem 4.3.3(b)]{WehrheimWoodward} for an in-depth discussion of the sign of gluing at a boundary node.
\end{proof}

\begin{figure}[ht]
\centering
\includegraphics[width=9cm]{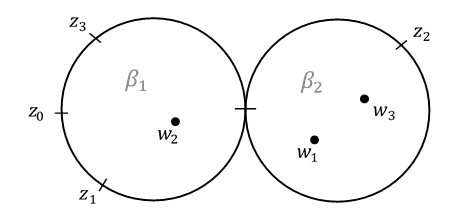}
\caption{The domain of an element of $B\subset \d\M_{4,3}(\beta),$ with $k_1=3,$ $k_2=1,$ $i=2,$ $I=\{2\},$ and $J=\{1,3\}$.}
\label{fig:q_eq}
\end{figure}

For $k\in\Z_{\ge 0}$ and $a=(a_1,\ldots,a_m)\in \Z^{\times m},$ denote
\[
\varepsilon'(k):=\sum_{j=1}^k j+1=\frac{k(k+1)}{2}+1,\qquad
\varepsilon''(a):=\sum_{j=1}^mj\cdot a_j.
\]
In particular, for a list $a$ of length $k$, we have
\[
\varepsilon(a)=\varepsilon'(k)+\varepsilon''(a).
\]
As with $\varepsilon,$ we allow differential forms as input for $\varepsilon''$ in lieu of their degrees.

\begin{lm}\label{lm:sepepsilon}\label{lm:sumofepsilons}
Let $\alpha=(\a_1,\ldots,\a_k)\in A^*(L;R)^{\otimes k}$ and $\gamma=(\gamma_1,\ldots,\gamma_l)\in A^*(X;Q)^{\otimes l}.$ Fix an element of $S_3[k]$ and $I\sqcup J$ a partition of $[l]$, and set $k_1=|(1:3)|+|(3:3)|+1$, $k_2=|(2:3)|$, and $i=|(1:3)|+1$.
Then
\begin{enumerate}
  \item $\varepsilon'(k_1)+\varepsilon'(k_2)
      \equiv\varepsilon'(k)+ k+k_1k_2\pmod 2.$
  \item
  \begin{multline*}
\qquad\varepsilon''(\alpha^{(1:3)},|\alpha^{(2:3)}|+|\gamma^J|+k_2, \alpha^{(3:3)}) +\varepsilon''(\a^{(2:3)})
      \equiv\\
     \equiv\varepsilon''(\a)+ik_2+k_2|\a^{(3:3)}|+ |\a|+|\a^{(1:3)}|+ i|\gamma^J|\pmod 2.
  \end{multline*}
  \item
  \begin{multline*}
  \qquad \varepsilon(\alpha^{(1:3)},|\alpha^{(2:3)}|+|\gamma^J|+k_2, \alpha^{(3:3)})+ \varepsilon(\alpha^{(2:3)})
\equiv\\
\equiv
\varepsilon(\a)+|\a|+k+|\a^{(1:3)}|
+i|\gamma^J|+k_2|\a^{(3:3)}|+k_1k_2+ik_2
\pmod 2.
\end{multline*}
\end{enumerate}
\end{lm}

\begin{proof}\leavevmode
\begin{enumerate}
\item
Recall that $k_1+k_2=k+1$. Therefore,
\begin{align*}
\varepsilon'(k_1)+\varepsilon'(k_2)&=
\frac{k_1(k_1+1)+k_2(k_2+1)}{2}+2\\
\equiv& \frac{k_1^2+k_2^2+k+1}{2}\\
=&\frac{(k_1+k_2)^2-2k_1k_2+k+1}{2}\\
\equiv&\frac{(k+1)(k+1+1)}{2}+k_1k_2\\
\equiv& \frac{k(k+1)}{2}+k+1+k_1k_2\\
=&\varepsilon'(k)+k+k_1k_2 \pmod 2.
\end{align*}
\item
\begin{align*}
\varepsilon''(\alpha^{(1:3)}&,|\alpha^{(2:3)}|+|\gamma^J|+k_2, \alpha^{(3:3)})
+\varepsilon''(\a^{(2:3)})
\equiv\\
\equiv&
\sum_{j=1}^{i-1}j|\a_j|+i(|\alpha^{(2:3)}|+|\gamma^J|+k_2)+ \sum_{j=i+1}^{k_1}j|\a_{j+k_2-1}|+
\sum_{j=1}^{k_2}j|\a_{j+i-1}|\\
\equiv&
\sum_{j=1}^{i-1}j|\a_j|+i(|\alpha^{(2:3)}|+k_2)+ \sum_{j=i+k_2}^{k}(j-k_2+1)|\a_j|+
\sum_{j=i}^{i+k_2-1}(j-i+1)|\a_j|+i|\gamma^J|\\
\equiv&
\sum_{j=1}^{k}j|\a_j|+ik_2 -k_2|\a^{(3:3)}|+|\a^{(3:3)}|+|\a^{(2:3)}|+ i|\gamma^J|\\
\equiv&
\varepsilon''(\a)+ik_2+k_2|\a^{(3:3)}|+ |\a|+|\a^{(1:3)}|+ i|\gamma^J| \pmod 2.
\end{align*}
\item
This is the result of summing the two first statements.
\end{enumerate}

\end{proof}

\begin{lm}\label{lm:star}
Let $B$ be the boundary component of $\M_{k+1,l}(\beta)$ described in Proposition~\ref{rem:gluingsign}, and let $\delta_1$ be the sign of the gluing map given there. Fix the ordered $3$-partition of $[k]$ such that $i=|(1:3)|+1$ and $k_2=|(2:3)|$. Write $l_1:=|I|,$ $l_2:=|J|$.
Then
\[
(evb_0|_B)_*\big(\bigwedge_{j=1}^l evi_j^*\gamma_j\wedge\bigwedge_{j=1}^k evb_j^*\alpha_j\big)=
(-1)^{*}
\q_{k_1,l_1}^{\beta_1}(\alpha^{(1:3)}\otimes \q_{k_2,l_2}^{\beta_2}(\alpha^{(2:3)};\gamma^J)\otimes \alpha^{(3:3)};\gamma^I),
\]
with
\begin{multline}\label{eq:starone}
*=
\delta_1+|\gamma^J|\cdot|\alpha^{(1:3)}|
+\rdim(evb_0^{\beta_2})\cdot|\a^{(3:3)}|+\\
+\varepsilon(\alpha^{(1:3)}, (evb_0^{\beta_2})_*\xi_2,\alpha^{(3:3)})+ \varepsilon(\alpha^{(2:3)})+sgn(\sigma_{I\cup J}^\gamma),
\end{multline}
or, equivalently,
\begin{equation}\label{eq:startwo}
*=i+n+|\gamma^J|\cdot|\alpha^{(1:3)}|+\varepsilon(\a)+ |\a|+k+|\a^{(1:3)}|
+i|\gamma^J|+sgn(\sigma_{I\cup J}^\gamma).
\end{equation}
\end{lm}

\begin{proof}
Write
\[
\xi=\bigwedge_{j=1}^l evi_j^*\gamma_j\wedge\bigwedge_{j=1}^k evb_j^*\alpha_j.
\]
Consider the pull-back diagram
\[
\xymatrix{
{\M_{k_1+1,\,I}(\beta_1)\times_L\M_{k_2+1,\,J}(\beta_2)}\ar[r]^(.6){\quad p_2}\ar[d]^{p_1}&
{\M_{k_2+1,\,J}(\beta_2)}\ar[d]^{evb_0^{\beta_2}}\\
{\M_{k_1+1,\,I}(\beta_1)}\ar[r]^{\quad evb_i^{\beta_1}}&L .
}
\]
We use the notation $evb_j^{\beta_i},evi_j^{\beta_i},$ for $i=1,2,$ to denote the evaluation maps on the spaces $\M_{k_1+1,\,I}(\beta_1),\M_{k_2+1,\,J}(\beta_2),$ respectively.
Set
\[
\bar\xi:=\vartheta^*\xi,
\]
with $\vartheta$ from Proposition~\ref{rem:gluingsign}, and
\begin{gather*}
\xi_1:=\bigwedge_{j\in I}(evi_j^{\beta_1})^*\gamma_j\wedge\bigwedge_{j=1}^{i-1}(evb_j^{\beta_1})^*\alpha_j\wedge\bigwedge_{j=i+1}^{k_1}(evb_j^{\beta_1})^*\alpha_{j+k_2-1},\\
\xi_2:=\bigwedge_{j\in J}(evi_j^{\beta_2})^*\gamma_j\wedge \bigwedge_{j=1}^{k_2} (evb_j^{\beta_2})^*\alpha_{j+i-1}.
\end{gather*}
Note that
\[
\bar\xi=(-1)^{\delta_2}p_1^*\xi_1\wedge p_2^*\xi_2,
\]
with
\[
\delta_2:=(|\alpha^{(2:3)}|+|\gamma^J|)\cdot|\alpha^{(3:3)}|+|\gamma^J|\cdot|\alpha^{(1:3)}|+
sgn(\sigma_{I\cup J}^\gamma).
\]
By property \eqref{prop:pushfiberprod} of the push forward,
\begin{equation*}
(evb_i^{\beta_1})^*(evb_0^{\beta_2})_*\xi_2=p_{1*}^{\mbox{}}p_2^*\xi_2.
\end{equation*}
Using in addition
properties \eqref{prop:pushcomp}-\eqref{prop:pushpull}, we compute
\begin{align*}
(evb_0|_{B}&)_*\xi=
(-1)^{\delta_1}(evb_0^{\beta_1})_*p_{1*}\bar\xi\\
=&(-1)^{\delta_1+\delta_2}(evb_0^{\beta_1})_*^{\mbox{}}p_{1*}^{\mbox{}}\left(p_1^*\xi_1\wedge p_2^*\xi_2\right)\\
=&(-1)^{\delta_1+\delta_2}(evb_0^{\beta_1})_*\left(\xi_1\wedge p_{1*}^{\mbox{}}p_2^*\xi_2\right)\\
=&(-1)^{\delta_1+\delta_2}(evb_0^{\beta_1})_*\left(\xi_1\wedge (evb_i^{\beta_1})^*(evb_0^{\beta_2})_*\xi_2\right)\\
=&(-1)^{\delta_1+\delta_2 +|(evb_0^{\beta_2})_*\xi_2|\cdot|\alpha^{(3:3)}|}(evb_0^{\beta_1})_*\Bigg(
\bigwedge_{j\in I}(evi_j^{\beta_1})^*\gamma_j\wedge\\
&\hspace{7em}\wedge\bigwedge_{j=1}^{i-1}(evb_j^{\beta_1})^*\alpha_j
\wedge(evb_i^{\beta_1})^*(evb_0^{\beta_2})_*\xi_2
\wedge\bigwedge_{j=i+1}^{k_1}(evb_j^{\beta_1})^*\alpha_{j+k_2-1}
\Bigg),\\
\intertext{and since $(evb_0^{\beta_2})_*\xi_2 = (-1)^{\varepsilon(\alpha^{(2:3)})}\q_{k_2,l_2}^{\beta_2}(\alpha^{(2:3)};\gamma^J),$ we continue}
=&(-1)^*
\q_{k_1,l_1}^{\beta_1}(\alpha^{(1:3)}\otimes\q_{k_2,l_2}^{\beta_2}(\alpha^{(2:3)};\gamma^J)\otimes\alpha^{(3:3)};\gamma^I)
\end{align*}
with
\[
*=\delta_1+\delta_2+ |(evb_0^{\beta_2})_*\xi_2|\cdot|\alpha^{(3:3)}|+
\varepsilon(\alpha^{(1:3)}, (evb_0^{\beta_2})_*\xi_2,\alpha^{(3:3)})+ \varepsilon(\alpha^{(2:3)}).
\]
Note that
\[
|(evb_0^{\beta_2})_*\xi_2|=|\alpha^{(2:3)}|+ |\gamma^J|-\rdim(evb_0).
\]
Therefore,
\begin{align*}
*
\equiv& \delta_1 +(|\alpha^{(2:3)}|+|\gamma^J|)\cdot|\alpha^{(3:3)}| +|\gamma^J|\cdot|\alpha^{(1:3)}|+
sgn(\sigma_{I\cup J}^\gamma) +\\
&+(|\alpha^{(2:3)}|+ |\gamma^J| +\rdim(evb_0))|\a^{(3:3)}|+
\varepsilon(\alpha^{(1:3)}, (evb_0^{\beta_2})_*\xi_2,\alpha^{(3:3)})+ \varepsilon(\alpha^{(2:3)})\\
\equiv& \delta_1 +|\gamma^J|\cdot|\alpha^{(1:3)}|+
sgn(\sigma_{I\cup J}^\gamma) +\rdim(evb_0)|\a^{(3:3)}|+\\
&\hspace{13em}+
\varepsilon(\alpha^{(1:3)}, (evb_0^{\beta_2})_*\xi_2,\alpha^{(3:3)})+ \varepsilon(\alpha^{(2:3)}) \pmod 2.
\end{align*}
This proves equation~\eqref{eq:starone}. By the definition~\eqref{eq:delta} of $\delta_1,$ Lemma~\ref{lm:deg_q}, and Lemma~\ref{lm:sumofepsilons}, we therefore have
\begin{align*}
*\equiv&
k_1k_2+ik_2+i+n +|\gamma^J|\cdot|\alpha^{(1:3)}|+
sgn(\sigma_{I\cup J}^\gamma) +\\
&+k_2|\a^{(3:3)}|+
\varepsilon(\a)+|\a|+k+|\a^{(1:3)}|
+i|\gamma^J|+k_2|\a^{(3:3)}|+k_1k_2+ik_2\\
\equiv&
i+n +|\gamma^J|\cdot|\alpha^{(1:3)}|+
sgn(\sigma_{I\cup J}^\gamma) +
\varepsilon(\a)+|\a|+k+|\a^{(1:3)}|
+i|\gamma^J| \pmod 2.
\end{align*}
This proves equation~\eqref{eq:startwo}.

\end{proof}

\begin{proof}[Proof of Proposition \ref{q_rel}]
Apply Proposition~\ref{stokes} to the case $M=\M_{k+1,l}(\beta)$, $f=evb_0$, and
\[
\xi=\bigwedge_{j=1}^l evi_j^*\gamma_j\wedge\bigwedge_{j=1}^k evb_j^*\alpha_j.
\]
Let us see how each of the elements in Stokes' theorem looks in terms of $\q$.

\textit{First element: $d(f_*\xi)$}.
This is
\[
d((evb_0)_*\xi)=
(-1)^{\varepsilon(\alpha)} \q_{\,1,0}^{\beta_0}\big(\qkl^{\beta}(\alpha;\gamma)\big).
\]

\textit{Second element: $f_*(d\xi)$}.
This gives
\begin{align*}
(evb_0)_*(d\xi)=&(evb_0)_*\left(d\left(\bigwedge_{j=1}^l evi_j^*\gamma_j\right)\wedge\bigwedge_{j=1}^k evb_j^*\alpha_j\right)+\mbox{}\\
&+(-1)^{|\gamma|}(evb_0)_*\left(\bigwedge_{j=1}^l evi_j^*\gamma_j\wedge d\left(\bigwedge_{j=1}^k evb_j^*\alpha_j\right)\right)\\
=&\sum_{\substack{S_3[l] \\ (2:3) = \{i\}}}(-1)^{\varepsilon(\alpha)+|\gamma^{(1:3)}|}\q^\beta_{k,l}(\alpha;\gamma^{(1:3)}\otimes d\gamma_i\otimes\gamma^{(3:3)})+ \mbox{}\\
&+\sum_{\substack{S_3[k] \\ (2:3) = \{i\}} }(-1)^{|\gamma|+\varepsilon(\alpha)+i+\sum_{j=1}^{i-1}|\alpha_j|}\q^{\beta}_{k,l}(\alpha^{(1:3)}\otimes d\alpha_i\otimes\alpha^{(3:3)};\gamma).
\end{align*}
Further,
\[
\qkl^\beta(\alpha^{(1:3)}\otimes d\alpha_i\otimes\alpha^{(3:3)};\gamma)
=
\qkl^{\beta}(\alpha^{(1:3)}\otimes\q_{1,\,0}^{\beta_0}
(\alpha^{(2:3)})\otimes\alpha^{(3:3)};\gamma).
\]

\textit{Third element: $\big(f\big|_{\partial M}\big)_*\xi$}.

Let $B$ be a boundary component as in Proposition~\ref{rem:gluingsign}. Write $l_1:=|I|,$ $l_2:=|J|$.

The dimension of the domain of $evb_0$ is
\[
k+1+2l+\mu(\beta)+n-3=k-2+2l+\mu(\beta)+n\equiv k+n\pmod 2,
\]
and $|\xi| = |\alpha| + |\gamma|.$ Therefore, the contribution of $(f|_B)_*\xi$ to Stokes' theorem comes with the sign $(-1)^{|\alpha| + |\gamma| + k +n}.$
We claim that
\begin{align*}
-(-1)^{|\alpha|+|\gamma|+k+n}(f|_{B})_*\xi
=(-1)^{\varepsilon(\alpha)+\iota(\alpha,\gamma;P,I)}
\q_{k_1,l_1}^{\beta_1}(\alpha^{(1:3)}\otimes \q_{k_2,l_2}^{\beta_2}(\alpha^{(2:3)};\gamma^J)\otimes \alpha^{(3:3)};\gamma^I).
\end{align*}
Indeed,
by Lemma~\ref{lm:star}, we have
\begin{align*}
*+|\alpha|&+|\gamma|+k+n+1+\varepsilon(\alpha)
\equiv\\
\equiv&
i+n+|\gamma^J|\cdot|\alpha^{(1:3)}|+\varepsilon(\a) +|\a|+k+|\a^{(1:3)}|+i|\gamma^J|+sgn(\sigma_{I\cup J}^\gamma)+\\
&+|\alpha|+|\gamma|+k+n+1+\varepsilon(\alpha)\\
\equiv&
i+|\gamma^J|\cdot|\alpha^{(1:3)}|+|\a^{(1:3)}|+i|\gamma^J| +sgn(\sigma_{I\cup J}^\gamma)+|\gamma|-1\\
\equiv&
|\gamma^J|\cdot(|\alpha^{(1:3)}|+(i-1))+(|\a^{(1:3)}|+i-1)
+|\gamma|+|\gamma^J|+sgn(\sigma_{I\cup J}^\gamma)\\
\equiv&
(|\gamma^J|+1)\cdot(|\alpha^{(1:3)}|+(i-1))
+|\gamma^I|+sgn(\sigma_{I\cup J}^\gamma)\\
\equiv&\iota(\a;\gamma;P,I) \pmod 2.
\end{align*}

Since there is one boundary node, $k_1$ is at least $1$.
Also, the stability of each of the disk components implies that
\[
(\beta_1,k_1,l_1)\ne (\beta_0,1,0),\quad (\beta_2,k_2,l_2)\ne (\beta_0,1,0),(\beta_0,0,0).
\]
So, the total contribution of the summand
$(-1)^{s+t+1}\big(f\big|_{\partial M}\big)_*\xi$ in Stokes' theorem is
\[
(-1)^{\varepsilon(\alpha)}\hspace{-3em}\sum_{\substack{\beta_1+\beta_2=\beta\\ k_1+k_2=k+1,\; k_1\ge 1\\ l_1+l_2=l \\(\beta_1,k_1,l_1)\ne(\beta_0,1,0)
\\ (\beta_2,k_2,l_2)\not\in\{(\beta_0,0,0),(\beta_0,1,0)\}}}
\hspace{-3em}(-1)^{\iota(\alpha,\gamma;P,I)}\q_{k_1,l_1}^{\beta_1}(\alpha^{(1:3)}\otimes\q_{k_2,l_2}^{\beta_2} (\alpha^{(2:3)};\gamma^{J})\otimes\alpha^{(3:3)};\gamma^{I}).
\]

\textit{Deducing the equations}.
All that is left now is to plug the various expressions into Stokes' formula. Let us rewrite it first:
\[
0=d (f_*\xi)-f_*(d \xi)-(-1)^{s+t}\big(f\big|_{\partial M}\big)_*\xi.
\]
We showed that
\begin{align*}
0=&(-1)^{\varepsilon(\alpha)}
\Big(\q_{\,1,0}^0(\qkl^\beta(\alpha;\gamma))+
(-1)^{|\gamma^{(1:3)}|+1}\qkl^\beta(\alpha;\gamma^{(1:3)},d\gamma_i,\gamma^{(3:3)})+\\
&\quad+
(-1)^{|\gamma|+\sum_{j=1}^{i-1}|\alpha_j|+i+1}\qkl^\beta(\alpha^{(1:3)},d\alpha_i,\alpha^{(3:3)};\gamma)+\\
&\quad+\hspace{-2em}\sum_{\substack{\beta_1+\beta_2=\beta\\ k_1+k_2=k+1,\; k_1\ge 1\\ l_1+l_2=l \\(\beta_1,k_1,l_1)\ne(0,1,0)
\\ (\beta_2,k_2,l_2)\not\in\{(0,0,0),(0,1,0)\}}}
\hspace{-3em}(-1)^{\iota(\alpha,\gamma;i,I)}
\q_{k_1,l_1}^{\beta_1}(\alpha^{(1:3)}\otimes\q_{k_2,l_2}^{\beta_2} (\alpha^{(2:3)};\gamma^J)\otimes\alpha^{(3:3)};\gamma^I)\Big)\\
=&(-1)^{\varepsilon(\alpha)}\Big((-1)^{|\gamma^{(1:3)}|+1} \qkl^\beta(\alpha;\gamma^{(1:3)},d\gamma_i,\gamma^{(3:3)})+\mbox{}\\
&\quad+\hspace{-1.5em}\sum_{\substack{\beta_1+\beta_2=\beta\\ k_1+k_2=k+1,\; k_1\ge 1\\ l_1+l_2=l}}\hspace{-1.5em}
(-1)^{\iota(\alpha,\gamma;P,I)}
\q_{k_1,l_1}^{\beta_1}(\alpha^{(1:3)}\otimes\q_{k_2,l_2}^{\beta_2}(\alpha^{(2:3)};\gamma^J)\otimes\alpha^{(3:3)};\gamma^I)\Big).
\end{align*}
Dividing by $(-1)^{\varepsilon(\alpha)}$ we get the desired equation.

\end{proof}

\subsection{Proof for \texorpdfstring{$k=-1$}{k= -1}}\label{ssec:q-1}
This section is devoted to the proof of Proposition~\ref{q_-1_rel}. Recall the definition of the projection $\pr$ from~\eqref{eq:varpi}
and recall that $w_{\s} \in H^2(X;\Z/2\Z)$ is the class with $w_2(TL) = i^* w_{\s}$ determined by the relative spin structure $\s$.
\begin{prop}\label{rem:pintgluing} 
Let $l\in\Z_{\ge 0},\,\beta\in\sly$, and $\hat\beta \in H_2(X;\Z)$ with $\pr(\hat\beta) = \beta.$
Let
$
B\subset \d\M_{0,l}(\beta)
$
be the boundary component where a generic point is a sphere of class $\hat\beta$ intersecting $L$ at a marked point. Such spheres arise when the boundary of a disk collapses to a point. Equivalently, one can view this as interior bubbling from a ghost disk component. Note that the ghost disk is not stable.
Then the map
\[
\vartheta:
L\times_X \M_{l+1}(\hat\beta)\stackrel{\sim}{\lrarr}B
\]
satisfies $sgn(\vartheta)=(-1)^{n+1+w_{\s}(\hat\beta)}$.
\end{prop}
\begin{proof}
This is~\cite[Proposition~8.10.6]{FOOO}, but with sign $(-1)^{n+1 + w_\s(\hat\beta)}$ instead of $(-1)^n$. The reason for the sign discrepancy of $+1$ is that in the notation of the proof of~\cite[Proposition~8.10.6]{FOOO}, we should have $\R_{out} = -\R_{>0}.$ The sign is illustrated in Figure~\ref{fig:bbgs} in the case $n = 0$ and $l = 2.$
The reason for the sign discrepancy of $(-1)^{w_\s(\hat \beta)}$ can be seen by following the construction of the orientation associated to a relative spin structure~\cite[Theorem 8.1.1.]{FOOO}.
\end{proof}

\begin{figure}[ht]
\centering
\includegraphics[width=12cm]{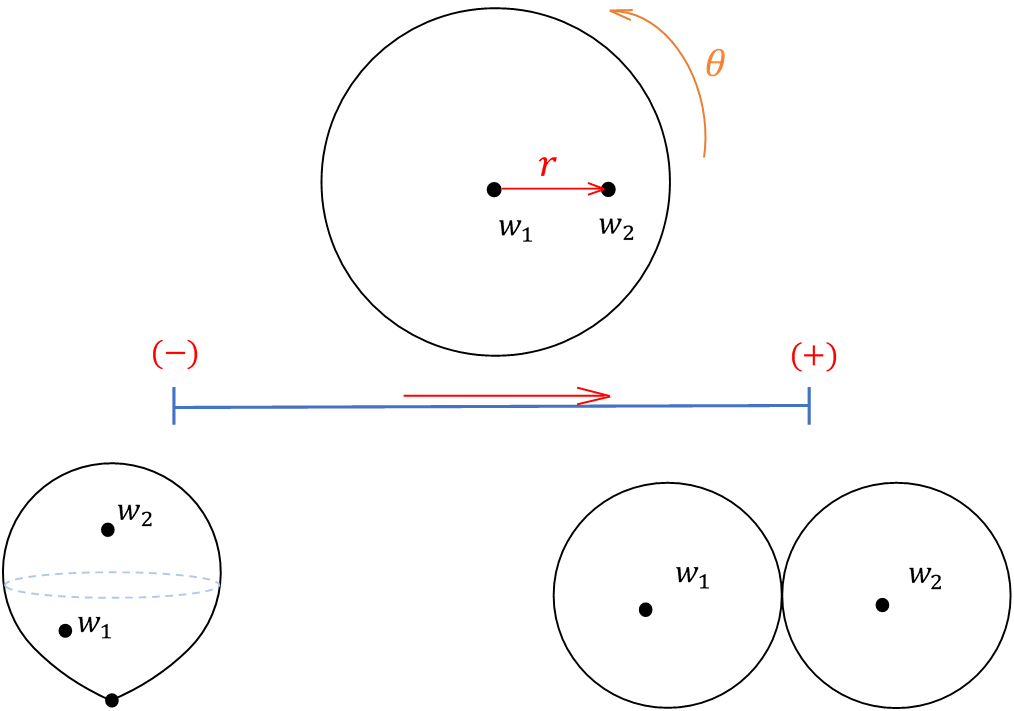}
\caption{$\M_{0,2}$, the moduli space of stable disks with two marked interior points. Here, $X$ and $L$ are a point, so $n = 0.$ Up to reparameterization, we can fix $w_1$ and the $\theta$ coordinate of $w_2$. Then the orientation is given by the positive direction of $r$. The boundary component of a sphere bubble has the sign $(-1)^{n+1}=-1$ in agreement with Proposition~\ref{rem:pintgluing}. The boundary component of two disks joined at a boundary node has the sign $(-1)^n = 1$ in agreement with Proposition~\ref{rem:gluingsign}.}
\label{fig:bbgs}
\end{figure}

Recall that $pt$ is the map from any space to a point, and $evi_j^\beta:\M_{0,l}(\beta)\to X$ is the evaluation map at the $j$th interior marked point.
\begin{lm}\label{lm:starofDtype}
Let $B$ be the boundary component of $\M_{0,l}(\beta)$ described in Proposition~\ref{rem:pintgluing}, and let $\gamma=(\gamma_1,\ldots,\gamma_l)$.
Then
\[
pt_*\big(\bigwedge_{j=1}^l (evi_j^\beta)^*\gamma_j\big)|_B=
(-1)^{n+1}
\int_Li^*\q_{\emptyset,l}^{\hat\beta}(\gamma),
\]
for $\hat\beta$ as in Proposition~\ref{rem:pintgluing}.
\end{lm}

\begin{proof}
Consider the following pullback diagram:
\[
\xymatrix{
{L\times_X\M_{l+1}(\hat\beta)}\ar[r]^{\quad p_2}\ar[d]^{p_1}&
{\M_{l+1}(\hat\beta)}\ar[d]^{ev_0^{\hat \beta}}\\
{L}\ar[r]^{\quad i}&X\;.
}
\]
Recall that $ev_j^{\hat\beta}:\M_{l+1}(\hat\beta)\to X$ is the evaluation map at the $j$th marked point.
Write $\xi:=\bigwedge_{j=1}^l (evi_j^\beta)^*\gamma_j$, and define $\xi', \xi^{\prime\prime},$ by
\[
\xi'=\vartheta^*\xi|_B, \qquad
\xi^{\prime\prime}=\bigwedge_{j=1}^l (ev_j^{\hat\beta})^*\gamma_j,
\]
where $\vartheta$ is the diffeomorphism from Proposition~\ref{rem:pintgluing}.
The result now follows from the fact that
\[
pt_*\xi|_B
=
(-1)^{n+1+w_\s(\hat\beta)}pt_*\xi',
\]
and
\begin{align*}
pt_*\xi'
=
pt_*p_2^*\xi^{\prime\prime}
=
pt_*(p_1)_*p_2^*\xi^{\prime\prime}
=
pt_*i^*(ev_0^{\hat\beta})_*\xi^{\prime\prime}
=
(-1)^{w_\s(\hat\beta)}
\int_Li^*\q_{\emptyset,l}^{\hat\beta}(\gamma).
\end{align*}

\end{proof}

\begin{lm}\label{lm:q-1diskbd}
Let $l\in \Z_{\ge 0},$ $\beta\in \sly$. Let $\beta_1,\beta_2\in\sly$ be such that $\beta_1+\beta_2=\beta$. Let $I\sqcup J=[l]$ be a partition of $[l]$.
Let $B\subset \d \M_{0,l}(\beta)$ be the boundary component
where a generic point is a stable map with two disk components, one carrying the interior marked points labeled by $I$ and the other carrying the points labeled by $J$.
If $\beta_1 \neq \beta_2$ or $I \neq \emptyset$ or $J \neq \emptyset,$ then
\[
\int_{B}\wedge_{j=1}^l evi_j^*\gamma_j
=(-1)^{sgn(\sigma^\gamma_{I\cup J})+|\gamma^J|+n} \langle\q_{0,|I|}^{\beta_1}(\gamma^{I}), \q_{0,|J|}^{\beta_2}(\gamma^{J})\rangle.
\]
If $\beta_1  = \beta_2 = \beta'$ and $I = \emptyset = J,$ then
\[
\int_{B}\wedge_{j=1}^l evi_j^*\gamma_j
= \frac{(-1)^{n}}{2} \langle\q_{0,0}^{\beta'}, \q_{0,0}^{\beta'}\rangle.
\]
\end{lm}
\begin{proof}
First, consider the case $\beta_1 \neq \beta_2$ or $I \neq \emptyset$ or $J \neq \emptyset.$
By Proposition~\ref{rem:gluingsign} applied to the case $i=k_1=k_2=0$, the diffeomorphism
\[
\vartheta:\M_{1,\,I}(\beta_1) \prescript{}{evb_0^{\beta_1}}\times_{evb_0^{\beta_2}}  \M_{1,\,J}(\beta_2)
\stackrel{\sim}{\lrarr} B
\]
has $sgn(\vartheta)=(-1)^{n}$.

Let $evi_j^{\beta_i}$ and $evb_0^{\beta_i}$ for $i=1,2,$ be the evaluation maps of $\M_{1,I}(\beta_1),\M_{1,J}(\beta_2),$ respectively.
Set
\[
\bar\xi:=\vartheta^*\Big(\bigwedge_{j=1}^l evi_j^*\gamma_j\Big),\quad \xi_1:=\bigwedge_{j\in I}(evi_j^{\beta_1})^*\gamma_j, \quad\xi_2:=\bigwedge_{j\in J}(evi_j^{\beta_2})^*\gamma_j.
\]
Similarly to the proof of Lemma~\ref{lm:star}, consider the pull-back diagram
\[
\xymatrix{
{\M_{1,I}(\beta_1)\times_L \M_{1,J}(\beta_2)}\ar[r]^(.6){p_2}\ar[d]^{p_1}&
{\M_{1,J}(\beta_2)}\ar[d]^{evb_0^{\beta_2}}\\
{\M_{1,I}(\beta_1)}\ar[r]^{evb_0^{\beta_1}}&L.
}
\]
By properties~\eqref{prop:pushpull}-\eqref{prop:pushfiberprod} and Lemma~\ref{lm:deg_q}, we compute
\begin{align*}
\int_{\M_{1,I}(\beta_1)\times_L \M_{1,J}(\beta_2)}\bar\xi&=
pt_*(\bar\xi)\\
=&(-1)^{sgn(\sigma^\gamma_{I\cup J})}pt_*{p_1}_*(p_1^*\xi_1 \wedge p_2^*\xi_2)\\
=&(-1)^{sgn(\sigma^\gamma_{I\cup J})}pt_*(\xi_1 \wedge {p_1}_*p_2^*\xi_2)\\
=&(-1)^{sgn(\sigma^\gamma_{I\cup J})}pt_*(\xi_1 \wedge
(evb_0^{\beta_1})^*(evb_0^{\beta_2})_*\xi_2)\\
=&(-1)^{sgn(\sigma^\gamma_{I\cup J})}pt_*(evb_0^{\beta_1})_*(\xi_1 \wedge
(evb_0^{\beta_1})^*(evb_0^{\beta_2})_*\xi_2)\\
=&(-1)^{sgn(\sigma^\gamma_{I\cup J})+|\xi_1||(evb_0^{\beta_2})_*\xi_2|}
pt_*(evb_0^{\beta_1})_*(
(evb_0^{\beta_1})^*(evb_0^{\beta_2})_*\xi_2 \wedge \xi_1)\\
=&(-1)^{sgn(\sigma^\gamma_{I\cup J})+|\xi_1||(evb_0^{\beta_2})_*\xi_2|}
pt_*(
((evb_0^{\beta_2})_*\xi_2 \wedge (evb_0^{\beta_1})_*\xi_1)\\
=&(-1)^{sgn(\sigma^\gamma_{I\cup J})+ |(evb_0^{\beta_2})_*\xi_2|(|\xi_1|+|(evb_0^{\beta_1})_*\xi_1|)}
pt_*
((evb_0^{\beta_1})_*\xi_1\wedge (evb_0^{\beta_2})_*\xi_2)\\
=&(-1)^{sgn(\sigma^\gamma_{I\cup J})+ |(evb_0^{\beta_2})_*\xi_2|(|\xi_1|+|\xi_1|)}
pt_*
((evb_0^{\beta_1})_*\xi_1\wedge (evb_0^{\beta_2})_*\xi_2)\\
=&(-1)^{sgn(\sigma^\gamma_{I\cup J})}
pt_*((evb_0^{\beta_1})_*\xi_1\wedge (evb_0^{\beta_2})_*\xi_2)\\
=&(-1)^{sgn(\sigma^\gamma_{I\cup J})+|\q_{0,l_2}^{\beta_2}(\gamma^{J})| +\varepsilon(\emptyset)+\varepsilon(\emptyset)} \langle\q_{0,l_1}^{\beta_1}(\gamma^{I}), \q_{0,l_2}^{\beta_2}(\gamma^{J})\rangle\\
=&(-1)^{sgn(\sigma^\gamma_{I\cup J})+|\gamma^J|}
\langle\q_{0,l_1}^{\beta_1}(\gamma^{I}), \q_{0,l_2}^{\beta_2}(\gamma^{J})\rangle.
\end{align*}
This proves the lemma in the case $\beta_1 \neq \beta_2$ or $I \neq \emptyset$ or $J \neq \emptyset.$
In the case $\beta_1 = \beta_2 = \beta'$ and $I = \emptyset = J,$ the same argument applies except we must divide the final result by $2$ because the map $\vartheta$ of Proposition~\ref{rem:gluingsign} is $2$ to $1.$
\end{proof}

\begin{proof}[Proof of Proposition~\ref{q_-1_rel}]
Stokes' theorem, Proposition~\ref{stokes}, gives
\begin{equation}\label{eq:stokes-1}
0=\int_{\M_{0,l}(\beta)}d(\wedge_{j=1}^l evi_j^*\gamma_j)+(-1)^{n+1+|\gamma|}\int_{\d\M_{0,l}(\beta)}\wedge_{j=1}^l evi_j^*\gamma_j.
\end{equation}
We have
\begin{equation}\label{eq:dterm-1}
\int_{\M_{0,l}(\beta)}d(\wedge_{j=1}^l evi_j^*\gamma_j)=
\sum_{(2:3)=\{j\}}(-1)^{|\gamma^{(1:3)}|} \q_{-1,l}^\beta(\gamma^{(1:3)}\otimes d\gamma_j\otimes\gamma^{(3:3)}).
\end{equation}
The expression $\int_{\d\M_{0,l}(\beta)}\wedge_{j=1}^l evi_j^*\gamma_j$ consists of two types of contributions.

\textit{First type -- disk bubbling.}
Let $B\subset \d\M_{0,l}(\beta)$ be a boundary component
of the type described in Lemma~\ref{lm:q-1diskbd}. By Lemma~\ref{lm:q-1diskbd}, we have
\begin{equation*}
\int_{B}\wedge_{j=1}^l evi_j^*\gamma_j
=
\begin{cases}
(-1)^{sgn(\sigma^\gamma_{I\cup J})+|\gamma^J|+n} \langle\q_{0,l_1}^{\beta_1}(\gamma^{I}), \q_{0,l_2}^{\beta_2}(\gamma^{J})\rangle, & \beta_1 \neq \beta_2 \text{ or } I \neq \emptyset \text{ or } J \neq \emptyset, \\
\frac{(-1)^{n}}{2} \langle\q_{0,0}^{\beta'}, \q_{0,0}^{\beta'}\rangle, & \beta_1 = \beta_2 = \beta' \text{ and } I = \emptyset = J,
\end{cases}
\end{equation*}
so
\begin{multline}\label{eq:type1bub}
(-1)^{n+1+|\gamma|}\int_{B}\wedge_{j=1}^l evi_j^*\gamma_j
= \\
=
\begin{cases}
(-1)^{sgn(\sigma^\gamma_{I\cup J})+|\gamma^I|+1} \langle\q_{0,l_1}^{\beta_1}(\gamma^{I}), \q_{0,l_2}^{\beta_2}(\gamma^{J})\rangle & \beta_1 \neq \beta_2 \text{ or } I \neq \emptyset \text{ or } J \neq \emptyset, \\
-\frac{1}{2} \langle\q_{0,0}^{\beta'}, \q_{0,0}^{\beta'}\rangle, & \beta_1 = \beta_2 = \beta' \text{ and } I = \emptyset = J.
\end{cases}
\end{multline}

\textit{Second type -- sphere bubbling from a ghost disk.}
Let $B\subset \d\Mt_{0,l}(\beta)$ be a boundary component
of the type described in Proposition~\ref{rem:pintgluing}.
Lemma~\ref{lm:starofDtype} gives
\[
\int_B \wedge_{j=1}^l evi_j^*\gamma_j
=(-1)^{n+1}\int_L i^*\q_{\emptyset,l}^{\hat\beta}(\gamma),
\]
so
\begin{equation}\label{eq:type2bub}
(-1)^{n+1+|\gamma|}\int_B \wedge_{j=1}^l evi_j^*\gamma_j
=(-1)^{|\gamma|}\int_L i^*\q_{\emptyset,l}^{\hat\beta}(\gamma).
\end{equation}

Substituting equations~\eqref{eq:dterm-1},~\eqref{eq:type1bub}, and~\eqref{eq:type2bub} into equation~\eqref{eq:stokes-1} and dividing by $-1$, we get
\begin{align*}
0=&\sum_{(2:3)=\{j\}}(-1)^{|\gamma^{(1:3)}|+1} \q_{-1,l}^\beta(\gamma^{(1:3)}\otimes d\gamma_j\otimes\gamma^{(3:3)})+\\
&+\frac{1}{2}\sum_{\substack{\beta_1+\beta_2=\beta\\I\sqcup J=[l]}}
(-1)^{sgn(\sigma^\gamma_{I\cup J})+|\gamma^I|}
\langle\q_{0,|I|}^{\beta_1}(\gamma^I), \q_{0,|J|}^{\beta_2}(\gamma^J)\rangle
+(-1)^{|\gamma|+1}\sum_{\varpi(\hat\beta)=\beta}\int_L i^* \q_{\emptyset,l}^{\hat\beta}(\gamma),
\end{align*}
with $\varpi$ as in~\eqref{eq:varpi}.
The factor of $1/2$ in the formula arises as follows. Each summand with $\beta_1 \neq \beta_2$ and $I \neq \emptyset \neq J$ appears twice while the corresponding boundary component appears only once. The factor of $1/2$ cancels this discrepancy. The summands with $\beta_1 = \beta_2 = \beta'$ and $I = \emptyset = J$ appear only once, but the contribution of the corresponding boundary component in equation~\eqref{eq:type1bub} comes with a factor of $1/2.$
\end{proof}

\section{Properties}\label{ssec:properties}
\subsection{Linearity}

\begin{cl}\label{lm:qlinear}
The $\q$ operators are multilinear, in the sense that for $a \in R$ we have
\begin{multline*}
\qquad\q_{k,l}^\beta(\a_1,\ldots,\a_{i-1},a\cdot\a_i,\ldots,\a_k;\gamma_1,\ldots,\gamma_l)=\\
		=(-1)^{|a|\cdot\big(i+\sum_{j=1}^{i-1}|\a_j|+\sum_{j=1}^l|\gamma_j|\big)}
		a\cdot\q_{k,l}^\beta(\a_1,\ldots,\a_k;\gamma_1,\ldots,\gamma_l),
\end{multline*}
and for $a \in Q$ we have
\[
\q_{k,l}^\beta(\a_1,\ldots,\a_k; \gamma_1,\ldots,a\cdot\gamma_i,\ldots,\gamma_l) =(-1)^{|a|\cdot\sum_{j=1}^{i-1}|\gamma_j|}
		a\cdot\q_{k,l}^\beta(\a_1,\ldots,\a_k;\gamma_1,\ldots,\gamma_l),
\]
and
\[
\q^\beta_{\emptyset,l}(\gamma_1,\ldots,a\cdot\gamma_i,\ldots,\gamma_l)=
(-1)^{|a|\cdot\sum_{j=1}^{i-1}|\gamma_j|}a\cdot\q^\beta_{\emptyset,l}(\gamma_1,\ldots,\gamma_l).
\]
In addition, the pairing $\langle\;,\,\rangle$ defined by~\eqref{eq:pairing} is $R$-bilinear in the sense of Definition~\ref{dfn:cycunit}\eqref{it:plin}.
\end{cl}
\begin{proof}
For $\q_{1,0}^{\beta_0}=d,$ we have
\[
d(a\a)=(-1)^{|a|}a d\a.
\]
For $(k,l,\beta)\ne (1,0,\beta_0)$, we have
\begin{multline*}
(evb_0)_*(\bigwedge_{j=1}^levi_j^*\gamma_j\wedge\bigwedge_{j=1}^{i-1}evb_j^*\a_j\wedge evb_i^*(a\a_i)\wedge\bigwedge_{j=i+1}^k evb_j^*\a_j)=\\
=
(-1)^{|a|(\sum_{j=1}^l|\gamma_j|+\sum_{j=1}^{i-1}|\a_j|)}
a\cdot(evb_0)_*(\bigwedge_{j=1}^levi_j^*\gamma_j\wedge\bigwedge_{j=1}^{k}evb_j^*\a_j).
\end{multline*}
The corresponding change in $\varepsilon$ is
\[
\varepsilon(\a_1,\ldots,\a_{i-1},a\a_i,\a_{i+1},\ldots ,\a_k)
-\varepsilon(\a_1,\ldots,\a_k)
=i\cdot |a|.
\]
Together, this gives the sign of the first identity.
Similarly, for the second identity,
\begin{multline*}
(evb_0)_*(\bigwedge_{j=1}^{i-1}evi_j^*\gamma_j\wedge evi_i^*(a\gamma_i)\wedge\bigwedge_{j=i+1}^l evi_j^*\gamma_j\wedge\bigwedge_{j=1}^k evb_j^*\a_j)=\\
=
(-1)^{|a|\cdot\sum_{j=1}^{i-1}|\gamma_j|}
a\cdot(evb_0)_*(\bigwedge_{j=1}^levi_j^*\gamma_j\wedge\bigwedge_{j=1}^{k}evb_j^*\a_j),
\end{multline*}
while $\varepsilon$ is not affected. If $k=-1,$ we use $pt$ instead of $evb_0,$ and the sign computation is  valid as before.

The third equation is immediate from definition.

To verify the linearity of the pairing, compute
\begin{gather*}
\langle a\cdot\xi,\eta\rangle=(-1)^{|\eta|}\;a\cdot\int_L\xi\wedge\eta
=a\cdot \langle\xi,\eta\rangle,\\
\langle\xi,a\cdot\eta\rangle=(-1)^{|a|+|\eta|}\int_L\xi\wedge a\cdot\eta
=(-1)^{|a|+|\eta|+|a|\cdot|\xi|}\;a\cdot\int_L\xi\wedge\eta
=(-1)^{|a|(1+|\xi|)}a\cdot\langle\xi,\eta\rangle.
\end{gather*}

\end{proof}

\subsection{Unit of the algebra}
We show that the constant form $1\in A^*(L;R)$ is a unit of the $A_{\infty}$ algebra $(C,\{\mg_k\}_{k\geq0}).$
\begin{cl}\label{cl:unit}
Fix $f\in A^0(L)\otimes R$, $\alpha_1,\ldots,\alpha_{k}\in C,$ and $\gamma_1,\ldots,\gamma_l \in A^*(X;Q).$ Then
\[
\q_{\,k,l}^{\beta}(\alpha_1,\ldots,\alpha_{i-1},f, \alpha_{i},\ldots,\alpha_{k-1} ;\otimes_{r=1}^l\gamma_r)=
\begin{cases}
df, & (k,\;l,\beta)=(1,0,\beta_0),\\
(-1)^{|f|}f\cdot\alpha_1, & (k,l,\beta)=(2,0,\beta_0),\\
&\hspace{5em} i=1,\\
(-1)^{|\alpha_1|(|f|+1)}f\cdot\alpha_1, & (k,l,\beta)=(2,0,\beta_0),\\
&\hspace{5em} i=2,\\
0,& \text{otherwise.}
\end{cases}
\]
In particular, $1\in A^0(L)$ is a strong unit for the $A_\infty$ operations $\mg$:
\[
\m_{k}^\gamma (\alpha_1,\ldots,\alpha_{i-1},1, \alpha_{i},\ldots,\alpha_{k-1})=
\begin{cases}
0, & k\ge 3 \mbox{ or } k=1,\\
\alpha_1, & k=2,\: i=1,\\
(-1)^{|\alpha_1|}\alpha_1, & k=2,\: i=2.
\end{cases}
\]
\end{cl}
\begin{proof}
The case $(k,l,\beta)=(1,0,\beta_0)$ is true by definition.
We proceed with the proof for other values of $(k,l,\beta)$.

Let $\pi:\M_{k+1,l}(\beta)\to  \M_{k,l}(\beta)$ be the map that forgets the $i$th marked boundary point, shifts the labels of the following boundary points, and stabilizes the resulting map. Thus, the map $\pi$ is defined only when stabilization is possible, that is, when $(k,l,\beta)\neq (2,0,\beta_0)$. Denote by
$evb_j^{k+1}$ and $evi_j^{k+1}$ (resp. $evb_j^{k}$ and $evi_j^{k}$) the evaluation maps for $\M_{k+1,l}(\beta)$ (resp. $\M_{k,l}(\beta)$). Set
\[
\xi:= \bigwedge_{j=1}^l (evi_j^{k})^*\gamma_j\wedge \bigwedge_{j=1}^{k-1}(evb_j^{k})^*\alpha_j.
\]
Note that
\[
\quad evi_j^{k+1}=evi_j^{k}\circ \pi,\quad
evb_j^{k+1}=
\begin{cases}
evb_j^{k}\circ \pi, &j<i,\\
evb_{j-1}^{k}\circ\pi, &j>i.
\end{cases}
\]
Thus, writing $g:=(evb_i^{k+1})^*f,$ we have
\begin{equation*}
\pm\q_{\,k,l}^\beta (\alpha_1,\ldots,\alpha_{i-1},f,\alpha_i,\ldots,\alpha_{k-1}; \otimes_{r=1}^l\gamma_r)=(evb_0^{k+1})_*(\pi^*\xi\wedge g)
\end{equation*}
whenever $\pi$ is defined. Using the map $\varphi$ from forms to currents given in Section~\ref{sssec:currents} together with the analog of the integration properties of Proposition~\ref{prop:proppp} for currents given in~\cite[Proposition 6.1]{ST4}, we obtain
\begin{multline}
\varphi\left((evb_0^{k+1})_*(\pi^*\xi\wedge g)\right)  = \\ =(evb_0^{k+1})_*(\pi^*\xi\wedge \varphi(g)) =(evb_0^{k})_*(\pi_*(\pi^*\xi\wedge \varphi(g)))
=(evb_0^{k})_*(\xi\wedge\pi_*\varphi(g)) \label{eq:unit}.
\end{multline}
Since $\dim \M_{k+1,l}(\beta)>\dim \M_{k,l}(\beta)$ and $g$ has degree zero, it follows that $\pi_*\varphi(g) = 0$, and the right-hand side of equation~\eqref{eq:unit} vanishes. Since $\varphi$ is injective, the desired vanishing result for $\q_{k,l}^\beta$ follows. The reason for using currents in this proof is that $\pi$ need not be a submersion, so the push-forward $\pi_*$ is not defined on differential forms.

Let us see what happens when $(k,l,\beta)=(2,0,\beta_0).$ In that case, the evaluation maps on $\M_{3,0}(\beta_0)$ satisfy $evb_0=evb_1=evb_2.$ So,
\begin{gather*}
\q_{2,0}^{\beta_0}(f,\alpha)
=(-1)^{|f|+1+2(|\alpha|+1)+1}(evb_0)_*evb_0^*(f\wedge\alpha),\\
\q_{2,0}^{\beta_0}(\alpha, f)
=(-1)^{|\alpha|+1+2\cdot (|f|+1)+1}(evb_0)_*evb_0^*(\alpha\wedge f).
\end{gather*}
Denote by $\M_{k+1,l}$ the moduli space of stable disks, that is, genus zero open stable maps to a point, with $k+1$ boundary marked points and $l$ interior marked points.
Since $\beta=\beta_0,$ the evaluation map $evb_0$ induces an identification of $\M_{k+1,l}(\beta_0)$ with $\M_{k+1,l}\times L$. Since $k+1 = 3$ and $l = 0,$ the space of stable disks is a point. Hence, $evb_0$ identifies $\M_{3,0}(\beta_0)$ diffeomorphically with $L$. This diffeomorphism preserves orientation by the argument on page~714 of~\cite{FOOO} based on their Convention~8.3.1. Thus,
\[
\q_{2,0}^{\beta_0}(f,\alpha)
=(-1)^{|f|}f\alpha
\quad\text{and}\quad
\q_{2,0}^{\beta_0}(\alpha, f)
=(-1)^{|\alpha|+|\a||f|}f\alpha.
\]
\end{proof}

\subsection{Cyclic structure}
Recall the definition of the pairing~\eqref{eq:pairing}.
Note that
\begin{equation}\label{eq:psgn}
\langle\xi,\eta\rangle:=(-1)^{|\eta|}\int_L\xi\wedge\eta
=(-1)^{|\eta|+|\eta|\cdot|\xi|}\int_L\eta\wedge\xi
=(-1)^{(|\eta|+1)(|\xi|+1)+1}\langle\eta,\xi\rangle.
\end{equation}

\begin{cl}\label{cl:cyclic}
For any $\alpha_1,\ldots,\alpha_{k+1}\in C$ and $\gamma_1,\ldots,\gamma_l\in A^*(X;Q)$,
\begin{align*}
\langle\qkl(\alpha_1,\ldots&,\alpha_k;\gamma_1,\ldots\gamma_l),\alpha_{k+1}\rangle=\\
&(-1)^{(|\alpha_{k+1}|+1)\sum_{j=1}^{k}(|\alpha_j|+1)}\cdot
\langle \qkl(\alpha_{k+1},\alpha_1,\ldots,\alpha_{k-1};\gamma_1,\ldots,\gamma_l),\alpha_k\rangle.
\end{align*}
In particular, $(C,\{\m_k^\gamma\}_{k\ge 0})$ is a cyclic $A_\infty$ algebra for any $\gamma$.
\end{cl}

\begin{proof}
For $(k,l,\beta)\ne (1,0,\beta_0),$ use Lemma~\ref{lm:deg_q} to compute
\begin{align}\label{eq:cs1}
\langle\qkl^\beta&(\alpha_1,\ldots,\alpha_k;\gamma_1,\ldots\gamma_l),\alpha_{k+1}\rangle=\notag\\
=&(-1)^{|\alpha_{k+1}|}pt_*(\qkl^\beta(\alpha_1,\ldots,\alpha_k;\gamma_1,\ldots\gamma_l)\wedge\alpha_{k+1})\notag\\
=&(-1)^{|\alpha_{k+1}|+\varepsilon(\alpha)}pt_*((evb_0)_*(\bigwedge_{j=1}^levi_j^*\gamma_j\wedge\bigwedge_{j=1}^k evb_j^*\alpha_j)\wedge\alpha_{k+1})\notag\\
=&(-1)^{|\alpha_{k+1}|+\varepsilon(\alpha)+|\alpha_{k+1}|\cdot\left(\sum_{j=1}^k|\alpha_j|+|\gamma|+k\right)}
pt_*(\alpha_{k+1}\wedge(evb_0)_*(\bigwedge_{j=1}^levi_j^*\gamma_j\wedge \bigwedge_{j=1}^k evb_j^*\alpha_j))\notag\\
=&(-1)^{|\alpha_{k+1}|+\varepsilon(\alpha)+|\alpha_{k+1}| \cdot\left(\sum_{j=1}^k|\alpha_j|+|\gamma|+k\right)}
pt_*(evb_0)_*(evb_0^*\alpha_{k+1}\wedge\bigwedge_{j=1}^levi_j^*\gamma_j\wedge\bigwedge_{j=1}^k evb_j^*\alpha_j)\notag\\
=&(-1)^{|\alpha_{k+1}|+\varepsilon(\alpha)+|\alpha_{k+1}| \cdot\left(\sum_{j=1}^k|\alpha_j|+|\gamma|+k\right)+|\alpha_k| \cdot\left(|\alpha_{k+1}|+|\gamma|+\sum_{j=1}^{k-1}|\alpha_j|\right) +|\alpha_{k+1}||\gamma|}\;\cdot\notag\\
&\quad\cdot(pt\circ evb_0)_*(evb_k^*\alpha_k\wedge \bigwedge_{j=1}^levi_j^*\gamma_j\wedge evb_0^*\alpha_{k+1}\wedge\bigwedge_{j=1}^{k-1} evb_j^*\alpha_j)\notag\\
=&(-1)^{|\alpha_{k+1}|+\varepsilon(\alpha)+|\alpha_{k+1}| \cdot\left(\sum_{j=1}^k|\alpha_j|+k\right)+|\alpha_k| \cdot\left(|\alpha_{k+1}|+|\gamma|+\sum_{j=1}^{k-1}|\alpha_j|\right)} \;\cdot\notag\\
&\quad\cdot(pt\circ evb_k)_*(evb_k^*\alpha_k\wedge \bigwedge_{j=1}^levi_j^*\gamma_j\wedge evb_0^*\alpha_{k+1}\wedge\bigwedge_{j=1}^{k-1} evb_j^*\alpha_j)\notag\\
=&(-1)^{|\alpha_{k+1}|+\varepsilon(\alpha)+|\alpha_{k+1}| \cdot\left(\sum_{j=1}^k|\alpha_j|+k\right)+|\alpha_k| \cdot\left(|\alpha_{k+1}|+|\gamma|+\sum_{j=1}^{k-1}|\alpha_j|\right)} \;\cdot\notag\\
&\quad\cdot pt_*(\alpha_k\wedge {evb_k}_*(\bigwedge_{j=1}^levi_j^*\gamma_j\wedge evb_0^*\alpha_{k+1}\wedge\bigwedge_{j=1}^{k-1} evb_j^*\alpha_j))\notag\\
=&(-1)^{|\alpha_{k+1}|+\varepsilon(\alpha)+|\alpha_{k+1}| \cdot\left(\sum_{j=1}^k|\alpha_j|+k\right)+|\alpha_k| \cdot\left(|\alpha_{k+1}|+|\gamma| +\sum_{j=1}^{k-1}|\alpha_j|\right)+|\alpha_k| \cdot\left(|\alpha_{k+1}|+|\gamma|+\sum_{j=1}^{k-1}|\alpha_j|+k\right)} \;\cdot\notag\\
&\quad\cdot pt_*({evb_k}_*( \bigwedge_{j=1}^levi_j^*\gamma_j \wedge evb_0^*\alpha_{k+1}\wedge\bigwedge_{j=1}^{k-1} evb_j^*\alpha_j)\wedge\alpha_k)\notag\\
=&(-1)^{|\alpha_{k+1}|+\varepsilon(\alpha)+|\alpha_{k+1}| \cdot\left(\sum_{j=1}^k|\alpha_j|+k\right)+k\cdot|\alpha_k|} \;\cdot\notag\\
&\quad\cdot pt_*({evb_k}_*(\bigwedge_{j=1}^levi_j^*\gamma_j \wedge evb_0^*\alpha_{k+1}\wedge\bigwedge_{j=1}^{k-1} evb_j^*\alpha_j)\wedge\alpha_k).
\end{align}
Let $\varphi : \M_{k+1,l}(\beta) \to \M_{k+1,l}(\beta)$ be given by
\[
\varphi(\Sigma,u,(z_0,\ldots,z_{k}),\vec w) = (\Sigma,u,(z_1,\ldots,z_k,z_0),\vec w).
\]
So,
\[
evi_j \circ \varphi = evi_j, \qquad evb_k \circ \varphi = evb_0, \qquad evb_j \circ \varphi = evb_{j+1}, \quad j = 0,\ldots,k-1,
\]
and $sgn(\varphi) = k.$ Thus, property~\eqref{prop:pushpull} of integration gives
\begin{align}\label{eq:cs2}
pt_*({evb_k}_*&(\bigwedge_{j=1}^levi_j^*\gamma_j \wedge evb_0^*\alpha_{k+1}\wedge\bigwedge_{j=1}^{k-1} evb_j^*\alpha_j)\wedge\alpha_k) = \notag\\
=&
(-1)^k pt_*({evb_k}_*\varphi_*\varphi^*(\bigwedge_{j=1}^levi_j^*\gamma_j \wedge evb_0^*\alpha_{k+1}\wedge\bigwedge_{j=1}^{k-1} evb_j^*\alpha_j)\wedge\alpha_k) \notag\\
=&
(-1)^k pt_*({evb_0}_*(\bigwedge_{j=1}^levi_j^*\gamma_j \wedge evb_1^*\alpha_{k+1}\wedge\bigwedge_{j=1}^{k-1} evb_{j+1}^*\alpha_j)\wedge\alpha_k)\notag\\
=&
(-1)^{\varepsilon(\a_{k+1},\a_1,\ldots,\a_{k-1}) + k + |\alpha_k|}
\langle\qkl^\beta(\alpha_{k+1},\alpha_1,\ldots,\alpha_{k-1}; \gamma_1,\ldots,\gamma_l), \alpha_k\rangle.
\end{align}
Combining~\eqref{eq:cs1} and~\eqref{eq:cs2}, we obtain
\[
\langle\qkl^\beta(\alpha_1,\ldots,\alpha_k;\gamma_1,\ldots\gamma_l),\alpha_{k+1}\rangle=(-1)^*
\langle\qkl^\beta(\alpha_{k+1},\alpha_1,\ldots,\alpha_{k-1}; \gamma_1,\ldots,\gamma_l), \alpha_k\rangle,
\]
where
\begin{align*}
*
=&|\alpha_{k+1}|+\sum_{j=1}^{k}j(|\alpha_j| + 1)+1 + |\alpha_{k+1}| \cdot\Big(\sum_{j=1}^k|\alpha_j|+k\Big) +k\cdot|\alpha_k| +\\
&+1\cdot(|\alpha_{k+1}|+1)+ \sum_{j=1}^{k-1}(j+1)(|\alpha_j|+1)+1+k+|\alpha_k|\\
=&\sum_{j=1}^{k-1}(|\alpha_j|+1)+ k(|\alpha_k|+1)+ |\alpha_{k+1}|\cdot\Big(\sum_{j=1}^k|\alpha_j|+k\Big) +k\cdot|\alpha_k|+1+k+|\alpha_k|\\
=&\sum_{j=1}^{k-1}(|\alpha_j|+1) +|\alpha_{k+1}|\cdot\sum_{j=1}^k(|\alpha_j|+1)+1+|\alpha_k|\\
=&(|\alpha_{k+1}|+1)\sum_{j=1}^{k}(|\alpha_j|+1).
\end{align*}

It remains to verify that $d$ is also cyclic. Indeed,
\begin{align*}
\langle d\alpha_1,\alpha_2\rangle=&
(-1)^{|\alpha_2|}\int_L d\alpha_1\wedge\alpha_2
=\int_L \left((-1)^{|\alpha_2|}d(\alpha_1\wedge\alpha_2)+(-1)^{|\alpha_2|+|\alpha_1|+1}\alpha_1\wedge d\alpha_2\right)\\
=&(-1)^{|\alpha_2|+|\alpha_1|+1+|\alpha_1|(|\alpha_2|+1)}\int_L d\alpha_2\wedge\alpha_1
=(-1)^{|\alpha_2|+1+|\alpha_1|(|\alpha_2|+1)}\langle d\alpha_2,\alpha_1\rangle\\
=&(-1)^{(|\alpha_1|+1)(|\alpha_2|+1)}\langle d\alpha_2,\alpha_1\rangle.
\end{align*}

\end{proof}

\begin{rem}
Intuitively, pairing $\qkl$ with $\alpha_{k+1}$ should be viewed as putting the constraint $\alpha_{k+1}$ on $z_0$. The cyclic property then translates to a symmetry under cyclic relabeling of the boundary marked points.
\end{rem}

\subsection{Degree of structure maps}
\begin{cl}\label{deg_str_map}
For $k\ge 0$ and $\gamma_1,\ldots,\gamma_l\in A^*(X;Q)$  with $|\gamma_j|=2$,
the map
\[
\qkl(\; ;\gamma_1,\ldots,\gamma_l):C^{\otimes k}\lrarr C
\]
is of degree $2-k$.
\end{cl}

\begin{proof}
It is enough to check that, for any $\beta$, the map
\[
T^{\beta}\qkl^\beta(\; ;\gamma_1,\ldots,\gamma_l):C^{\otimes k}\lrarr C
\]
is of degree $2-k$. Indeed,
\begin{align*}
|T^{\beta}\qkl^{\beta}(\alpha_1,\ldots,\alpha_k;\gamma_1,\ldots,\gamma_l)|
=& \mu(\beta)+\sum_{j=1}^k|\alpha_j|+2l-\rdim(evb_0)\\
=& \mu(\beta)+\sum_{j=1}^k|\alpha_j|+2l-(n-3+\mu(\beta)+k+1+2l-n)\\
=& \sum_{j=1}^k|\alpha_j|+2-k.
\end{align*}
The special case $\q_{1,0}^{\beta_0}=d$ also aligns with the above formula, being of degree $1=2-1$.

\end{proof}

\subsection{Symmetry}
\begin{cl}\label{cl:symmetry}
Let $k\ge -1$. For any permutation $\sigma\in S_l,$
\[
\qkl(\alpha_1,\ldots,\alpha_k;\gamma_1,\ldots,\gamma_l)=
(-1)^{s_\sigma(\gamma)}\qkl(\alpha_1,\ldots,\alpha_k;\gamma_{\sigma(1)},\ldots,\gamma_{\sigma(l)}),
\]
where
\begin{equation}\label{eq:sgnsigmagamma}
s_\sigma(\gamma):=
\sum_{\substack{i<j\\ \sigma^{-1}(i)>\sigma^{-1}(j)}}|\gamma_i|\cdot|\gamma_j|
=
\sum_{\substack{i>j\\ \sigma(i)<\sigma(j)}}|\gamma_{\sigma(i)}|\cdot|\gamma_{\sigma(j)}|\pmod 2.
\end{equation}
\end{cl}
\begin{proof}
First note that $\varepsilon(\alpha)$ is independent from $\gamma$ and thus is not influenced by applying $\sigma$ to $\gamma$. Besides, changing the labeling of interior marked points does not affect the orientation of the moduli space. So, for $k\ge 0,$
\begin{align*}
\qkl^\beta(\alpha_1,\ldots,\alpha_k;&\gamma_1,\ldots,\gamma_l)=
(-1)^{\varepsilon(\alpha)}
(evb_0^\beta)_* \left(\bigwedge_{j=1}^l(evi_j^\beta)^*\gamma_j\wedge\bigwedge_{j=1}^k (evb_j^\beta)^*\alpha_j\right)\\
=&(-1)^{\varepsilon(\alpha)+s_\sigma(\gamma)}
(evb_0^\beta)_* \left(\bigwedge_{j=1}^l(evi_{\sigma(j)}^\beta)^*\gamma_{\sigma(j)}\wedge\bigwedge_{j=1}^k (evb_j^\beta)^*\alpha_j\right)\\
=&(-1)^{s_\sigma(\gamma)} \qkl^\beta(\alpha_1,\ldots,\alpha_k;\gamma_{\sigma(1)},\ldots,\gamma_{\sigma(l)}).
\end{align*}
The case $k=-1$ is similar, with $pt$ instead of $evb_0^\beta$ and without $\varepsilon(\a)$.

\end{proof}

\subsection{Fundamental class}
\begin{cl}\label{q_fund}
For $k\ge 0,$
\[
\qkl^\beta(\alpha_1,\ldots,\alpha_k;1,\gamma_1,\ldots,\gamma_{l-1})=
\begin{cases}
-1, & (k,l,\beta)=(0,1,\beta_0),\\
0, & \text{otherwise}.
\end{cases}
\]
Furthermore,
\[
\q_{-1,l}^\beta(1,\gamma_1,\ldots,\gamma_{l-1})=0.
\]
\end{cl}
\begin{proof}
Whenever defined, consider $\pi:\M_{k+1,l}(\beta)\to\M_{k+1,l-1}(\beta),$ the forgetful map that forgets the first interior marked point, shifts the labeling of the rest, and stabilizes the resulting map. Similarly to the proof of Proposition~\ref{cl:unit}, using the notation $\varphi$ from Section~\ref{sssec:currents}, we get
\[
\varphi(\qkl^\beta(\alpha_1,\ldots,\alpha_k;1,\gamma_1,\ldots,\gamma_{l-1}))
=
\pm
(evb_0)_*(\wedge_{j=1}^{l-1}evi_j^*\gamma_j\wedge \wedge_{j=1}^k evb_j^*\a_j\wedge \pi_*\varphi(1))
=0
\]
whenever $\pi$ is defined. So, since $\varphi$ is injective, it follows that
\[
\qkl^\beta(\alpha_1,\ldots,\alpha_k;1,\gamma_1,\ldots,\gamma_{l-1}) = 0.
\]
The forgetful map $\pi$ is not defined only when forgetting the point will result in a non-stabilizable curve. This happens exactly when $\beta=\beta_0$ and $(k,l)\in\{(0,1),(1,1),(-1,2)\}$.

The case $(k,l,\beta)=(1,1,\beta_0)$ is treated as follows.
Since the stable maps in $\M_{2,1}(\beta_0)$ are constant, we have
\[
evb_0=evb_1,\quad evi_1=i\circ evb_0.
\]
So,
\begin{align*}
\q^{\beta_0}_{1,1}(\alpha_1;\gamma_1)
=&(-1)^{|\alpha_1|+1+1}(evb_0)_*evb_0^*( i^*\gamma_1\wedge\alpha_1)=(-1)^{|\alpha_1|}i^*\gamma_1\wedge\alpha_1\wedge (evb_0)_*1.
\end{align*}
But $\rdim(evb_0)=n-3+\mu(\beta_0)+k+1+2l-n>0$, so $(evb_0)_*1=0$.

The case $(k,l,\beta)=(-1,2,\beta_0)$ corresponds to the moduli space $\M_{0,2}(\beta_0).$ Again,
\[
evi_1=evi_2=:ev.
\]
Moreover, there is a unique map $evb : \M_{0,2}(\beta_0) \to L$ such that
\[
ev=i\circ evb.
\]
Thus,
\begin{align*}
\q^{\beta_0}_{-1,2}(\gamma_1,\gamma_2)=& pt_*ev^*(\gamma_1\wedge\gamma_2) \\
=& pt_*ev_*ev^*(\gamma_1\wedge\gamma_2)\\
=& pt_*((\gamma_1\wedge\gamma_2)\wedge ev_*1)\\
=& pt_*((\gamma_1\wedge\gamma_2)\wedge i_*evb_*1).
\end{align*}
But $\rdim(evb)=n-3+\mu(\beta_0)+2l-n>0$, so $evb_*1=0$.

The only case left is $(0,1,\beta_0)$, which corresponds to the moduli space $\M_{1,1}(\beta_0)$. As in the proof of Proposition~\ref{cl:unit}, the evaluation map $evb_0$ identifies the moduli space of maps with $L$, preserving orientation. Using this identification, we see that
\[
\q^{\beta_0}_{0,1}(1)=-(evb_0)_*evb_0^*i^*1=-\Id_*\Id^*1=-1.
\]

\end{proof}

\subsection{Energy zero}
\begin{cl}\label{q_zero}
For $k\ge 0,$
\[
\qkl^{\beta_0}(\alpha_1,\ldots,\alpha_k;\gamma_1,\ldots,\gamma_l)=
\begin{cases}
d\alpha_1, & (k,l)=(1,0),\\
(-1)^{|\alpha_1|}\alpha_1\wedge\alpha_2, & (k,l)=(2,0),\\
-\gamma_1|_L, & (k,l)=(0,1),\\
0, & \text{otherwise}.
\end{cases}
\]
Furthermore,
\[
\q_{-1,l}^{\beta_0}(\gamma_1,\ldots,\gamma_l)=0.
\]
\end{cl}
\begin{proof}
The case $\q_{1,0}^{\beta_0}=d$ is true by definition. Let us consider the cases where $\q$ is defined by push-pull operations.

Since the stable maps in $\M_{k+1,l}(\beta_0)$ are constant, we have
\[
evb_0 = \cdots = evb_k =: evb, \qquad evi_1 = \cdots = evi_l = i \circ evb.
\]
Thus, for $k\ge 0,$
\begin{align*}
\qkl^{\beta_0}(\alpha_1,\ldots,\alpha_k;\gamma_1,\ldots,\gamma_l)=& (-1)^{\varepsilon(\alpha)}evb_*evb^*(\wedge_{j=1}^li^*\gamma_j\wedge\wedge_{j=1}^k\alpha_j)\\
=&(-1)^{\varepsilon(\alpha)}(\wedge_{j=1}^l\gamma_j|_L\wedge\wedge_{j=1}^k\alpha_j)\wedge evb_*1.
\end{align*}
For $k=-1,$
\begin{align*}
\q^{\beta_0}_{-1,l}(\gamma_1,\ldots,\gamma_l)
=& pt_*(i\circ evb)_*(i\circ evb)^*(\wedge_{j=1}^l\gamma_j)\\
=&pt_* ((\wedge_{j=1}^l\gamma_j)\wedge i_*evb_*1).
\end{align*}
In order for $evb_*1$ to be nonzero, we need
\[
0=\rdim(evb)=n-3+\mu(\beta_0)+k+1+2l-n
=k+2l-2.
\]
Let us analyze when this equality is possible.

If $l=1$, then $k=0$,  and $evb:\M_{1,1}(\beta_0)\stackrel{\sim}{\to}L$. This diffeomorphism preserves orientation by the argument on page~739 of~\cite{FOOO}. So, $\q^{\beta_0}_{k,l}(\gamma_1)=-\gamma_1|_L$ by the above computation.

If $l=0$, then $k=2$, and $evb:\M_{3,0}(\beta_0)\stackrel{\sim}{\to}L$. This diffeomorphism preserves orientation by the argument on page~714 of~\cite{FOOO} based on their Convention~8.3.1. So, again by the computation above,
\[
\q^{\beta_0}_{k,l}(\alpha_1,\alpha_2)=(-1)^{|\alpha_1|+1+2(|\alpha_2|+1)+1}\alpha_1\wedge\alpha_2.
\]

\end{proof}

\subsection{Divisors}

\begin{cl}\label{cl:q_div}
Assume $\gamma_1\in A^2(X,L)\otimes Q$, $d\gamma_1 = 0$,
and the map $H_2(X,L;\Z)\to Q$ given by $\beta\mapsto\int_\beta\gamma_1$ descends to $\sly$.
Then
	\begin{equation}
	\qkl^{\beta}(\otimes_{j=1}^k\alpha_j;\otimes_{j=1}^{l}\gamma_j)=
	\left(\int_\beta\gamma_1\right) \cdot\q_{k,l-1}^{\beta} (\otimes_{j=1}^k\alpha_j;\otimes_{j=2}^{l}\gamma_j)
	\end{equation}
for $k\ge -1$.
	\end{cl}
	
 The proof requires the following two results, which will be proved after the main proposition.
\begin{lm}\label{lm:current_bd}
Let $M$ be a connected oriented orbifold with corners and let $\alpha$ be a degree-$0$ current on $M.$ Suppose there is a current $f$ on $\d M$ such that for any $\eta\in \Ac^{top-1}(M),$
\[
\alpha(d\eta)=f(i_{M}^*\eta),\qquad i_{M}:\d M\lrarr M.
\]
Then there is a constant $\kappa\in \R$ such that
\[
\alpha(\gamma)=\kappa\cdot\int_M\gamma\qquad \forall \gamma\in \Ac^{top}(M).
\]
\end{lm}
In the following, we use the inclusion of forms in currents from Section~\ref{sssec:currents}.
\begin{lm}\label{lm:currentconst}
Suppose either $\beta\neq \beta_0$ or $\beta = \beta_0$ and $(k,l)\notin\{(0,1),(1,1),(-1,2)\}.$
Let $\pi:\M_{k+1,l}(\beta)\to\M_{k+1,l-1}(\beta)$ be the map that forgets the first interior marked point, shifts the labels of the others down by one, and stabilizes the resulting map.
Denote by $evi_1$ the evaluation map at the first interior point for $\M_{k+1,l}(\beta)$.
Let $\gamma\in A^*(X)$ such that $\gamma|_{L}=0,$ $|\gamma|=2$, and $d\gamma = 0$.
Assume the map $H_2(X,L;\Z)\to \R$ given by $\beta\mapsto\int_\beta\gamma$ descends to $\sly$.
Then, the current $\pi_* evi_1^*\gamma$ coincides with the current corresponding to the constant $\int_\beta \gamma.$
\end{lm}
	
\begin{proof}[Proof of Proposition~\ref{cl:q_div}]
When $\beta = \beta_0,$ the proposition follows from Proposition~\ref{q_zero}, so we may assume $\beta \neq \beta_0.$
Denote by $\pi:\M_{k+1,l}(\beta)\to\M_{k+1,l-1}(\beta)$ the forgetful map as in Lemma~\ref{lm:currentconst}.
Denote by $evb_j^l, evi_j^l,$ the evaluation maps for $\M_{k+1,l}(\beta)$, and denote by $evb_j^{l-1},evi_j^{l-1},$ the evaluation maps for $\M_{k+1,l-1}(\beta)$.

For $k\ge 0,$ set $\xi:=\wedge_{j=2}^l (evi^{l-1}_{j-1})^*\gamma_{j}\wedge\wedge_{j=1}^k(evb^{l-1}_j)^*\alpha_j$. Then, we have the equality of currents,
\begin{equation}
\begin{split}\label{eq:split}
\qkl^{\beta}(\alpha_1,\ldots,\alpha_k;\gamma_1,\ldots,\gamma_l)
&=(-1)^{\varepsilon(\alpha)} (evb^l_0)_*((evi^l_1)^*\gamma_1\wedge\pi^*\xi)\\
=&(-1)^{\varepsilon(\alpha)+|\xi|\cdot|\gamma_1|}(evb^{l-1}_0)_*\pi_*(\pi^*\xi\wedge (evi^l_1)^*\gamma_1)\\
=&(-1)^{\varepsilon(\alpha)+|\xi|\cdot|\gamma_1|}(evb^{l-1}_0)_*(\xi\wedge \pi_*(evi^l_1)^*\gamma_1)\\
=&
(-1)^{\varepsilon(\alpha)+|\xi|\cdot\rdim \pi}(evb^{l-1}_0)_*(\pi_*(evi^l_1)^*\gamma_1 \wedge\xi)\\
=&
(-1)^{\varepsilon(\alpha)}(evb^{l-1}_0)_*(\pi_*(evi^l_1)^*\gamma_1 \wedge\xi).
\end{split}
\end{equation}
Similarly, for $k=-1,$ set $\xi:=\wedge_{j=2}^l(evi^{l-1}_{j-1})^*\gamma_j$ and compute
\begin{equation}
\begin{split}
\q^\beta_{-1,l}(\gamma_1,\ldots,\gamma_l)
=&
pt_*(\pi_*(evi^l_1)^*\gamma_1 \wedge\xi).
\end{split}
\label{eq:split2}
\end{equation}
By Lemma~\ref{lm:currentconst}, $\pi_*(evi_1^l)^*\gamma_1$ is the current corresponding to the constant $\int_\beta \gamma_1.$
Substituting this value in~\eqref{eq:split}, we get
\begin{multline*}
\qkl^{\beta}(\alpha_1,\ldots,\alpha_k;\gamma_1,\ldots,\gamma_l)
=(-1)^{\varepsilon(\alpha)} \int_{\beta}\gamma_1\cdot(evb^{l-1}_0)_*\xi=\\
=\int_{\beta}\gamma_1\cdot\q_{k,l-1}^{\beta}(\alpha_1,\ldots,\alpha_k;\gamma_2,\ldots,\gamma_l).
\end{multline*}
Similarly, substituting $\int_\beta \gamma_1$ in~\eqref{eq:split2}, we get
\[
\q^\beta_{-1,l}(\gamma_1,\ldots,\gamma_l)=
\int_{\beta}\gamma_1\cdot pt_*\xi
=\int_{\beta}\gamma_1\cdot\q^\beta_{-1,l-1}(\gamma_2,\ldots,\gamma_l).
\]
\end{proof}

We return to the proof of the auxiliary lemmas.
\begin{proof}[Proof of Lemma~\ref{lm:current_bd}]
Let $\gamma,\gamma'\in \Ac^{top}(M).$ Then $i_{M}^*\gamma= i_{M}^*\gamma'=0$ for degree reasons.
Assume $[\gamma]=[\gamma']\in H_c^{top}(M,\d M).$ Choose $\zeta\in \Ac^{top-1}(M)$ such that  $i_{M}^*\zeta=0$ and $\gamma-\gamma'=d\zeta.$ Then
\[
\alpha(\gamma)-\alpha(\gamma')
=f(i_M^*\zeta)=0,
\]
so $\alpha(\gamma)=\alpha(\gamma').$ This shows $\alpha(\gamma)$ depends only on the relative cohomology class of $\gamma$.
On the other hand, by Poincar\'e duality, we have an isomorphism $H^{top}_c(M,\d M) \to \R$ given by integration over $M.$

\end{proof}

\begin{proof}[Proof of Lemma~\ref{lm:currentconst}]
Since $\pi$ is not a submersion, the push-forward of a differential form along $\pi$ is not defined as a differential form. Rather, for a differential form $\zeta \in A^*(\M_{k+1,l}(\beta)),$
we abbreviate $\pi_*\zeta$ for the push-forward along $\pi$ of the current corresponding to $\zeta$ as explained in Section~\ref{sssec:currents}.

Decompose the boundary,
\[
\d\M_{k+1,l}(\beta) = \d^{\text{hor}}\M_{k+1,l}(\beta)\coprod \d^{\text{vert}}\M_{k+1,l}(\beta),
\]
where $\d^{\text{hor}}\M_{k+1,l}(\beta)$ is the part of the boundary that does not require stabilization after forgetting $w_1,$ and $\d^{\text{vert}}\M_{k+1,l}(\beta)$ is the part of the boundary that does. Generic points of $\d^{\text{vert}}\M_{k+1,l}(\beta)$ are mapped by $\pi$ to interior points of $\M_{k+1,l-1}(\beta)$, whereas $\d^{\text{hor}}\M_{k+1,l}(\beta)$ is mapped to $\d\M_{k+1,l-1}(\beta)$.
Thus, we have the following commutative diagram:
\[
\xymatrix{
\d^{\text{hor}}\M_{k+1,l}(\beta)\ar@{^{(}->}[r]\ar[d]^{\pi_{\d}}\ar @/^15pt/[rr]^{i_l^{\text{hor}}}&\d\M_{k+1,l}(\beta)\ar[r]_{i_{l}} &\M_{k+1,l}(\beta)\ar[d]^\pi\\
\d\M_{k+1,l-1}(\beta)\ar[rr]^{i_{l-1}}&&\M_{k+1,l-1}(\beta)\:.
}
\]
Take $\gamma\in A^*(X)$ as in the statement of the lemma.
For short, write
$M_1:=\M_{k+1,l}(\beta),$ $M_2:=\M_{k+1,l-1}(\beta),$ and $\zeta:=(evi_1)^*\gamma \in A^*(M_1)$.
By definition, for arbitrary $\eta\in \Ac^{top-1}(\M_{k+1,l-1}(\beta)),$ since $|\zeta|=\rdim \pi=2,$ we have
\begin{multline*}
(\pi_*\zeta)(d\eta)=
(-1)^{\rdim \pi \cdot |d\eta|}
\int_{M_1}\zeta\wedge\pi^*d\eta= \\
=\int_{M_1}\zeta\wedge d(\pi^*\eta)=\int_{M_1}d(\zeta\wedge\pi^*\eta)= \int_{\d M_1}(i_l)^*(\zeta\wedge\pi^*\eta).
\end{multline*}
Note that $\zeta|_{\d^{\text{vert}} M_1}=0$, because the interior marked point $w_1$ is located on a ghost bubble that maps entirely to $L$, and $\gamma|_L=0.$ So, the computation continues
\begin{multline*}
(\pi_*\zeta)(d\eta)
=\int_{\d^{\text{hor}}M_1}(i_l^{\text{hor}})^*(\zeta\wedge\pi^*\eta)
=\int_{\d^{\text{hor}}M_1}(i_l^{\text{hor}})^*\zeta\wedge (i_l^{\text{hor}})^*\pi^*\eta=\\
=\int_{\d^{\text{hor}}M_1}(i_l^{\text{hor}})^*\zeta\wedge \pi_{\d}^*(i_{l-1}^*\eta)=\left((\pi_{\d})_*((i_l^{\text{hor}})^*\zeta)\right)(i_{l-1}^*\eta),
\end{multline*}
where the sign in the last equality is trivial again because $\rdim\pi_{\d}=2$. Note that if $\M_{k+1,l-1}(\beta)$ is not connected, the same computation is valid for each connected component separately.

By Lemma~\ref{lm:current_bd}, for each connected component $B$ of $\M_{k+1,l-1}(\beta)$ there is a constant $\kappa_B$ such that
\[
(\pi_*\zeta)(\eta)=\kappa_B\cdot\int\limits_{B}\eta,\qquad \forall\eta\in \Ac^{top}(B).
\]
To compute the value of $\kappa_B,$ consider a point $p=(\Sigma,u,\vec{z},\vec{w})\in B$ that is a regular value of~$\pi$. In a neighborhood of such $p$, we can calculate $\pi_*\zeta$ as the push-forward of a differential form. Indeed, using identification~\eqref{eq:fiberid} together with properties~\eqref{prop:pushfiberprod} and~\eqref{normalization} of integration, we obtain
\begin{equation}\label{eq:kappa}
\kappa_B=(\pi_*\zeta)_{p}=\int_{\pi^{-1}(p)}\zeta = \int_{\pi^{-1}(p)} evi_1^*\gamma.
\end{equation}
To continue,
denote by $v:\widetilde{\Sigma}\to\Sigma$ the oriented real blowup of $\Sigma$ at $z_0,\ldots,z_k.$
\begin{figure}[ht]
\centering
\includegraphics[width=14cm]{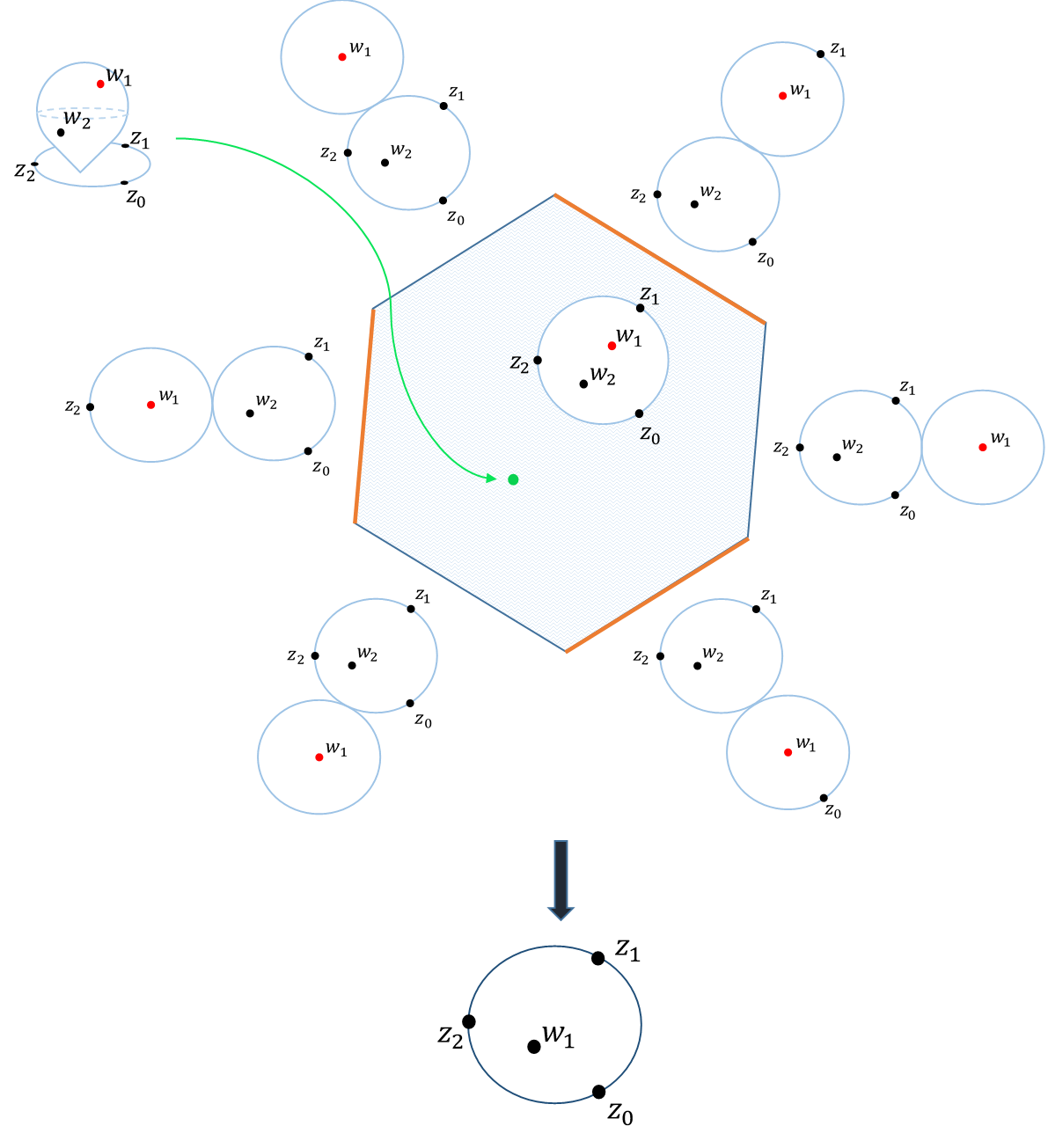}
\caption{The fiber of $\pi$ over
$(\Sigma,u,\vec{z},\vec{w})$ is the oriented real blowup of the domain at the boundary marked points. The exceptional locus of the blowup is shown in orange.}
\label{fig:fib}
\end{figure}
As explained in the proof of~\cite[Lemma 4.5]{PandharipandeSolomonTessler}, there exists a canonical orientation preserving isomorphism
\[
\psi:\widetilde{\Sigma}\stackrel{\sim}{\lrarr}\pi^{-1}(p).
\]
See Figure~\ref{fig:fib}.
Moreover, $evi_1 \circ\psi= u\circ v$.
Since $u_*[\Sigma]=(u\circ v)_*[\widetilde{\Sigma}]\in H_2(X,L;\Z)$, using equation~\eqref{eq:kappa}, we obtain
\[
\kappa_B = \int_{\pi^{-1}(p)} evi_1^*\gamma = \int_{\widetilde{\Sigma}}v^* u^*\gamma=\int_\beta\gamma.
\]
Since the last value does not depend on the component $B$, this proves the desired result.
\end{proof}

\subsection{Top degree}\label{ssec:no_top_deg}

Given $\alpha$, a homogeneous differential form with coefficients in $R$, denote by $\deg^d(\alpha)$ the degree of the differential form, ignoring the grading of $R$.
That is, for $\alpha =T^\beta t_1^{r_1}\cdots t_N^{r_N}\alpha'$ with $\alpha'\in A^j(L),$ we have $\deg^d(\alpha)=j.$

Denote by $(\alpha)_j$ the part of $\alpha$ that has degree $j$ as a differential form, ignoring the grading of $R$.
In particular, $\deg^d((\alpha)_j)=j$.

\begin{cl}\label{no_top_deg}
Suppose $(k,l,\beta)\not\in\{(1,0,\beta_0),(0,1,\beta_0),(2,0,\beta_0)\}$.
Then $(\qkl^\beta(\alpha;\gamma))_n=0$ for all lists $\alpha,\gamma$.
\end{cl}
\begin{proof}
Assume without loss of generality that $\qkl^\beta(\alpha;\gamma)$ is homogeneous with respect to the grading $\deg^d.$ Let $evb_j^{k+1},evi_j^{k+1},$ be the evaluation maps for $\M_{k+1,l}(\beta).$ Set
\[
\xi:=\bigwedge_{j=1}^l(evi_j^{k+1})^*\gamma_j\wedge\bigwedge_{j=1}^k(evb_j^{k+1})^*\alpha_j,
\]
that is, $\qkl^\beta(\alpha;\gamma)=(-1)^{\varepsilon(\alpha)}(evb_0^{k+1})_*\xi$. If
\[
\deg^d(\qkl^\beta(\alpha;\gamma))=n,
\]
then
\[
n=\deg^d(\xi)-\rdim(evb_0)
=\deg^d(\xi)-(\dim\M_{k+1,l}(\beta)-n)=\deg^d(\xi)-\dim\M_{k+1,l}(\beta)+n,
\]
so $\deg^d(\xi)=\dim\M_{k+1,l}(\beta)$.

On the other hand, if $\pi:\M_{k+1,l}(\beta)\to\M_{k,l}(\beta)$ is the map that forgets $z_0$, and $evb_j^k,evi_j^k,$ are the evaluation maps for $\M_{k,l}(\beta),$ then $\xi=\pi^*\xi'$ where
\[
\xi'=\bigwedge_{j=1}^l(evi^k_j)^*\gamma_j\wedge \bigwedge_{j=1}^k(evb^k_{j-1})^*\alpha_j \in A^*(\M_{k,l}(\beta)).
\]
In particular
\[
\deg^d(\xi')=\deg^d(\xi)=\dim\M_{k+1,l}(\beta)>\dim\M_{k,l}(\beta).
\]
Therefore, $\xi'=0$ and so $\xi=0$.
\end{proof}

\subsection{Chain map}
Write
\[
T(D):=\bigoplus_{l\ge 0}D^{\otimes l}.
\]
This forms a complex with the inherited differential
defined,
for $\eta_l=\bigotimes_{j=1}^l\eta^j_l \in D^{\otimes l}$,
by
\[
d(\bigoplus_{l\ge 0}\eta_l)
:=\bigoplus_{l\ge 0}\left(\sum_{i=1}^l(-1)^{\sum_{j=1}^{i-1}|\eta^j_l|}
\bigg(\bigotimes_{j=1}^{i-1}\eta^j_l\otimes d\eta^i_l
\otimes\bigotimes_{j=i+1}^l\eta^j_l\bigg)\right).
\]
The operators $\q_{\emptyset,l}$ extend naturally to a map
\[
\q_{\emptyset}: T(D) \to A^*(X;Q)
\]
given by
\[
\q_{\emptyset}(\bigoplus_{l\ge 0}\eta_l)
:=\sum_{l\ge 0}\q_{\emptyset,l}(\eta_l).
\]
\begin{cl}\label{cl:chain}
The operator $\q_{\emptyset}$ is a chain map on $T(D).$ That is,
\[
\q_{\emptyset}(d\eta)=d\q_{\emptyset}(\eta),\quad
\forall\eta\in T(D).
\]
\end{cl}
\begin{proof}
Since the fiber of $ev_0:\M_{l}(\beta)\to X$ has no boundary, Stokes' theorem implies that $(ev_0)_*$ commutes with $d$.
\end{proof}

\subsection{Proofs of Theorems~\ref{thm:str} and~\ref{thm:prop}}\label{ssec:proofs}

\begin{proof}[Proof of Theorem~\ref{thm:str}]
The degree of $\mg_k$ is given by Proposition~\ref{deg_str_map}. Properties~\eqref{it:lin}-\eqref{it:plin} follow from Lemma~\ref{lm:qlinear}.
Property~\eqref{it:a_infty} follows from Proposition~\ref{cl:a_infty_m}.
Properties~\eqref{it:val} and~\eqref{it:val2} are immediate from the definitions.
Properties~\eqref{it:symm} and~\eqref{it:cyclic} follow from equation~\eqref{eq:psgn} and Proposition~\ref{cl:cyclic} respectively.
Properties~\eqref{it:unit1} and~\eqref{it:unit3} follow from Proposition~\ref{cl:unit}.
Property~\eqref{it:unit2} follows from Proposition~\ref{no_top_deg}, Proposition~\ref{q_zero}, and because by assumption $\langle\gamma|_L,1\rangle=\int_L\gamma|_L=0$.

\end{proof}

\begin{proof}[Proof of Theorem~\ref{thm:prop}]
Properties~\eqref{prop1},~\eqref{prop2}, and~\eqref{prop3}, follow from Propositions~\ref{q_fund}, \ref{cl:q_div}, and~\ref{q_zero}, respectively.

\end{proof}

\section{Pseudo-isotopies}\label{pseudoisot}
\subsection{Structure}
Recall from Section~\ref{eq:mC} that $\mC:=A^*(I\times L;R)$.
We construct a family of $A_\infty$ structures on $\mC$. Fix a family of $\omega$-tame almost complex structures $\{J_t\}_{t\in I}$.
For each $\beta, k, l,$ set
\[
\Mt_{k+1,l}(\beta):=\{(t,\uu)\,|\,\uu\in\M_{k+1,l}(\beta;J_t)\}.
\]
The moduli space $\Mt_{k+1,l}(\beta)$ comes with evaluation maps
\begin{gather*}
\evbt_j:\Mt_{k+1,l}(\beta)\lrarr I\times L, \quad j\in\{0,\ldots,k\},\\
\evbt_j(t,(\Sigma,u,\vec{z},\vec{w})):=(t,u(z_j)),
\end{gather*}
and
\begin{gather*}
\evit_j:\Mt_{k+1,l}(\beta)\lrarr I\times X, \quad j\in\{1,\ldots,l\},\\
\evit_j(t,(\Sigma,u,\vec{z},\vec{w})):=(t,u(w_j)).
\end{gather*}
As with the usual moduli spaces, we assume all $\Mt_{k+1,l}(\beta)$ are smooth orbifolds with corners, and $\evbt_0$ is a proper submersion.

\begin{ex}\label{ex:constJ}
In the special case when $J_t$ is a constant family, that is, $J_t=J$ for all $t\in I$, we have
\[
\Mt_{k+1,l}(\beta)=I\times\M_{k+1,l}(\beta;J).
\]
The evaluation maps in this case are $\evbt_j=\Id \times evb_j$ and $\evit_j=\Id\times evi_j.$ In particular, the smoothness assumptions for $\Mt_{k+1,l}(\beta)$ follow from the assumptions for $\M_{k+1,l}(\beta)$.

Even in this special case, we will see below that the moduli space $\Mt_{k+1,l}(\beta)$ allows one to prove that the $A_\infty$ algebra $(C,\mg_k)$ for a fixed $J$ is determined up to pseudoisotopy by the cohomology class of $\gamma.$
\end{ex}
\begin{rem}
The assumption that $\evbt_0$ is a submersion presumably imposes strong restrictions on the possible topology of $\Mt_{k+1,l}(\beta).$
In particular, this limits significantly the possible changes in $J_t$.
The main example is where $J_t=\varphi_t^*J$ for a one-parameter family of symplectomorphisms $\varphi_t$ and $J$ is as in Example~\ref{rem:assumptions} or Remark~\ref{rem:assumptionsgeneral}. See Example~\ref{ex:constJ} above for a special case. Using virtual cycle techniques should allow the extension of the theory to the general setting.
\end{rem}

Let
\[
p:I\times L\lrarr I,\qquad p_\M: \Mt_{k+1,l}(\beta)\lrarr I,
\]
denote the projections.

For all $\beta\in\sly$, $k,l\ge 0$,  $(k,l,\beta) \not\in\{ (1,0,\beta_0),(0,0,\beta_0)\}$, define
\[
\qt_{k,l}^{\beta}:\mC^{\otimes k}\otimes A^*(I\times X;Q)^{\otimes l}\lrarr \mC
\]
by
\[ \qt_{k,l}^{\beta}(\otimes_{j=1}^k\at_j;\otimes_{j=1}^l\gt_j):= (-1)^{\varepsilon(\at)}(\evbt_0)_*(\bigwedge_{j=1}^l\evit_j^*\gt_j \wedge\bigwedge_{j=1}^k\evbt_j^*\at_j)).
\]
For $l\ge 0$, $(l,\beta)\neq (1,\beta_0),(0,\beta_0)$, define
\[
\qt_{-1,l}^{\beta}:A^*(I\times X;Q)^{\otimes l}\lrarr A^*(I;Q)
\]
by
\[
\qt_{-1,l}^{\beta}(\otimes_{j=1}^l\gt_j):= (p_\M)_*\wedge_{j=1}^l\evit_j^*\gt_j.
\]
Define also
\[
\qt_{1,0}^{\beta_0}(\at)=d\at,\quad
\qt_{0,0}^{\beta_0}:=0,\quad \qt_{-1,1}^{\beta_0}:=0,\quad \qt_{-1,0}^{\beta_0}:=0.
\]
Denote by
\[
\qt_{k,l}:\mC^{\otimes k}\otimes A^*(I\times X;Q)^{\otimes l}\lrarr \mC,\hspace{3em}
\qt_{-1,l}:A^*(I\times X;Q)^{\otimes l}\lrarr \mathfrak{R},
\]
the sums over $\beta$:
\begin{gather*}
\qt_{k,l}(\otimes_{j=1}^k\at_j;\otimes_{j=1}^l\gt_j):=
\sum_{\beta\in \sly}
T^{\beta}\qt_{k,l}^{\beta}(\otimes_{j=1}^k\at_j;\otimes_{j=1}^l\gt_j),\\
\qt_{-1,l}(\otimes_{j=1}^l\gt_j):=\sum_{\beta\in \sly} T^{\beta}\qt_{-1,l}(\gt^l).
\end{gather*}
Lastly, define similar operations using spheres,
\[
\qt_{\emptyset,l}:A^*(I\times X;Q)^{\otimes l}\lrarr A^*(I\times X;R),
\]
as follows. For $\beta\in H_2(X;\Z)$ let
\[
\Mt_{l+1}(\beta):=\{(t,\uu)\;|\,\uu\in \M_{l+1}(\beta;J_t)\}.
\]
For $j=0,\ldots,l,$ let
\begin{gather*}
\evt_j^\beta:\Mt_{l+1}(\beta)\to I\times X,\\
\evt_j^\beta(t,(\Sigma,u,\vec{w})):=(t,u(w_j)),
\end{gather*}
be the evaluation maps. Assume that all the moduli spaces $\Mt_{l+1}(\beta)$ are smooth orbifolds and $\evt_0$ is a submersion. Recall that $w_{\s} \in H^2(X;\Z/2\Z)$ is the class with $w_2(TL) = i^* w_{\s}$ determined by the relative spin structure $\s$. For $l\ge 0$, $(l,\beta)\ne (1,0),(0,0)$, set
\begin{equation*}
\qt_{\emptyset,l}^\beta(\gt_1,\ldots,\gt_l):=
(-1)^{w_\s(\beta)}
(\evt_0^\beta)_*(\bigwedge_{j=1}^l(\evt_j^\beta)^*\gt_j)
\end{equation*}
and
\begin{equation*}
\qt_{\emptyset,1}^0:= 0,\qquad \qt_{\emptyset,0}^0:= 0.
\end{equation*}
Define
\[
\qt_{\emptyset,l}(\gt_1,\ldots,\gt_l):=
\sum_{\beta\in H_2(X)}T^{\pr(\beta)}\qt_{\emptyset,l}^\beta(\gt_1,\ldots,\gt_l).
\]

\begin{cl}\label{qt_rel}
For any fixed $\at=(\at_1,\ldots,\at_k)$, $\gt=(\gt_1,\ldots,\gt_l)$,
\begin{align*}
&0=\sum_{\substack{S_3[l]\\(2:3)=\{j\}}}
(-1)^{|\gamma^{(1:3)}|+1}\qt_{k,l}(\at;
\gt^{(1:3)}\otimes d\gt_j\otimes\gt^{(3:3)})+\mbox{}\\
&+\sum_{\substack{P\in S_3[k]\\I\sqcup J=[l]}}
(-1)^{\iota(\at,\gt;P,I)}
\qt_{|(1:3)|+|(3:3)|+1,|I|}(\at^{(1:3)}\otimes \qt_{|(2:3)|,|J|}(\at^{(2:3)};\gt^J)\otimes\at^{(3:3)};\gt^I).
\end{align*}
\end{cl}
\begin{proof}
The proof is similar to that of Proposition~\ref{q_rel}. The gluing sign $\delta_1$ from Proposition~\ref{rem:gluingsign} becomes $\tilde\delta_1=\delta_1+1,$ and
the
contribution of $s=\dim M$ to the sign of Proposition~\ref{stokes} becomes $\dim \Mt_{k+1,l}(\beta) = \dim \M_{k+1,l}(\beta) + 1$, so the total computation of $\iota$ results in the same value.

\end{proof}

Define a pairing
\[
\ll\,,\,\gg:\mC\otimes\mC\lrarr \mathfrak{R}
\]
by
\[
\ll\tilde{\xi},\tilde{\eta}\gg:=(-1)^{|\etat|}p_*(\tilde{\xi}\wedge\tilde{\eta}).
\]
Note that
\begin{equation}\label{eq:pseudopair}
\ll\xit,\etat\gg=(-1)^{|\etat|}p_*(\xit\wedge\etat)
=(-1)^{|\etat|+|\etat|\cdot|\xit|}p_*(\etat\wedge\xit)
=(-1)^{(|\etat|+1)(|\xit|+1)+1}\ll\etat,\xit\gg.
\end{equation}
\begin{cl}\label{cl:qt_-1}
For any fixed $\gt=(\gt_1,\ldots,\gt_l)$,
\begin{align*}
-d\qt_{-1,l}(\gt)=&\sum_{(2:3)=\{j\}}(-1)^{|\gt^{(1:3)}|+1} \qt_{-1,l}(\gt^{(1:3)}\otimes d\gt_j\otimes\gt^{(3:3)})+\\
&+\frac{1}{2}\sum_{I\sqcup J=\{1,\ldots,l\}}
(-1)^{\iota(\gt;I)}
\ll\qt_{0,|I|}(\gt^I),\qt_{0,|J|}(\gt^J)\gg
+ (-1)^{|\gt|+1}p_* i^* \qt_{\emptyset,l}(\gt).
\end{align*}
\end{cl}
\begin{proof}
The proof uses the generalization of Stokes' theorem given in Proposition~\ref{stokes} applied to
\[
f:=p_\M:\Mt_{0,l}(\beta)\lrarr I,
\qquad
\xit:=\bigwedge_{j=1}^l\evit_j^*\gt_j,
\]
in a way similar to the proof of Proposition~\ref{q_rel}.

\textit{Contribution from $d(f_*\xit)$.}
By definition,
\[
d(f_*\xit)=d\qt_{-1,l}^\beta(\gt).
\]

\textit{Contribution from $f_*(d\xit)$.}
Again, by definition,
\[
f_*(d\xit)=\sum_{(2:3)=\{j\}}(-1)^{|\gt^{(1:3)}|} \qt_{-1,l}^\beta(\gt^{(1:3)}\otimes d\gt_j\otimes \gt^{(3:3)}).
\]

\textit{Contributions from $(f|_{\d\Mt})_*\xit$ -- first type (disk bubbling).}
Let $B\subset \d\Mt_{0,l}(\beta)$ be a boundary component
of the type described in Lemma~\ref{lm:q-1diskbd}.
Note that the gluing sign corresponding to~\eqref{eq:delta} in this case is
$
\tilde{\delta}_1:=n+1.
$
Similarly to Lemma~\ref{lm:q-1diskbd}, we find that
\[
(p_{\M}|_B)_*\xit
=
(-1)^{n+1+sgn(\sigma^{\gt}_{I\sqcup J}) +|\gt^J|}
\ll\qt^{\beta_1}_{0,|I|}(\gt^I), \qt^{\beta_2}_{0,|J|}(\gt^J)\gg.
\]
The contribution to Stokes' theorem is therefore
\begin{align*}
(-1)^{s+t}(f|_B)_*\xit&= (-1)^{|\gt|+n+n+sgn(\sigma^{\gt}_{I\sqcup J})+|\gt^J|+1}\sum_{I\sqcup J=[l]} \ll\qt_{0,|I|}^{\beta_1}(\gt^I),\qt_{0,|J|}^{\beta_2}(\gt^J)\gg\\
&=
(-1)^{sgn(\sigma^{\gt}_{I\sqcup J})+|\gt^I|+1}\sum_{I\sqcup J=[l]} \ll\qt_{0,|I|}^{\beta_1}(\gt^I),\qt_{0,|J|}^{\beta_2}(\gt^J)\gg\\
&=
(-1)^{\iota(\gt;I)+1}\sum_{I\sqcup J=[l]} \ll\qt_{0,|I|}^{\beta_1}(\gt^I),\qt_{0,|J|}^{\beta_2}(\gt^J)\gg.
\end{align*}

\textit{Contributions from $(f|_{\d\Mt})_*\xit$ -- second type (sphere bubbling from a ghost disk).}
Let $B\subset \d\Mt_{0,l}(\beta)$ be a boundary component
of the type described in Proposition~\ref{rem:pintgluing}.
Note that the gluing sign in this case is $(-1)^{n+w_\s(\beta)}$.
Similarly to Lemma~\ref{lm:starofDtype}, we find that
\[
(f|_B)_*\xit=
(p_\M)_*\xit
=
(-1)^{n}
p_*i^*\qt_{\emptyset,l}^{\hat\beta}(\gt).
\]

The total contribution to Stokes' theorem is therefore
\begin{align*}
(-1)^{s+t}(f|_B)_*\xit&= (-1)^{|\gt|+n+n}p_*i^*\qt_{\emptyset,l}^{\hat\beta}(\gt)\\
&=
(-1)^{|\gt|}p_*i^*\qt_{\emptyset,l}^{\hat\beta}(\gt).
\end{align*}

\end{proof}

For each closed $\gt\in \mI_Q\mD$ with $|\gt|=2,$ define structure maps
\[
\mt^{\gt,\beta}_k,\mgt_k:\mC^{\otimes k}\lrarr \mC
\]
by
\begin{gather*}
\m^{\gt,\beta}_k(\otimes_{j=1}^k\at_j):=\sum_{l}
\frac{1}{l!}\;\qt^\beta_{k,l}(\otimes_{j=1}^k\at_j;\gt^{\otimes l}),\\
\mgt_k(\otimes_{j=1}^k\at_j):=\sum_{l}
\frac{1}{l!}\;\qt_{k,l}(\otimes_{j=1}^k\at_j;\gt^{\otimes l}),
\end{gather*}
and define
\[
\mgt_{-1}:=\sum_{l}\frac{1}{l!}\;\qt_{-1,l}(\gt^{\otimes l})\in \mR.
\]
Denote
\begin{equation}\label{eq:GWT}
\widetilde{GW}:=\sum_{l\ge 0} \frac{1}{l!}p_* i^* \qt_{\emptyset,l}(\gt^{\otimes l}).
\end{equation}

\begin{cl}\label{cl:mgt_str}
The maps $\mgt$ define an $A_\infty$ structure on $\mC$. That is,
\[
\sum_{\substack{k_1+k_2=k+1\\k_1,k_2\ge 0\\ 1\le i\le k_1}}
(-1)^{\sum_{j=1}^{i-1}(|\at_j|+1)}
\mgt_{k_1}(\at_1,\ldots,\at_{i-1},\mgt_{k_2}(\at_i,\ldots,\at_{i+k_2-1}),\at_{i+k_2},\ldots,\at_k)=0
\]
for all $\at_j\in\mC$.
\end{cl}
\begin{proof}
Since $d\gt=0$ and $|\gt|=2$, this is a special case of Proposition~\ref{qt_rel}.
\end{proof}

\subsection{Properties}
The properties formulated for the $\q$-operators can be equally well formulated for the $\qt$-operators, with similar proofs. Below we discuss them explicitly, and add a few properties that are specific to the pseudoisotopy context.

\subsubsection{Linearity}

Observe that $\mC$ is an $\mR$ module with the action
\[
f\cdot\a:=p^*f\wedge \a, \qquad f\in \mR, \quad \a\in \mC.
\]
Similarly, let $p_X:I\times X\to I$ be the projection. Then $\mR$ acts on $A^*(I\times X;Q)$ via
\[
f\cdot \gamma:=p_X^*f\wedge \gamma, \qquad f\in \mR, \quad \gamma\in A^*(I\times X;Q).
\]

\begin{cl}\label{lm:t_linear}
The operations $\qt$ are $\mR$-multilinear in the sense that for $f\in \mR,$
\begin{multline*}
\qt_{k,l}^\beta(\at_1,\ldots,\at_{i-1},f\cdot\at_i,\ldots,\at_k;\gt_1,\ldots,\gt_l)=\\
		=(-1)^{|f|\cdot\big(i+\sum_{j=1}^{i-1}|\at_j|+\sum_{j=1}^l|\gt_j|\big)}
		f\cdot\qt_{k,l}^\beta(\at_1,\ldots,\at_k;\gt_1,\ldots,\gt_l)+\delta_{1,k}\cdot df\cdot\at_1,
\end{multline*}
and for $f\in A^*(I;Q),$
\[
\qt_{k,l}^\beta(\at_1,\ldots,\at_k;\gt_1,\ldots,f\cdot\gt_i,\ldots,\gt_l)
=(-1)^{|f|\cdot\sum_{j=1}^{i-1}|\gt_j|}
		f\cdot\qt_{k,l}^\beta(\at_1,\ldots,\at_k;\gt_1,\ldots,\gt_l),
\]
and
\[
\qt_{\emptyset,l}^\beta(\gt_1,\ldots,f\cdot\gt_i,\ldots,\gt_l)
=(-1)^{|f|\cdot\sum_{j=1}^{i-1}|\gt_j|}f\cdot\qt^\beta_{\emptyset,l}(\gt_1,\ldots,\gt_l).
\]
In addition, the pairing $\ll\;,\,\gg$ is $\mR$-bilinear in the sense of Definition~\ref{dfn:cycunit}\eqref{it:plin}.
\end{cl}
\begin{proof}

For $\qt_{1,0}^{\beta_0}=d$ we have
\[
d(f.\at)=d(p^*f\wedge \at)=d(p^*f)\wedge\at+(-1)^{|f|}p^*f\wedge d\at
=(df).\at+(-1)^{|f|}f.d\at.
\]
For $\qt_{k,l}^\beta$ with $(k,l,\beta)\ne(1,0,\beta_0),$ we have
\begin{align*}
(\evbt_0)_*(\bigwedge_{j=1}^l&\evit_j^*\gt_j\wedge\bigwedge_{j=1}^{i-1}\evbt_j^*\at_j\wedge\evbt_i^*(p^*f\wedge\at_i)\wedge\bigwedge_{j=i+1}^k\evbt_j^*\at_j)=\\
&=(\evbt_0)_*(\bigwedge_{j=1}^l\evit_j^*\gt_j\wedge\bigwedge_{j=1}^{i-1}\evbt_j^*\at_j\wedge(p\circ\evbt_i)^*f\wedge\bigwedge_{j=i}^k\evbt_j^*\at_j)\\
&=(-1)^{|f|\cdot\big(\sum_{j=1}^{i-1}|\at_j|+\sum_{j=1}^l|\gt_j|\big)}(\evbt_0)_*((p\circ\evbt_i)^*f\wedge\bigwedge_{j=1}^l\evit_j^*\gt_j\wedge\bigwedge_{j=1}^k\evbt_j^*\at_j)\\
&=(-1)^{|f|\cdot\big(\sum_{j=1}^{i-1}|\at_j|+\sum_{j=1}^l|\gt_j|\big)}(\evbt_0)_*((p\circ\evbt_0)^*f\wedge\bigwedge_{j=1}^l\evit_j^*\gt_j\wedge\bigwedge_{j=1}^k\evbt_j^*\at_j)\\
&=(-1)^{|f|\cdot\big(\sum_{j=1}^{i-1}|\at_j|+\sum_{j=1}^l|\gt_j|\big)}(p^*f)\wedge(\evbt_0)_*(\bigwedge_{j=1}^l\evit_j^*\gt_j\wedge\evbt_1^*\at_1\wedge\bigwedge_{j=2}^k\evbt_j^*\at_j).
\end{align*}
Taking into consideration the sign $\varepsilon(\at),$ we see that
\[
\qt_{k,l}(\at_1,\ldots,f\cdot\at_i,\ldots,\at_k;\gt_1,\ldots,\gt_l)=
(-1)^{|f|\cdot\big(i+\sum_{j=1}^{i-1}|\at_j|+\sum_{j=1}^l|\gt_j|\big)}f\qt_{k,l}(\at_1,\ldots,\at_k;\gt_1,\ldots,\gt_l).
\]

The second equality for $k\ge 0$ follows from
\begin{align*}
(\evbt_0)_*(\bigwedge_{j=1}^{i-1}&\evit_j^*\gt_j\wedge \evit_i^*(p_X^*f\wedge\gt_i)\wedge\bigwedge_{j=i+1}^l\evit_j^*\gt_j\wedge\bigwedge_{j=1}^k\evbt_j^*\at_j)=\\
=&(-1)^{|f|\cdot\sum_{j=1}^{i-1}|\gt_j|}(\evbt_0)_*(\evit_i^*p_X^*f\wedge\bigwedge_{j=1}^l\evit_j^*\gt_j\wedge\bigwedge_{j=1}^k\evbt_j^*\at_j)\\
=&(-1)^{|f|\cdot\sum_{j=1}^{i-1}|\gt_j|}(\evbt_0)_*((p_X\circ \evit_j)^*f\wedge\bigwedge_{j=1}^l\evit_j^*\gt_j\wedge\bigwedge_{j=1}^k\evbt_j^*\at_j)\\
=&(-1)^{|f|\cdot\sum_{j=1}^{i-1}|\gt_j|}(\evbt_0)_*((p\circ\evbt_0)^*f \wedge\bigwedge_{j=1}^l\evit_j^*\gt_j\wedge\bigwedge_{j=1}^k\evbt_j^*\at_j)\\
=&(-1)^{|f|\cdot\sum_{j=1}^{i-1}|\gt_j|}(p^*f)\wedge(\evbt_0)_*(\bigwedge_{j=1}^l\evit_j^*\gt_j\wedge\bigwedge_{j=1}^k\evbt_j^*\at_j),
\end{align*}
while $\varepsilon$ is not affected. For $k=-1,$ note that $p_\M=p_X\circ  \evit_i$. So,
\begin{align*}
(-1)^{|f|\cdot\sum_{j=1}^{i-1}|\gt_j|}(p_\M)_*&((p_X\circ \evit_j)^*f\wedge\bigwedge_{j=1}^l\evit_j^*\gt_j\wedge \bigwedge_{j=1}^k\evbt_j^*\at_j)=\\
=&(-1)^{|f|\cdot\sum_{j=1}^{i-1}|\gt_j|}(p_\M)_*(p_\M^*f\wedge\bigwedge_{j=1}^l\evit_j^*\gt_j\wedge\bigwedge_{j=1}^k\evbt_j^*\at_j)\\
=&(-1)^{|f|\cdot\sum_{j=1}^{i-1}|\gt_j|}f\wedge(p_\M)_*(\bigwedge_{j=1}^l\evit_j^*\gt_j\wedge\bigwedge_{j=1}^k\evbt_j^*\at_j),
\end{align*}
and again the required equality follows.

For $\qt_{\emptyset,l}^\beta,$ we have
\begin{align*}
\qt_{\emptyset,l}^\beta(\gt_1,\ldots,&f\cdot\gt_i,\ldots,\gt_l)=
(-1)^{w_\s(\beta)}
(\evt_0^\beta)_*(\bigwedge_{j=1}^{i-1}(\evt_j^\beta)^*\gt_j\wedge (\evt_i^\beta)^*(p_X^*f\wedge\gt_i)\wedge\bigwedge_{j=i+1}^l (\evt_j^\beta)^*\gt_j)\\
=&(-1)^{w_\s(\beta)+|f|\cdot\sum_{j=1}^{i-1}|\gt_j|} (\evt_0^\beta)_*((p_X\circ\evt_i^\beta)^*f\wedge\bigwedge_{j=1}^{l}(\evt_j^\beta)^*\gt_j)\\
=&(-1)^{w_\s(\beta)+|f|\cdot\sum_{j=1}^{i-1}|\gt_j|} (\evt_0^\beta)_*((p_X\circ\evt_0^\beta)^*f\wedge\bigwedge_{j=1}^{l}(\evt_j^\beta)^*\gt_j)\\
=&(-1)^{w_\s(\beta)+|f|\cdot\sum_{j=1}^{i-1}|\gt_j|} (\evt_0^\beta)_*((\evt_0^\beta)^*p_X^*f\wedge\bigwedge_{j=1}^{l}(\evt_j^\beta)^*\gt_j)\\
=&(-1)^{w_\s(\beta)+|f|\cdot\sum_{j=1}^{i-1}|\gt_j|} p_X^*f\wedge(\evt_0^\beta)_*\big(\bigwedge_{j=1}^{l}(\evt_j^\beta)^*\gt_j\big)\\
=&(-1)^{w_\s(\beta)+|f|\cdot\sum_{j=1}^{i-1}|\gt_j|} f\cdot\qt_{\emptyset,l}^\beta(\gt_1,\ldots,\gt_l).
\end{align*}

For the pairing, compute
\begin{multline*}
\;\;\ll p^*f\wedge\at_1,\at_2\gg=(-1)^{|\at_2|}p_*(p^*f\wedge\at_1\wedge\at_2)
=(-1)^{|\at_2|}f\wedge p_*(\at_1\wedge\at_2)=f\wedge\ll\at_1,\at_2\gg,\\
\ll\at_1,p^*f\wedge\at_2\gg
=(-1)^{|f|+|\at_2|+|f|\cdot|\at_1|}p_*(p^*f\wedge\at_1\wedge\at_2)
=(-1)^{|f|+|\at_2|+|f|\cdot|\at_1|}f\wedge p_*(\at_1\wedge\at_2)\\
=(-1)^{|f|\cdot(1+|\at_1|)}f\wedge\ll\at_1,\at_2\gg.
\end{multline*}

\end{proof}

\subsubsection{Pseudoisotopy}
For $t\in I$ and $M=pt,L,X,$ denote by $j_t:M\hookrightarrow I\times M$ the inclusion $p\mapsto (t,p)$. Denote by $\qkl^t$ the $\q$-operators associated to the complex structure $J_t$.
\begin{cl}\label{lm:pseudo}
For $t\in I$, we have
\[
j_t^*\qt_{k,l}(\at_1,\ldots,\at_k;\gt_1,\ldots,\gt_l)=
\qkl^t(j_t^*\at_1,\ldots,j_t^*\at_k;j_t^*\gt_1,\ldots,j_t^*\gt_l).
\]
\end{cl}
\begin{proof}
Consider the pull-back diagrams
\[
\xymatrix{
{\M_{k+1,l}(\beta;J_t)}\ar[r]^{j_t}\ar[d]^{evb_i}&
{\Mt_{k+1,l}(\beta)}\ar[d]^{\evbt_i}\\
{L}\ar[r]^{j_t}&I\times L
},\qquad
\xymatrix{
{\M_{k+1,l}(\beta;J_t)}\ar[r]^{j_t}\ar[d]^{evi_i}&
{\Mt_{k+1,l}(\beta)}\ar[d]^{\evit_i}\\
{X}\ar[r]^{j_t}&I\times X
}.
\]
By property~\eqref{prop:pushfiberprod} of integration, we have
\begin{align*}
j_t^*(\evbt_0)_*(\bigwedge_{i=1}^l\evit_i^*\gt_i\wedge\bigwedge_{i=1}^k\evbt_i^*\at_i) &=
(evb_0)_*(j_t)^*(\bigwedge_{i=1}^l\evit_i^*\gt_i\wedge\bigwedge_{i=1}^k\evbt_i^*\at_i)\\
&= (evb_0)_*(\bigwedge_{i=1}^levi_i^*j_t^*\gt_i\wedge\bigwedge_{i=1}^kevb_i^*j_t^*\at_i).
\end{align*}
\end{proof}

The next result relates the cyclic structure $\ll\;,\,\gg$ on $\mC$ with $\langle\;,\,\rangle$ on $C$.

\begin{cl}\label{lm:pseudoprod}
For $t\in I,$ we have
\[
j_t^*\ll\at_1,\at_2\gg
=
\langle j_t^*\at_1,j_t^*\at_2\rangle.
\]
\end{cl}
\begin{proof}
Consider the pullback diagram
\[
\xymatrix{
{L}\ar[r]^{j_t}\ar[d]^{pt}&
{I\times L}\ar[d]^{p}\\
{\{t\}}\ar[r]^{j_t}&I
}
\]
By property~\eqref{prop:pushfiberprod} of integration, we have
\begin{align*}
j_t^*\ll\at_1,\at_2\gg
=
(-1)^{|\at_2|}
j_t^*p_*(\at_1\wedge\at_2)
=
(-1)^{|\at_2|}
pt_*(j_t^*\at_1\wedge j_t^*\at_2)
=
\langle j_t^*\at_1,j_t^*\at_2\rangle.
\end{align*}

\end{proof}

\begin{lm}\label{lm:d_ll_gg}
For any $\xit,\etat\in \mC,$
\[
(-1)^{|\xit|+|\etat|+n}\int_I d\ll\xit,\etat\gg=\langle j_1^*\xit,j_1^*\etat\rangle- \langle j_0^*\xit,j_0^*\etat\rangle.
\]
\end{lm}
\begin{proof}
By Proposition~\ref{lm:pseudoprod} and Stokes' theorem, Proposition~\ref{stokes}, we have
\begin{align*}
\langle j_1^*\xit,j_1^*\etat\rangle- \langle j_0^*\xit,j_0^*\etat\rangle&=
j_1^*\ll\xit,\etat\gg-j_0^*\ll\xit,\etat\gg\\
=&\int_{\d I}\ll\xit,\etat\gg\\
=&(-1)^{|\xit|+|\etat|+n}\int_Id\ll\xit,\etat\gg.
\end{align*}

\end{proof}

\subsubsection{Unit of the algebra}
\begin{cl}\label{cl:qt_unit}
Let $f\in A^0(I\times L)\otimes R$, $\at_1,\ldots,\at_{k}\in \mC,$ and $\gt_1,\ldots,\gt_l \in A^*(I\times X;Q).$ Then
\[
\qt_{\,k\!,l}^{\beta} (\at_1,\ldots,\at_{i-1},f,\at_{i},\ldots,\at_{k-1} ;\otimes_{r=1}^l\gt_r)=
\begin{cases}
df, & (k,\;l,\beta)=(1,0,\beta_0),\\
(-1)^{|f|}f\cdot\at_1, & (k,l,\beta)=(2,0,\beta_0),\\
&\hspace{5em} i=1,\\
(-1)^{|\at_1|(|f| + 1)}f\cdot\at_1, & (k,l,\beta)=(2,0,\beta_0),\\
&\hspace{5em} i=2,\\
0,& \text{otherwise.}
\end{cases}
\]
In particular, $1\in A^0(I\times L)$ is a strong unit for the $A_\infty$ operations $\mgt$:
\[
\mgt_{k} (\at_1,\ldots,\at_{i-1},1,\at_{i},\ldots,\at_{k-1})=
\begin{cases}
0, & k\ge 3 \mbox{ or } k=1,\\
\at_1, & k=2,\: i=1,\\
(-1)^{|\at_1|}\at_1, & k=2,\: i=2.
\end{cases}
\]
\end{cl}

\begin{proof}

Repeat the proof of Proposition~\ref{cl:unit} with $\Mt$, $\evit_j$, $\evbt_j$, and $\qt$, instead of $\M$, $evi_j$, $evb_j$, and $\q$, respectively. In the case $(k,l,\beta)=(2,0,\beta),$ the map $\evbt_0$ gives an orientation preserving identification of $\Mt_{3,0}(\beta_0)$ with $I\times L$, and the rest of the computation is again the same.

\end{proof}

\subsubsection{Cyclic structure}
\begin{cl}\label{cl:qt_cyclic}
The $\qt$ are cyclic with respect to the inner product $\ll\;,\,\gg.$ That is,
\begin{multline*}
\ll\qt_{k,l}(\at_1,\ldots,\at_k;\gt_1,\ldots\gt_l),\at_{k+1}\gg=\\
=(-1)^{(|\at_{k+1}|+1)\sum_{j=1}^{k}(|\at_j|+1)}\cdot
\ll \qt_{k,l}(\at_{k+1},\at_1,\ldots,\at_{k-1};\gt_1,\ldots,\gt_l),\at_k\gg
+\delta_{1,k}\cdot d\ll\at_1,\at_2\gg.
\end{multline*}
In particular,
\[
\ll d\at_1,\at_2\gg
=
d\ll\at_1,\at_2\gg+(-1)^{(|\at_1|+1)(|\at_2|+1)}\ll d\at_2,\at_1\gg.
\]
\end{cl}
\begin{proof}
For $(k,l,\beta)\ne (1,0,\beta_0),$
the proof of Proposition~\ref{cl:cyclic} can be repeated with $\q$, $evb_j$, and $evi_j$, replaced by $\qt$, $\evbt_j$, and $\evit_j$, respectively, since $\rdim(evb_j)=\rdim(\evbt_j)$. The appropriate relabeling automorphism is now given by
\[
\tilde\varphi(t,\Sigma,u,(z_0,\ldots,z_{k}),\vec w) = (t,\Sigma,u,(z_1,\ldots,z_k,z_0),\vec w),
\]
and its sign is still $sgn(\tilde\varphi)=k$.

For $(k,l,\beta)=(1,0,\beta_0)$, we compute
\begin{align*}
\ll d\at_1,\at_2\gg=&
(-1)^{|\at_2|}p_*(d\at_1\wedge\at_2)\\=&
p_*\big((-1)^{|\at_2|}d(\at_1\wedge\at_2)-(-1)^{|\at_1|+|\at_2|}\at_1\wedge d\at_2\big)\\
=&(-1)^{|\at_2|}d(p_*(\at_1\wedge\at_2))+(-1)^{|\at_1|+|\at_2|+1+|\at_1|(|\at_2|+1)}p_*(d\at_2\wedge\at_1)\\
=&d\ll\at_1,\at_2\gg+(-1)^{(|\at_1|+1)(|\at_2|+1)}\ll d\at_2,\at_1\gg.
\end{align*}

\end{proof}

\subsubsection{Degree of structure maps}
\begin{cl}\label{qt_deg_str_map}
For $k\ge 0$ and $\gt_1,\ldots,\gt_l\in A^*(I \times X;Q)$ with $|\gt_j| = 2,$ the map
\[
\qt_{k,l}(\; ;\gt_1,\ldots,\gt_l):\mC^{\otimes k}\lrarr \mC
\]
is of degree $2-k$.
\end{cl}

\begin{proof}
Note that $\rdim(evb_0)=\rdim(\evbt_0)$. Therefore, the proof of Proposition~\ref{deg_str_map} is valid verbatim in our case, with $\q$ replaced by $\qt$ and $evb_0$ by $\evbt_0$.

\end{proof}

\subsubsection{Symmetry}

\begin{cl}
Let $k\ge -1$. For any permutation $\sigma\in S_l,$
\[
\qt_{k,l}(\at_1,\ldots,\at_k;\gt_1,\ldots,\gt_l)=
(-1)^{s_\sigma(\gamma)}\qt_{k,l}(\at_1,\ldots,\at_k; \gt_{\sigma(1)},\ldots,\gt_{\sigma(l)}),
\]
where
$s_\sigma(\gamma)$ is as in~\eqref{eq:sgnsigmagamma}.
\end{cl}
\begin{proof}
The proof of Proposition~\ref{cl:symmetry} is valid verbatim, with $\qt$, $\evbt_j$, and $\evit_j$, instead of $\q$, $evb_j$, and $evi_j$, respectively.

\end{proof}

\subsubsection{Fundamental class}

\begin{cl}
For $k\ge 0,$
\[
\qt_{k,l}^\beta(\at_1,\ldots,\at_k;1,\gt_1,\ldots,\gt_{l-1})=
\begin{cases}
-1, & (k,l,\beta)=(0,1,\beta_0),\\
0, & \text{otherwise}.
\end{cases}
\]
Furthermore,
\[
\qt_{-1,l}^\beta(1,\gt_1,\ldots,\gt_{l-1})=0.
\]
\end{cl}
\begin{proof}
Since $\rdim \evbt_0=\rdim evb_0$, we can repeat the proof of
Proposition~\ref{q_fund}
with $\Mt$, $\evbt_j$, $\evit_j$, and $\qt$, instead of $\M$, $evb_j$, $evi_j$, and $\q$, respectively.
In the case $(k,l,\beta)=(0,1,\beta_0)$ the map
$\evbt_0$ now identifies the moduli space with $I\times L$.

\end{proof}

\subsubsection{Energy zero}
\begin{cl}\label{cl:qt_zero}
For $k\ge 0,$
\[
\qt_{k,l}^{\beta_0}(\at_1,\ldots,\at_k;\gt_1,\ldots,\gt_l)=
\begin{cases}
d\at_1, & (k,l)=(1,0),\\
(-1)^{|\at_1|}\at_1\wedge\at_2, & (k,l)=(2,0),\\
-\gt_1|_{I\times L}, & (k,l)=(0,1),\\
0, & \text{otherwise}.
\end{cases}
\]
Furthermore,
\[
\qt_{-1,l}^{\beta_0}(\gt_1,\ldots,\gt_l)=0.
\]
\end{cl}
\begin{proof}
Note that $\rdim(evi_j)=\rdim(\evit_j)$ and $\rdim(evb_j)=\rdim(\evbt_j)$ for any $j$. Therefore the proof of Proposition~\ref{q_zero} is valid verbatim in our case, with $\q$ replaced by $\qt$ everywhere.

\end{proof}

\subsubsection{Divisors}

Note that $H_2(X,L;\Z)\simeq H_2(I\times X,I\times L;\Z)$. Therefore, the integral $\int_\beta\gt$ is defined for $\gt\in A^2(I\times X,I\times L)$ and $\beta\in H_2(X,L;\Z)$.

\begin{cl}
Assume $\gt_1\in A^2(I\times X,I\times L)\otimes Q$, $d\gt_1 = 0$,
and the map
$
H_2(X,L;\Z)\to Q
$
given by $\beta\mapsto\int_\beta\gt_1$ descends to $\sly$.
Then
	\begin{equation}\label{eq:qtdiv} \qt_{k,l}^{\beta}(\otimes_{j=1}^k\at_j;\otimes_{j=1}^{l}\gt_j)=
	\left(\int_\beta\gt_1\right) \cdot\qt_{k,l-1}^{\beta} (\otimes_{j=1}^k\at_j;\otimes_{j=2}^{l}\gt_j)
	\end{equation}
for $k\ge -1$.
\end{cl}
\begin{proof}
The proof or Proposition~\ref{cl:q_div} holds verbatim with $\Mt$, $\evit_j$, $\evbt_j$, and $\qt$, instead of $\M$, $evi_j$, $evb_j$, and $\q$, respectively.

\end{proof}

\subsubsection{Top degree}
\leavevmode
In this section, we use the notation introduced in Section~\ref{ssec:no_top_deg}.

\begin{cl}\label{cl:qt_no_top_deg}
Assume $(k,l,\beta)\not\in\{(1,0,\beta_0),(0,1,\beta_0),(2,0,\beta_0)\}$.
Then $(\qt_{k,l}^\beta(\at;\gt))_{n+1}=0$ for all lists $\at,\gt$.
\end{cl}

\begin{proof}
Follow the proof of Proposition~\ref{no_top_deg} with $\q$ replaced by $\qt$ and $evb_0$ by $\evbt_0$. In this case, $\rdim \evbt_0=\dim\Mt_{k+1,l}(\beta)-n-1$, so the assumption $\deg^d (\qt_{k,l}^\beta(\at;\gt))=n+1$ is what implies $\deg^d(\xi)=\dim\Mt_{k+1,l}(\beta)$. The rest of the proof is then valid.

\end{proof}

\begin{cl}\label{lm:mt0}
For all lists $\gt=(\gt_1,\ldots,\gt_l)$, we have
\[
\ll\qt_{0,l}(\gt),1\gg=
\begin{cases}
0, & l\ne 1,\\
-p_*(\gt_1|_{I\times L}), & l=1.
\end{cases}
\]
\end{cl}
\begin{proof}
By Proposition~\ref{cl:qt_no_top_deg}, the only possible contribution to $(\ll\qt_{0,l}(\gt),1\gg)_1$ is from $\qt_{0,1}^{\beta_0}$, but
$\qt_{0,1}^{\beta_0}(\gt_1)=-\gt_1|_{I\times L}$ and $\ll\gt_1|_{I\times L},1\gg=p_*(\gt_1|_{I\times L})$.
It remains to compute $(\ll\qt_{0,l}(\gt),1\gg)_0.$ To do this, we evaluate at an arbitrary point $t\in I$.
For clarity, denote by $j_t^{pt}:pt\to I,$ $j_t^L:L\to I\times L$, and $j_t^X:X\to I\times X,$ the inclusions.
Consider the pull-back diagram
\[
\xymatrix{
L\ar[r]^(.4){j^L_t}\ar[d]& I\times L\ar[d]^p\\
pt\ar[r]^{j_t^{pt}}&I.
}
\]
By property~\eqref{prop:pushfiberprod} of integration and Proposition~\ref{lm:pseudo} we have
\begin{multline*}
(\ll\qt_{0,l}(\gt),1\gg)_0(t)= (j^{pt}_t)^*(p_*\qt_{0,l}(\gt))_0
= (j^{pt}_t)^*p_*(\qt_{0,l}(\gt))_n=\\
=\int_L (j^L_t)^*(\qt_{0,l}(\gt))_n
=\int_L (\q_{0,l}^t((j^X_t)^*\gt)).
\end{multline*}
By Proposition~\ref{no_top_deg}, this can only be nonzero when $l=1$, and then
\begin{multline*}
(\ll\qt_{0,l}(\gt),1\gg)_0(t)=\int_L (\q_{0,l}^t((j^X_t)^*\gt))
=\int_L\q_{0,1}^{t,\beta_0}((j^X_t)^*\gt)\\
= -\int_Li^*(j^X_t)^*\gt
= -(pt)_*(j_t^L)^*(\Id\times i)^*\gt_1
= - (j_t^{pt})^*p_*(\Id\times i)^*\gt_1.
\end{multline*}

\end{proof}

\subsubsection{Chain map}
As in Section
consider the complex
\[
T(\mD):=\bigoplus_{l\ge 0}\mD^{\otimes l}
\]
with the differential inherited from $\mD$.
Then the operators $\qt_{\emptyset,l}$ extend naturally to a map
\[
\qt_{\emptyset}: T(\mD) \to A^*(I\times X;Q).
\]

\begin{cl}
The operator $\qt_{\emptyset}$ is a chain map on $T(\mD).$ That is,
\[
\qt_{\emptyset}(d\eta)=d\qt_{\emptyset}(\eta),\quad
\forall\eta\in T(\mD).
\]
\end{cl}
\begin{proof}
The proof is the same as for Proposition~\ref{cl:chain}, with $\evt_0$ instead of $ev_0$.

\end{proof}

\subsubsection{Proof of Theorem~\ref{thm:isot}}
\begin{proof}[Proof of Theorem~\ref{thm:isot}]
Choose $\eta\in D$ with $|\eta|=1$ such that $\gamma'-\gamma=d\eta$. Take
\[
\gt:=\gamma+t(\gamma'-\gamma)+dt\wedge\eta\in\mD.
\]
Then $|\gt|=2$ and
\begin{gather*}
d\gt=dt\wedge(\gamma'-\gamma)-dt\wedge d\eta=0,\\
j_0^*\gt=\gamma,\qquad j_1^*\gt=\gamma'.
\end{gather*}

From Propositions~\ref{cl:mgt_str},~\ref{lm:t_linear},~\ref{lm:pseudo},~\ref{lm:pseudoprod},~\ref{cl:qt_unit},~\ref{cl:qt_cyclic},~\ref{qt_deg_str_map},~\ref{lm:mt0}, and equation~\eqref{eq:pseudopair},
it follows that $(\mC,\mgt)$ is a cyclic unital pseudoisotopy from $(C,\mg)$ to $(C,\m^{\gamma'}).$

\end{proof}

\subsubsection{Relaxed assumptions}

Define a subcomplex of $A^*(X)$ by
\[
\widehat{A}^*(X,L):=\left\{\eta\in A^*(X)\;\bigg|\,\int_Li^*\eta=0\right\}.
\]
Then Theorems~\ref{thm:str} and~\ref{thm:isot} hold for $\gamma\in\widehat{A}^*(X,L)$ by verbatim the same proof as for $\gamma\in A^*(X,L)$. Specifically, for closed $\gamma\in (\mI_Q\widehat{A}^*(X,L))_2$, we have that
$(\{\mg_k\}_{k\ge 0},\langle\,,\,\rangle,1)$ is a cyclic unital $A_\infty$ structure on $C$.
Moreover, set
\[
\widehat{A}^*(I\times X,I\times L):=\left\{\etat\in A^*(I\times X)\;\bigg|\,p_*(\Id\times i)^*\etat=0\right\}.
\]
Then given closed $\gamma,\gamma'\in (\mI_Q\widehat{A}^*(X,L))_2$ with $[\gamma]=[\gamma']\in H^*(\widehat{A}^*(X,L),d)$, there exists a cyclic unital pseudoisotopy $\mgt$ from $\mg$ to $\mgp$ with $\gt\in \widehat{A}^*(I\times X,I\times L).$

As for Theorem~\ref{thm:prop}, the fundamental class and zero properties are satisfied for $\mg$. Namely, if $\gamma\in (\mI_Q\widehat{A}^*(X,L))_2$ is closed and $\d_{t_0}\gamma=1$, then
$\d_{t_0}\mg_k=-1\cdot \delta_{0,k}$ and $\bar{\m}^{\gamma}$ is a deformation of the standard differential graded algebra structure.
However, the divisor property is not necessarily satisfied.

\subsection{Uniform formulation of structure equations}\label{ssec:uniform}
Using the cyclic structure $\ll\;,\,\gg$,
the $A_\infty$ relations can be rephrased so the case $k=-1$ fits more uniformly. Recall the definition of $\widetilde{GW}$ from~\eqref{eq:GWT}.
\begin{prop}\label{prop:mgt_a_infty}
For $k\ge 0$,
\begin{multline*}
d\ll\mgt_k(\at_1,\ldots,\at_k),\at_{k+1}\gg=\\
=\sum_{\substack{k_1+k_2=k+1\\k_1\ge 1,k_2\ge 0\\1\le i\le k_1}}
(-1)^{\nu(\at;k_1,k_2,i)}\ll\mgt_{k_1}(\at_{i+k_2},\ldots,\at_{k+1},\at_1,\ldots,\at_{i-1}), \mgt_{k_2}(\at_i,\ldots,\at_{k_2+i-1})\gg
\end{multline*}
with
\[
\nu(\at;k_1,k_2,i):=\sum_{j=1}^{i-1}(|\at_j|+1)+
\sum_{j=i+k_2}^{k+1}(|\at_j|+1)\Big(\sum_{\substack{m\ne j\\1\le m\le k+1}}(|\at_m|+1)+1\Big)+1
\]
For $k = -1,$
\[
d\mgt_{-1} = -\frac{1}{2}\ll\mgt_0,\mgt_0\gg+\widetilde{GW}.
\]
\end{prop}

\begin{proof}
For $k\ge 0,$ we use Propositions~\ref{cl:qt_cyclic} and~\ref{cl:mgt_str} to obtain
\begin{align*}
d\ll&\mgt_k(\at_1,\ldots,\at_k),\at_{k+1}\gg=\\
=&\ll d\mgt_k(\at_1,\ldots,\at_k),\at_{k+1}\gg-
(-1)^{(|\at_{k+1}|+1)(|\mgt_k(\at_1,\ldots,\at_k)|+1)}\ll d\at_{k+1},\mgt_k(\at_1,\ldots,\at_k)\gg\\
=&-\hspace{-1.5em}\sum_{\substack{k_1+k_2=k+1\\(k_1,\beta)\ne(1,\beta_0)\\1\le i \le k_1}}\hspace{-1.5em} (-1)^{\sum_{j=1}^{i-1}(|\at_j|+1)}\ll T^\beta\mt^{\gt,\beta}_{k_1}(\at_1,\ldots,\at_{i-1},\mgt_{k_2}(\at_i,\ldots,\at_{i+k_2-1}),\at_{i+k_2},\ldots,\at_k),\at_{k+1}\gg+\\
&+(-1)^{(|\at_{k+1}|+1)(|\mgt_k(\at_1,\ldots,\at_k)|+1)+1}\ll d\at_{k+1},\mgt_k(\at_1,\ldots,\at_k)\gg\\
=&\hspace{-0.5em}\sum_{\substack{k_1+k_2=k+1\\(k_1,\beta)\ne(1,\beta_0)\\1\le i \le k_1}}\hspace{-1.5em} (-1)^{1+\sum_{j=1}^{i-1}(|\at_j|+1)+\nu'}\ll T^\beta\mt^{\gt,\beta}_{k_1}(\at_{i+k_2},\ldots,\at_k,\at_{k+1}\at_1,\ldots,\at_{i-1}),\mgt_{k_2}(\at_i,\ldots,\at_{i+k_2-1})\gg+\\
&+(-1)^{(|\at_{k+1}|+1)(\sum_{j=1}^k(|\at_j|+1)+1)+1}\ll d\at_{k+1},\mgt_k(\at_1,\ldots,\at_k)\gg,
\end{align*}
with the sign $\nu'$ as follows:
\begin{align*}
\nu'=&
\sum_{j=i+k_2}^{k+1}(|\at_j|+1)\Big(\hspace{-0.5em}\sum_{\substack{m\ne j\\1\le m\le k+1\\m\not\in \{i,\ldots,k_2+i-1\}}}\hspace{-1em}(|\at_m|+1)+(|\mgt(\at_i,\ldots,\at_{i+k_2-1})|+1)\Big)\\
\equiv&\sum_{j=i+k_2}^{k+1}(|\at_j|+1)\Big(\sum_{\substack{m\ne j\\1\le m\le k+1}}(|\at_m|+1)+1\Big)\pmod 2.
\end{align*}

For $k=-1,$ note that $|\gt|=2$, so
$
sgn(\sigma^{\gt}_{I\cup J})\equiv 0\pmod 2.
$
This implies that Proposition~\ref{cl:qt_-1} reads
\[
-d\mgt_{-1}=\frac{1}{2}\ll\mgt_0,\mgt_0\gg
- \widetilde{GW}.
\]
\end{proof}

\bibliography{../../bibliography_exp}
\bibliographystyle{../../amsabbrvcnobysame}
\end{document}